\documentclass[10pt,a4paper]{article}
\usepackage{amsfonts}

\usepackage{mathrsfs}

\usepackage{stmaryrd}
\usepackage{amsmath}
\usepackage{amssymb}
\usepackage{xcolor}
\usepackage{bbm}
\usepackage{mathrsfs}
\usepackage{graphicx}
\usepackage{lipsum}
\usepackage{appendix}

\usepackage{enumerate}
\usepackage{theorem}
\usepackage{indentfirst}
\usepackage{hyperref}
\hypersetup{
	colorlinks=true,    
	linkcolor=blue,
	citecolor=red,
}
\makeatletter
\def\tank#1{\protected@xdef\@thanks{\@thanks
		\protect\footnotetext[0]{#1}}}
\def\bigfoot{
	
	\@footnotetext}
\makeatother

\renewcommand{\theequation}{\arabic{section}.\arabic{equation}}
\newcommand{\ea}{\end{array}}

\allowdisplaybreaks[4]

\newtheorem{theorem}{Theorem}[section]

\newtheorem{lemma}{Lemma}[section]
\newtheorem{definition}{Definition}[section]
\newtheorem{remark}{Remark}[section]
\numberwithin{equation}{section}

{\theorembodyfont{\rmfamily}
}

\newenvironment{proof}[1][\normalfont\bfseries Proof]{
	\noindent #1.\hspace{0.5em}
	\ignorespaces
}{%
	\hfill $\square$
	\par
}

\renewcommand{\theequation}{\arabic{section}.\arabic{equation}}

\begin{document}
\title{{\Large \bf Averaging principle for slow--fast systems of PDEs with rough drivers }}
\author{{Miaomiao Li$^{a}$},~~{Bin Pei$^{a,b,c}$}\footnote{Corresponding author: binpei@nwpu.edu.cn},~~{Yong Xu$^{a,c}$},~~{Xiaole Yue$^{a,c}$}
	\\
	\scriptsize $a.$ School of Mathematics and Statistics, Northwestern Polytechnical University, Xi'an 710072, China \\
    \scriptsize $b.$ Research and Development Institute of Northwestern Polytechnical University in Shenzhen, Shenzhen 518057, China \\
    \scriptsize $c.$ MOE Key Laboratory of Complexity Science in Aerospace, Northwestern Polytechnical University, Xi'an 710072, China }

\date{}
\maketitle
\begin{center}
	\begin{minipage}{145mm}
		
		{\bf Abstract.} This paper investigates a class of slow--fast systems of rough partial differential equations defined over a monotone family of interpolation Hilbert spaces. By employing the controlled rough path framework tailored to a monotone family of interpolation spaces, together with a time discretization argument, we demonstrate that the slow component strongly converges to the solution of the averaged system in the supremum norm as the time-scale parameter $\varepsilon$ tends to $0$.

\vspace{3mm} {\bf Keywords:} Averaging principle; rough partial differential equations; slow--fast system

	\end{minipage}
\end{center}

\tableofcontents

\section{Introduction}

Lyons' seminal rough path theory \cite{Ly98} gives meaning and well-posedness of rough differential equations (RDEs) of the form
\begin{equation}\label{troeq1}
\mathrm{d} Y_t=f(Y_t)\mathrm{d} \mathbf{X}_t,
\end{equation}
where $\mathbf X=(X, \mathbb X)$ is a two-step $\alpha$-H\"{o}lder rough path with $\alpha \in (1/3,1/2]$, and $f$ is a sufficiently regular vector field.
The central idea of the theory is to extend integration to driving signals $X \in \mathcal C^\alpha$ with low regularity $\alpha <1/2$, for which the classical Young integral is no longer well-defined.
To overcome this, one augments $X$ with additional information $\mathbb X_{s,t}$ for $t>s$ that postulates the values of the (ill-defined) iterated integrals $\int_s^t X_{s,r} \otimes \mathrm{d}X_r$ (where $X_{s,r}=X_r-X_s$) and that satisfies the bound $|\mathbb X_{s,t}| \lesssim |t-s|^{2\alpha}$ consistent with the regularity of $X$,
as well as the algebraic identity $\mathbb X_{s,t}-\mathbb X_{s,u}-\mathbb X_{u,t}=X_{s,u} \otimes X_{u,t}$.
Since its inception, rough path theory has become a powerful tool with which one can solve stochastic differential equations (SDEs) driven by irregular signals, such as SDEs by fractional Brownian motion with Hurst parameter $H<1/2$ (see e.g. \cite{CQ02}).
In particular, when $X = W$ is a Brownian motion, it can be enhanced to a rough path by incorporating its iterated integrals defined via It\^{o} integration.
Thus, rough path theory provides a novel pathwise approach for interpreting It\^{o}'s stochastic calculus.
One of the main advantages of this approach is that it allows to show that the solution depends continuously on the noise and the initial condition. This is in strong contrast to the approach using It\^{o} calculus, where in general only measurability of the solution map is available.
For a more comprehensive overview of rough path theory, we refer the reader to the monograph by Friz and Hairer \cite{FH20}. 
Further studies on RDEs can be found in \cite{CF25,DG24,FHL24,IXY24,YXP25} and the references therein.

As already conjectured in Lyons' seminal work \cite{Ly98}, rough path theory also holds great potential for the analysis of partial differential equations (PDEs) driven by irregular signals.
In recent years, various approaches have been developed to analyze rough PDEs.
One of the important approaches is the semigroup theory, which investigates mild solutions to rough PDEs.
Mild solutions of semilinear rough PDEs were first established by Gubinelli and Tindel \cite{GT10} in the case of
polynomial nonlinearities.
This was followed by significant progress on the local existence and uniqueness of mild solutions to general semilinear rough PDEs, notably through the works of Gerasimovi\v{c}s and Hairer \cite{GH19}, Hesse and Neam\c{t}u \cite{HN19}, and Gerasimovi\v{c}s et al. \cite{GHN21}.
Furthermore, the existence of global mild solutions was established by Deya et al. \cite{DGT12} and Hesse and Neam\c{t}u \cite{HN20,HN22}.
For alternative approaches to the study of rough PDEs, we refer the reader to the stochastic viscosity theory \cite{CaF09,CFO11,FGLS17} and the Feynman-Kac approach \cite{DOR15}, among others.

In this paper, we investigate a slow--fast system consisting of rough PDEs within the framework of semigroup theory.
In particular, we consider the following system:
\begin{equation}\label{troeq2}
	\left\{ \begin{aligned}
		&\mathrm{d} X_t^\varepsilon=\big[L X_t^\varepsilon+F_1 (X_t^\varepsilon,Y_t^\varepsilon)\big]\mathrm{d} t+G_1 (X_t^\varepsilon)\mathrm{d} \mathbf B_t,\\
		&\mathrm{d} Y_t^\varepsilon=\frac 1\varepsilon \big[L Y_t^\varepsilon+F_2 (X_t^\varepsilon,
		Y_t^\varepsilon)\big]\mathrm{d} t+\frac 1{\sqrt{\varepsilon}}G_2 (X_t^\varepsilon,
		Y_t^\varepsilon)  \mathrm{d} \mathbf W^{\text{It\^{o}} }_t,\\
		&X_0^\varepsilon=x, ~~ Y_0^\varepsilon=y,
	\end{aligned}  \right.
\end{equation}
where $\varepsilon >0$ is a small parameter characterizing the separation of time scales between the slow component $X_t^\varepsilon$ and the fast component $Y_t^\varepsilon$;
$\mathbf B=(B, \mathbb B)$ is a $d_1$-dimensional random $\alpha$-H\"{o}lder rough path defined on a filtered probability space $(\Omega, \mathcal{F}, \{\mathcal{F}_t\}_{t \in [0, T]},\mathbb{P})$, for some $\alpha \in (1/3,1/2)$, and the process $(B_{0,t}, \mathbb B_{0,t})_{t \in [0,T]}$ is adapted to the filtration $\{\mathcal{F}_t\}_{t \in [0, T]}$;
$\mathbf W^{\text{It\^{o}} }=(W,\mathbb W^{\text{It\^{o}} })$ is a $d_2$-dimensional It\^{o}-type Brownian rough path (see Definition \ref{Xi} below), and it is assumed that $W$ and $\mathbf B$ are independent;
the operator $L$ is negative definite and self-adjoint;
and the coefficients $F_1,F_2,G_1, G_2$ satisfy certain assumptions, which will be detailed in Section \ref{AP} below.
Slow--fast systems arise naturally in many real-world dynamical systems, such as climate dynamics, material science, biology and chemical kinetics (cf. \cite{EE03,Ki01} and more references therein).
In fact, as mentioned in \cite{SA61}, almost all physical systems have a certain hierarchy in which different components evolve on distinct time scales, that is, some components vary rapidly while others change very slowly.
However, the direct analysis or simulation of such systems is often challenging due to the presence of widely separated time scales and the cross interactions between the slow and fast components.
The averaging principle is a powerful analytical tool for studying slow--fast systems and addressing the challenges they present.
Its key idea is to construct an effective approximation of the slow component by averaging its parameters with respect to the fast variables.
This effective approximation system no longer depends on the fast variables and is therefore more suitable for analysis and simulation.

There have been fruitful results in the study of averaging principle of slow--fast systems.
The first rigorous result of averaging principle was presented by Bogoliubov and Mitropolsky \cite{BM61} in the case of ordinary differential equations.
The further development of the theory in the case of SDEs was established by Khasminskii \cite{K68}.
Since then, there has been a growing body of research devoted to the investigation of the averaging principle of various types of slow--fast systems.
For slow--fast systems driven by Brownian motion, results on the averaging principle for SDEs can be found in \cite{GKK06,GL90,GR16,LD10,LRSX20} and the references therein, see also e.g. \cite{C09,CF09,CL17,CZ25,HLL22,LRSX23,SXX21} for related advances concerning stochastic PDEs.
For slow--fast systems driven by fractional Brownian motion with Hurst parameter $H > 1/2$,
some notable works include \cite{HL20,HXPW22,LGQ25,PIX23,PIX20,SXW22}, where, due to the requirement of ergodicity, the slow component is driven by fractional Brownian motion while the fast component is driven by Brownian motion.
Notably, Li and Sieber \cite{LS22} made the first attempt to establish the averaging principle for slow--fast SDEs in which the fast component is also driven by fractional Brownian motion, employing a geometric ergodicity theorem to address the challenges arising from the non-Markovian structure.
This result was later extended to stochastic PDEs by Pei et al. \cite{PSX24}.
These developments naturally raise the question of whether the averaging principle still holds for slow--fast systems driven by fractional Brownian motion with Hurst parameter $H<1/2$, or more generally, by irregular (non-semimartingale) signals.
From a mathematical standpoint, this presents significant challenges due to the increased roughness and the lack of both the semimartingale and Markov properties.

To this end, this work aims to investigate the averaging principle for slow--fast systems driven by irregular signals.
To the best of the authors' knowledge, such systems have received limited attention in the existing literature.
Pei et al. \cite{PIX21} established the averaging principle for slow-fast RDEs via a fractional calculus approach, where the fast component is driven by Brownian motion and the slow component by fractional Brownian motion with Hurst parameter $H \in (1/3, 1/2]$.
Later, Inahama \cite{In22} extended the result to a broader setting in which the slow component is driven by a general random $\alpha$-H\"{o}lder rough path with $\alpha \in (1/3, 1/2)$, employing the theory of controlled rough paths.
Furthermore, in \cite{In25}, he treated the more singular regime $\alpha \in (1/4, 1/3)$.
In a related development, Pei et al. \cite{PHSX23} considered a slow--fast system where both the slow and fast components are driven by fractional Brownian motion with Hurst parameter $H \in (1/3, 1/2]$, and established an almost sure version of the averaging principle.
Very recently, Li et al. \cite{LLPX25} derived the averaging principle for a class of semilinear slow--fast rough PDEs.

However, we would like to emphasize that, although \cite{LLPX25} established the validity of the averaging principle for system (\ref{troeq2}), the convergence of the slow component to the solution of the averaged equation depends on the semigroup $S$ generated by the operator $L$.
This dependency arises from the formulation of controlled rough paths in \cite{LLPX25}, which employs the reduced increment operator $\hat{\delta}$, defined by $\hat{\delta} f_{s,t} = f_t - S_{t-s}f_s$.
Consequently, the underlying function spaces explicitly depend on the semigroup $S$.
More precisely, \cite{LLPX25} established the following convergence result:
\begin{equation*}
\lim_{\varepsilon\rightarrow 0}  \mathbb E\left[\sup_{0\leq s<t\leq T} \frac {|X_t^\varepsilon-\bar{X}_t-S_{t-s}(X_s^\varepsilon-\bar{X}_s)|_{\mathcal H}}{|t-s|^\eta}\right]=0,~~~~\text{for} ~\eta \in (\alpha-1/4,\alpha),
\end{equation*}
where $\bar{X}$ denotes the solution of the corresponding averaged equation (see Eq. (\ref{aveq}) below).

In the present work, our primary objective is to examine whether the averaging principle for system (\ref{troeq2}) remains valid when the dependence on the semigroup $S$ is entirely removed.
To this end, we introduce a notion of controlled paths inspired by \cite{GHN21}, defined over a monotone family of interpolation Hilbert spaces.
In contrast to the framework employed in \cite{GH19,GT10}, our formulation is independent of the semigroup $S$ and does not involve the reduced increment operator $\hat{\delta}$.
Within this semigroup-free controlled rough path framework, we demonstrate that the slow component converges to the solution of the averaged equation in the supremum norm. Specifically, under appropriate assumptions and using a time discretization argument, we rigorously derive the strong convergence result:
\begin{equation*}
\lim_{\varepsilon\rightarrow 0}  \mathbb E\left[\sup_{t\in [0,T]} |X_t^\varepsilon-\bar{X}_t|_{\mathcal H_{\gamma} }\right]=0, ~~~~\text{for}~ \gamma \in  \mathbf R.
\end{equation*}
It should be noted that, compared to \cite{LLPX25}, we allow the nonlinear coefficients in Eq. (\ref{troeq2}) to lose some spatial regularity.

The paper is organized as follows: In Section \ref{sect2}, we introduce our settings, some preliminary notations, and state the main result of this paper. Section \ref{sect3} provides a rigorous formulation of the slow--fast system of rough PDEs from both deterministic and probabilistic viewpoints.
In Section \ref{sect4}, we present the detailed proof of the main result.

\section{Preliminaries and main result}\label{sect2}

Throughout the paper, we set $\alpha_0 \in (1/3,1/2]$ and fix a finite time horizon $T > 0$.
We write $\mathbf N:=\{1,2,3,\cdots\}$ and write $\mathbf N_0:=\mathbf N \cup \{0\}$.
The set of real numbers is denoted by $\mathbf R$.
For some Banach spaces $V$ and $\bar{V}$, let $\mathcal{L} (V; \bar{V})$ represent the space of all bounded linear operators from $V$ to $\bar{V}$ endowed with the norm $|\cdot|_{\mathcal{L} (V; \bar{V})}$.
For simplicity we write $\mathcal{L}(V):= \mathcal{L} (V; V )$.
For some $k\in\mathbf{N}_{0}$ we define the space $\mathcal{C}^{k}_{b}$ as the space of bounded $\mathcal C^k$ functions with bounded derivatives of all orders up to $k$.
We use the notation $a \lesssim b$ to indicate that $a\leq C b $ for some multiplicative positive constant $C$.
In the sequel, the notation $C$ with or without subscripts will represent a positive constant, whose value may vary from one place to another,
and $C$ with subscripts will be used to emphasize that it depends on certain parameters.

\subsection{Settings}

Throughout this paper, we work on a monotone family of interpolation Hilbert spaces.
To this end, we first introduce the concept of a monotone family of interpolation Hilbert spaces.

\begin{definition}
A family of separable Hilbert spaces $(\mathcal H_\gamma, |\cdot |_{\mathcal H_\gamma})_{\gamma \in \mathbf R}$ is called a monotone family of interpolation spaces if the following conditions are satisfied:

(i) For every $\gamma_1\leq\gamma_2$,  $\mathcal{H}_{\gamma_2}$ is a continuously embedded, dense
subspace of $\mathcal{H}_{\gamma_1}$, i.e. $\mathcal{H}_{\gamma_2} \hookrightarrow  \mathcal{H}_{\gamma_1}$ densely;

(ii) The following interpolation inequality holds:
\begin{equation}\label{interp}
	|x|_{\mathcal H_{\gamma_2}}^{\gamma_3-\gamma_1} \lesssim |x|_{\mathcal H_{\gamma_1}}^{\gamma_3-\gamma_2} |x|_{\mathcal H_{\gamma_3}}^{\gamma_2-\gamma_1},
\end{equation}
for all $\gamma_1 \leq \gamma_2 \leq \gamma_3$ and all $x \in \mathcal H_{\gamma_3}$.
In particular, if $\gamma=0$ we simply write $\mathcal H=\mathcal H_0$,
the inner product and the norm on $\mathcal H$ are denoted by $\langle \cdot, \cdot \rangle_{\mathcal H}$ and $| \cdot |_{\mathcal H}$, respectively.
\end{definition}

Let $L $ be a negative definite self-adjoint operator acting on a monotone family of interpolation Hilbert spaces $(\mathcal{H}_{\gamma})_{\gamma \in \mathbf R}$ and assume that $(S_{t})_{t \geq 0}$ is the semigroup generated by $L$.
The main advantage of considering a monotone family as above is that the semigroup $(S_{t})_{t \geq 0}$ can be viewed as a linear mapping between these interpolation spaces,
and the following property is obtained (see (2.2) in \cite{GHN21}).
If $S: [0,T] \to \mathcal{L}(\mathcal{H}_{\gamma})\cap \mathcal{L}(\mathcal{H}_{\gamma+1})$ is such that for each $x\in\mathcal{H}_{\gamma+1}$ and any $t \in [0,T]$, $|(S_{t }-\operatorname{id})x|_{\mathcal H_\gamma} \lesssim t |x|_{\mathcal H_{\gamma+1}}$
while $|S_{t}x|_{\mathcal H_{\gamma+1}} \lesssim t^{-1}|x|_{\mathcal H_\gamma}$,
then for every $\sigma\in[0,1]$ we have that $S_{t} \in \mathcal{L}(\mathcal{H}_{\gamma+\sigma})$ and the following estimates hold true:
\begin{align}
	|S_{t} x |_{\mathcal H_{\gamma+\sigma}} \lesssim & t^{-\sigma} |x|_{\mathcal H_\gamma}, \label{group1} \\
	|(S_{t}-\operatorname{id} ) x |_{\mathcal H_\gamma} \lesssim & t^\sigma |x|_{\mathcal H_{\gamma+\sigma}}, \label{group2}
\end{align}
where $\operatorname{id}$ is the identity operator.

\subsection{Rough path theory}

For any $\gamma \in \mathbf{R}$, given a function $f:[0,T] \to \mathcal{H}_\gamma$, we set $f_{s,t}=f_t-f_s$.
We denote by $\mathcal C([0,T]; \mathcal H_\gamma)$ the space of continuous functions from $[0,T]$ into $\mathcal{H}_\gamma$, endowed with the supremum norm $$\|f\|_{\infty,\mathcal H_\gamma}= \sup_{t\in [0,T]} |f_t|_{\mathcal H_\gamma}.$$
For $\alpha \in (0, 1]$, we define the space of $\alpha$-H\"{o}lder continuous functions by
\begin{equation}\label{onepath}
\mathcal{C}^\alpha ([0,T]; \mathcal H_\gamma) =\left\{f:[0,T] \to \mathcal{H}_\gamma : |f|_{\alpha,\mathcal H_\gamma} := \sup_{0\leq s<t\leq T} \frac{|f_{s,t}|_{\mathcal H_\gamma}}{|t-s|^{\alpha}}<\infty \right\}.
\end{equation}
We write $\Delta_{[0,T]} = \{(s, t) \in [0,T] \times [0,T]: s \leq t\}$.
For a two-parameter function $F:\Delta_{[0,T]} \to \mathcal H_\gamma$, we similarly write
\begin{equation}\label{twopara}
|F|_{\alpha,\mathcal H_\gamma}:=\sup_{0\leq s<t\leq T} \frac {|F_{s,t}|_{\mathcal H_\gamma}} {|t-s|^\alpha},
\end{equation}
and we will denote the space of $\alpha$-H\"{o}lder continuous functions on $\Delta_{[0,T]}$ by
$$\mathcal{C}^\alpha(\Delta_{[0,T]}; \mathcal H_\gamma)=\left\{ F:\Delta_{[0,T]} \to \mathcal H_\gamma : |F|_{\alpha,\mathcal H_\gamma} <\infty \right\}.$$

\begin{remark}
To avoid confusion, we stress that if $f$ is a path then $f_{s,t}$ means the increment $f_t-f_s$, but if $F$ is a two-parameter function defined on $\Delta_{[0,T]}$ then $F_{s,t}$ just means $F$ evaluated at the pair of times $(s,t) \in \Delta_{[0,T]}$.
\end{remark}

Next,  we recall the definition of a two-step rough path.

\begin{definition}\label{RP}
	(\textit{Rough path}).
	Let $\alpha \in (1/3, 1/2]$.
	We define the space of rough paths $\mathscr{C}^{\alpha}([0,T];\mathbf{R}^d)$ to consist of the pairs $(X,\mathbb{X}) =: \mathbf{X}$
	such that $X \in \mathcal{C}^\alpha ([0,T];\mathbf{R}^d)$, $\mathbb{X} \in \mathcal{C}^{2\alpha}(\Delta_{[0,T]};\mathbf{R}^{d \times d})$ satisfy the Chen's relation:	
\begin{equation*}
 \mathbb{X}_{s,t}^{i,j} - \mathbb{X}_{u,t}^{i,j} - \mathbb{X}_{s,u}^{i,j} = X_{u,t}^{i}  X_{s,u}^{j},
\end{equation*}
	for every $0\leq s \leq u \leq t \leq T$ and $1 \leq i,j \leq d $.
	The rough paths space $\mathscr{C}^{\alpha}$ is equipped with the pseudometric
\begin{equation*}
	\varrho_{\alpha}(\mathbf{X},\bar{\mathbf{X}}) = |X - \bar{X}|_{\alpha} + |\mathbb{X} - \bar{\mathbb{X}}|_{2\alpha},	
\end{equation*}
where the quantities $|X|_{\alpha}$ and $|\mathbb{X}|_{2\alpha}$ are the H\"{o}lder seminorms defined in $(\ref{onepath})$ and $(\ref{twopara})$ respectively.
For simplicity, we shall write $\varrho_{\alpha}(\mathbf{X}) := \varrho_{\alpha}(0,\mathbf{X})$ in the sequel.
\end{definition}

\begin{definition}
Let $\mathcal{H}_\gamma^d :=\mathcal L(\mathbf{R}^d; \mathcal{H}_{\gamma})$ be the space of continuous linear operators from $\mathbf{R}^d$ to $\mathcal{H}_{\gamma}$, endowed with the operator norm $|\cdot|_{\mathcal{H}_{\gamma}^{d}}:=|\cdot|_{\mathcal L(\mathbf{R}^d; \mathcal{H}_{\gamma})}$.
Similarly, we write $\mathcal{H}_\gamma^{d \times d}:=\mathcal L(\mathbf{R}^{d \times d}; \mathcal{H}_{\gamma})$.
\end{definition}

\begin{definition}\label{CRP}
(\textit{Controlled path according to a monotone family}).
Assume that $(\mathcal H_\gamma)_{\gamma \in \mathbf R}$ be a monotone family of interpolation Hilbert spaces.
Let $\mathbf X=(X,\mathbb{X}) \in \mathscr{C}^\alpha ([0,T]; \mathbf{R}^d)$ for some $\alpha \in (1/3,1/2]$.
The pair $(Y, Y^{\prime})$ is called controlled rough path on the family $(\mathcal H_\gamma)_{\gamma \in \mathbf R}$ if

(i) $Y\in \mathcal{C} ([0,T];\mathcal{H}_\gamma), Y^{\prime} \in \mathcal{C}([0,T]; \mathcal{H}_{\gamma-\alpha}^d) \cap \mathcal{C}^\alpha ([0,T]; \mathcal{H}_{\gamma-2\alpha}^d) $;

(ii) The remainder
	\begin{equation}\label{remainder}
		R_{s,t}^Y = Y_{s,t}-Y_s^{\prime}  X_{s,t}
	\end{equation}
belongs to $\mathcal{C}^\alpha (\Delta_{[0,T]};\mathcal{H}_{\gamma-\alpha}) \cap \mathcal{C}^{2\alpha} (\Delta_{[0,T]};\mathcal{H}_{\gamma-2\alpha})$,
where $Y^{\prime}$ is the Gubinelli derivative of $Y$.
\end{definition}

The space of all such controlled rough paths will be denoted by $\mathcal{D}_{X}^{2\alpha}([0,T];\mathcal{H}_{\gamma})$ or simply written $\mathcal{D}_{X,\gamma}^{2\alpha}$.
We endow the space $\mathcal{D}_{X,\gamma}^{2\alpha}$ with the following norm (and sometimes we will use $\|\cdot, \cdot\|_{\mathcal{D}_{X,\gamma}^{2\alpha}([0,T])}$ to emphasize the time horizon):
\begin{equation}\label{norm}
	\|Y,Y^{\prime}\|_{ \mathcal{D}_{X,\gamma}^{2\alpha} }=
	\|Y\|_{\infty, \mathcal H_\gamma}
	+ \|Y^{\prime}\|_{\infty, \mathcal H_{\gamma-\alpha}^d}
	+|Y^{\prime}|_{\alpha, \mathcal H_{\gamma-2\alpha}^d}
	+|R^Y|_{\alpha, \mathcal H_{\gamma-\alpha}}
	+|R^Y|_{2\alpha,\mathcal H_{\gamma-2\alpha}} .
\end{equation}

\begin{remark}
(i) We do not define H\"{o}lder seminorm of $Y$ as a part of the definition of a controlled rough path,
since by using (\ref{remainder}) one can immediately obtain the following:
\begin{equation}\label{Y}
|Y|_{ \alpha , \mathcal H_{\gamma-\theta} } \leq \| Y^{\prime} \|_{\infty, \mathcal H_{\gamma-\theta}^d } |X|_{\alpha} + |R^Y|_{\alpha, \mathcal H_{\gamma-\theta} }, ~~~~ \theta \in \{\alpha,2\alpha\}.
\end{equation}

(ii) When the time interval $[0, T]$ is replaced by a subinterval $[s, t]\subset [0,T]$, all relevant (semi)norms and notations will be written with explicit dependence on the time interval. For instance, we write $\|Y\|_{\infty, \mathcal H_\gamma,[s,t]}$, $|R^Y|_{\alpha,\mathcal H_{\gamma-\alpha},[s,t]}$, and so on.
\end{remark}

We next introduce the notion of a rough integral.

\begin{lemma}\label{RoughInt}
(\cite[Theorem 4.5]{GHN21}).
Let $\gamma \in \mathbf{R}$ and let $\mathbf{X}=(X,\mathbb{X}) \in \mathscr{C}^{\alpha}([0,T];\mathbf{R}^{d})$ for some $\alpha \in (1/3, 1/2]$.
Let $(Y,Y^{\prime}) \in \mathcal{D}_{X}^{2\alpha}([0,T];\mathcal{H}_{\gamma}^d)$.
Then the rough integral defined by
\begin{equation}\label{int}
\int_{0}^{t} S_{t-s} Y_{s} \mathrm{d}  \mathbf{X}_{s}:=\lim_{| \pi|\to 0}\sum_{[u,v]\in \pi} S_{t-u}(Y_{u} X_{u,v}+Y_{u}^{\prime}\mathbb{X}_{u,v}),
\end{equation}
exists in $\mathcal H_{\gamma}$, where the limit is taken along any partitions $\pi$ of $[0,t]$ and is independent of the specific choice of these partitions.

Moreover, for every $0\leq\beta<3\alpha$ the above integral satisfies the following bound:
\begin{align} \label{IntBound}
		& \left|\int_{s}^{t}S_{t-u}Y_{u} \mathrm{d} \mathbf X_{u}-S_{t-s}Y_{s}  X_{s,t}-S_{t-s}Y_{s}^{\prime} \mathbb X_{s,t} \right|_{\mathcal H_{\gamma-2\alpha+\beta} } \nonumber\\
		\lesssim & \varrho_{\alpha}(\mathbf{X}) \|Y, Y^{\prime}\|_{ \mathcal{D}_{X,\gamma}^{2\alpha} } |t-s|^{3\alpha-\beta},
\end{align}
for all $(s,t)\in\Delta_{[0,T]}$. Here and in the following, for notational simplicity, the norm of a controlled rough path $(Y,Y^{\prime}) \in \mathcal{D}_{X}^{2\alpha}([0,T];\mathcal{H}_{\gamma}^{d})$ is still denoted by $\|Y,Y^{\prime}\|_{\mathcal{D}_{X,\gamma}^{2\alpha} }$.
\end{lemma}

The following lemma tells us not only about the stability of rough integration, but also that `rough convolution' improves the spatial regularity of the controlled rough path, see \cite[Lemma 3.4]{HN22} or \cite[Corollary 4.6]{GHN21}. For the reader's convenience, the proof is provided in the Appendix.

\begin{lemma} \label{stab}
Let $\mathbf{X}=(X,\mathbb{X}) \in \mathscr{C}^{\alpha}([0,T];\mathbf{R}^{d})$ for some $\alpha \in (1/3, 1/2]$. Let $(Y,Y^{\prime}) \in \mathcal{D}_{X}^{2\alpha}([0,T];\mathcal{H}_{\gamma}^d)$. Then for every $0 \leq \sigma < \alpha$, the mapping
$$
(Y,Y^{\prime}) \to (Z,Z^{\prime}): = \left( \int_{0}^{\cdot} S_{\cdot-s} Y_{s} \mathrm{d}\mathbf{X}_{s}, Y \right) \in \mathcal D_{X}^{2\alpha} ([0,T];\mathcal{H}_{\gamma+\sigma})
$$
is continuous and the following estimate holds true:
\begin{equation}\label{Int}
\|Z,Z^{\prime}\|_{\mathcal{D}_{X,\gamma+\sigma}^{2\alpha}} \lesssim \left(1+\varrho_{\alpha}(\mathbf{X}) \right) \left(|Y_{0}|_{\mathcal H_{\gamma}^d} +|Y_{0}^\prime|_{\mathcal H_{\gamma-\alpha}^{d \times d}} \right)+ T^{\alpha - \sigma} \left(1+\varrho_{\alpha}(\mathbf{X}) \right) \|Y,Y^{\prime}\|_{\mathcal{D}_{X,\gamma}^{2\alpha}}.
\end{equation}
\end{lemma}

\subsection{Main result}\label{AP}

Let $\alpha_0 \in (1/3,1/2]$ and let $(\Omega, \mathcal{F}, \{\mathcal{F}_t\}_{t \in [0, T]},\mathbb{P})$ be a filtered probability space satisfying the usual conditions.
We consider a semilinear slow--fast system defined on a monotone family of interpolation Hilbert spaces $(\mathcal H_\gamma)_{\gamma \in \mathbf R}$.
\begin{equation}\label{eq1}
	\left\{ \begin{aligned}
		&\mathrm{d} X_t^\varepsilon=\big[L X_t^\varepsilon+F_1 (X_t^\varepsilon,Y_t^\varepsilon)\big]\mathrm{d} t+G_1 (X_t^\varepsilon)\mathrm{d} \mathbf B_t,\\
		&\mathrm{d} Y_t^\varepsilon=\frac 1\varepsilon \big[L Y_t^\varepsilon+F_2 (X_t^\varepsilon,
		Y_t^\varepsilon)\big]\mathrm{d} t+\frac 1{\sqrt{\varepsilon}}G_2 (X_t^\varepsilon,
		Y_t^\varepsilon)  \mathrm{d} \mathbf W^{\text{It\^{o}} }_t,\\
		&X_0^\varepsilon=x, ~~ Y_0^\varepsilon=y .
	\end{aligned}  \right.
\end{equation}
Here,
$\mathbf B=\{(B_{s,t}, \mathbb B_{s,t})\}_{0\leq s\leq t\leq T}$ is an $\mathscr{C}^{\alpha}([0,T];\mathbf{R}^{d_1})$-valued random variable (i.e., random rough path) defined on probability space $(\Omega, \mathcal{F}, \{\mathcal{F}_t\}_{t \in [0, T]}, \mathbb{P})$ for every $\alpha \in (1/3, \alpha_0)$,
and let $t\mapsto (B_{0,t}, \mathbb B_{0,t})$ is $\{\mathcal{F}_t\}$-adapted,
$\{W_t\}_{t\in [0,T]}$ is a standard $d_2$-dimensional $\{\mathcal{F}_t\}$-Brownian motion.
The It\^{o} rough path lift of $W$ is denoted by $\mathbf W^{\text{It\^{o}} }=(W,\mathbb W^{\text{It\^{o}} })$.
We assume that $W$ and $\mathbf B$ are independent.
The coefficients $F_1, F_2$ and $G_1, G_2$ are nonlinear terms satisfying certain suitable assumptions.
By means of (\ref{int}), Eq. (\ref{eq1}) can be formulated in a mild form
\begin{equation}\label{solution}
	\left\{ \begin{aligned}
		&X_t^\varepsilon=S_t x+ \int_0^t S_{t-s} F_1 (X_s^\varepsilon,Y_s^\varepsilon) \mathrm{d}s+ \int_0^t S_{t-s} G_1 (X_s^\varepsilon) \mathrm{d} \mathbf B_s, \\
		&Y_t^\varepsilon= S_{t/\varepsilon} y+ \frac 1\varepsilon \int_0^t S_{(t-s)/{\varepsilon}} F_2 (X_s^\varepsilon,Y_s^\varepsilon) \mathrm{d}s +\frac 1{\sqrt{\varepsilon}}\int_0^t S_{(t-s)/{\varepsilon}} G_2 (X_s^\varepsilon,Y_s^\varepsilon) \mathrm{d} \mathbf W^{\text{It\^{o}}}_s.
	\end{aligned}  \right.
\end{equation}

We impose the following assumptions on the coefficients $F_1, F_2, G_1$ and $G_2$ in Eq. (\ref{eq1}).

\begin{enumerate}
	\item [$(\mathbf{H1})$]
	Let $\gamma \in \mathbf R$,
and let $F_1:\mathcal{H}_\gamma \times \mathcal{H}_{\gamma}  \to  \mathcal{H}_{\gamma-\alpha}$ and
$F_2:\mathcal{H}_\gamma \times \mathcal{H}_{\gamma} \to\mathcal{H}_{\gamma} $.
Moreover, there exist constants $C,L_{F_2}>0$ such that for all $x,x_1,x_2 \in  \mathcal{H}_\gamma$, $y, y_1,y_2\in \mathcal{H}_\gamma$ we have
		\begin{equation}\label{F1Lip}
			|F_1(x_1,y_1)-F_1(x_2,y_2)|_{\mathcal H_{\gamma-\alpha}} \leq  C(|x_1-x_2|_{\mathcal H_\gamma}+|y_1-y_2|_{\mathcal H_\gamma}),
		\end{equation}
		\begin{equation}\label{F1Bound}
		\|F_1\|_{\infty, \mathcal H_{\gamma-\alpha}}:=\sup_{(x,y)\in \mathcal{H}_\gamma \times \mathcal{H}_\gamma } |F_1(x,y)|_{\mathcal H_{\gamma-\alpha}} <\infty,
	\end{equation}
	and
	\begin{equation}\label{F2Lip}
	|F_2(x_1,y_1)-F_2(x_2,y_2)|_{\mathcal H_{\gamma}} \leq  C|x_1-x_2|_{\mathcal H_\gamma}+L_{F_2}|y_1-y_2|_{\mathcal H_\gamma}.
	\end{equation}

\item [$(\mathbf{H2})$] Let $G_1 \in \mathcal C_b^3(\mathcal{H}_{\gamma-\theta} ; \mathcal{H}_{\gamma-\theta-\sigma}^{d_1})$
with some $\sigma \in [0,\frac \alpha2)$ and $\theta \in \{0,\alpha, 2\alpha\}$,
i.e. there is a constant $C>0$ such that for any $x,y \in \mathcal{H}_{\gamma-\theta}$ we have
\begin{align}
	&|G_1(x)-G_1(y)|_{\mathcal{H}_{\gamma-\theta-\sigma}^{d_1}} \leq C |x-y|_{\mathcal H_{\gamma-\theta}}, \label{G1} \\
	&  |DG_1(x)-DG_1(y)|_{\mathcal L(\mathcal H_{\gamma-\theta}; \mathcal{H}_{\gamma-\theta-\sigma}^{d_1})}
		\leq C |x-y|_{\mathcal H_{\gamma-\theta}},   \label{DG1}    \\
	&  |D^2G_1(x)-D^2G_1(y)|_{\mathcal L (\mathcal H_{\gamma-\theta}; \mathcal L(\mathcal H_{\gamma-\theta}; \mathcal{H}_{\gamma-\theta-\sigma}^{d_1})) }
		\leq C |x-y|_{\mathcal H_{\gamma-\theta}}.   \label{D2G1}
	\end{align}	
Moreover, we assume that
\begin{equation}\label{DG1G1}
	|DG_1(x)G_1(x)-DG_1(y)G_1(y)|_{\mathcal{H}_{\gamma-\theta-\sigma}^{d_1 \times d_1}}\leq  C |x-y|_{\mathcal H_{\gamma-\theta}}.
\end{equation}

\item [$(\mathbf{H3})$] Let
		$G_2 \in \mathcal C_b^3(\mathcal{H}_{\gamma-\theta} \times \mathcal{H}_{\gamma-\theta}; \mathcal{H}_{\gamma-\theta-\sigma}^{d_2})$ with some $\sigma \in [0,\frac \alpha2)$ and $\theta \in \{0,\alpha, 2\alpha\}$,
	and there exist constants $C, L_{G_2}>0$ such that for any $x_1,x_2,y_1,y_2 \in \mathcal{H}_{\gamma-\theta}$, we have
	\begin{equation*}
		|G_2(x_1,y_1)-G_2(x_2,y_2)|_{\mathcal{H}_{\gamma-\theta}^{d_2} } \leq  C|x_1-x_2|_{\mathcal H_{\gamma-\theta}}+L_{G_2}|y_1-y_2|_{\mathcal H_{\gamma-\theta}}.
	\end{equation*}

	\item [$(\mathbf{H4})$] Let $\{e_n\}_{n \in \mathbf N}$ be a complete orthonormal basis of $\mathcal{H}_{\gamma}$, there exists a non-decreasing sequence of real positive numbers $\{ \lambda_n\}_{n \in \mathbf N}$ such that
	\begin{align*}
		L e_n=-\lambda_n e_n, ~~~~\text{for all}~ n \in \mathbf N.
	\end{align*}

	\item[$(\mathbf{H5})$] The smallest eigenvalue $\lambda_1$ of $-L$ and the constants $L_{F_2},L_{G_2}$ satisfy
	$$ \lambda_1- L_{F_2}-3 L_{G_2}^2 >0.$$

	\item[$(\mathbf{H6})$] For any $\alpha \in (1/3,\alpha_0)$ and any $p\in [1,\infty)$, we set
	$$\mathbb E\big[\interleave \mathbf{B} \interleave_\alpha^p \big] <\infty,$$
	where the homogeneous norm $\interleave \mathbf{B} \interleave_\alpha:=|B|_\alpha +\sqrt{ |\mathbb B|_{2\alpha} }$,
	which, although not a norm in the usual sense of normed linear spaces, is a very adequate concept for rough paths.

\end{enumerate}

\begin{remark}
(i) The assumption $(\mathbf{H5})$ is mainly used to ensure that the mixed random rough path $\mathbf M=(M, \mathbb M) \in \mathscr{C}^\alpha ([0,T]; \mathbf R^{d_1+d_2}), \mathbb P$-a.s. for $\alpha \in (1/3, \alpha_0)$ ~(see Lemma \ref{mixrandom} in Section \ref{proba}).

(ii) It is worth noting that, for the well-posedness, the spatial regularity assumption on $F_1$ can be relaxed to $F_1:\mathcal{H}_\gamma \times \mathcal{H}_{\gamma}  \to  \mathcal{H}_{\gamma-\delta}$ for some $\delta \in [2\alpha,1)$.
However, for the proof of the averaging principle, a stronger regularity condition is needed, namely $F_1:\mathcal{H}_\gamma \times \mathcal{H}_{\gamma}  \to \mathcal{H}_{\gamma-\alpha}$, in order to derive crucial estimates required in the convergence analysis.

(iii) In contrast to the assumptions in \cite{HN22,YLZ23}, we impose a more stringent constraint $\sigma \in [0, \frac{\alpha}{2})$ in conditions $(\mathbf{H2})$ and $(\mathbf{H3})$, which is necessitated by the technical estimates employed in the proof of the averaging principle.
\end{remark}

Throughout this paper, we assume that there exists a unique global mild solution to Eq. (\ref{eq1}) (see Section \ref{sect3.1} for an explanation).

\begin{lemma}\label{wellpose}
	Let $\gamma \in \mathbf R$ and $(x,y) \in \mathcal H_\gamma \times \mathcal H_\gamma$, and suppose that $(\mathbf{H1})$--$(\mathbf{H3})$ and $(\mathbf{H6})$ hold. Then there exists a unique function $(X^\varepsilon, Y^\varepsilon) \in \mathcal C([0,T]; \mathcal H_\gamma) \times \mathcal C([0,T]; \mathcal H_\gamma)$ such that $\left(X^\varepsilon,G_1(X^\varepsilon) \right) \times \left( Y^\varepsilon,G_2(X^\varepsilon,Y^\varepsilon) \right)$ are controlled rough paths in the sense of Definition \ref{CRP} and $(X^\varepsilon, Y^\varepsilon) $ is a mild solution to the slow--fast system (\ref{eq1}) of rough PDEs, namely for $\omega \in \Omega$
	\begin{equation*}
		\left\{ \begin{aligned}
			&X_t^\varepsilon=S_t x+ \int_0^t S_{t-s} F_1 (X_s^\varepsilon,Y_s^\varepsilon) \mathrm{d}s+ \int_0^t S_{t-s} G_1 (X_s^\varepsilon) \mathrm{d} \mathbf B_s(\omega), \\
			&Y_t^\varepsilon= S_{t/\varepsilon} y+ \frac 1\varepsilon \int_0^t S_{(t-s)/{\varepsilon}} F_2 (X_s^\varepsilon,Y_s^\varepsilon) \mathrm{d} s +\frac 1{\sqrt{\varepsilon}}\int_0^t S_{(t-s)/{\varepsilon}} G_2 (X_s^\varepsilon,Y_s^\varepsilon) \mathrm{d} \mathbf W^{\text{It\^{o}} }_s(\omega).
		\end{aligned}  \right.
	\end{equation*}
	
\end{lemma}

Given $x \in \mathcal H_\gamma$, consider the following frozen equation:
 \begin{equation}\label{frozen}
	\left\{ \begin{aligned}
		\mathrm{d} Y_t&=\big[L Y_t+F_2(x,Y_t)\big] \mathrm{d} t+G_2(x,Y_t) \mathrm{d} \widetilde{W}_t,\\
		Y_0&= y,
	\end{aligned} \right.
\end{equation}
where $\{\widetilde{W}_t\}_{t\geq0}$ is a $d_2$-dimensional Brownian motion defined on another probability space $(\widetilde{\Omega} ,\widetilde{\mathcal F},\{\widetilde{\mathcal F}_t\}_{t\geq0},\widetilde{\mathbb P})$, and independent of $\{W_t\}_{t \in [0,T]}$.

Under conditions (\ref{F2Lip}) and $(\mathbf{H3})$, it is easy to prove that Eq. (\ref{frozen}) has a unique strong solution (see \cite{DF10}), given by
\begin{equation*}
 Y_t^{x,y}=y + \int_0^t L Y_s^{x,y} \mathrm{d} s +\int_0^t F_2(x, Y_s^{x,y}) \mathrm{d}s+ \int_0^t G_2(x, Y_s^{x,y}) \mathrm{d} \widetilde{W}_s, ~~\text{for any}~t >0.
 \end{equation*}
Moreover, the process $\{Y_t^{x,y}\}_{t \geq 0}$ admits a unique invariant probability measure $\mu^x (\mathrm{d} y)$ under the condition ($\mathbf{H5}$) (see \cite[Theorem 2.2]{HLL20}).
The averaged equation for system (\ref{eq1}) is given by
\begin{equation}\label{aveq}
	\left\{ \begin{aligned}
		\mathrm{d} \bar{X}_t&=\big[ L \bar{X}_t+\bar{F}_1 (\bar{X}_t)\big] \mathrm{d} t+G_1(\bar{X}_t)\mathrm{d} \mathbf B_t,\\
		\bar{X}_0&=x,
	\end{aligned} \right.
\end{equation}
where
$$\bar{F}_1 (x):=\displaystyle \int_{\mathcal{H}_\gamma} F_1 (x,y) \mu^x (\mathrm{d} y).$$
According to (\ref{F1Lip}), we know that $F_1$ is Lipschitz continuous. Therefore, we can show that $\bar{F}_1:\mathcal{H}_\gamma \to \mathcal{H}_{\gamma-\alpha}$ is also Lipschitz continuous (see e.g. \cite[Remark 4.1]{LLPX25}), i.e. for any $x_1,x_2 \in \mathcal H_\gamma$ we have
\begin{equation}\label{Fbar}
	|\bar{F}_1(x_1)-\bar{F}_1(x_2)|_{\mathcal H_{\gamma-\alpha}} \leq C |x_1-x_2|_{\mathcal H_\gamma}.
\end{equation}
Thus Eq. (\ref{aveq}) admits a unique global solution $\{\bar{X}_t\}_{t \in [0,T]}$ (see \cite[Theorem 3.2]{HN22}).
Below we give the main result of this paper.

\begin{theorem}\label{MainRes}
	Let $\gamma \in \mathbf R$, and let $X_t^\varepsilon$ and $\bar{X}_t$ satisfy Eqs. (\ref{eq1}) and (\ref{aveq}), with common initial value $x \in \mathcal H_{\gamma+\zeta}$ with $\zeta \in (\frac \alpha2,\alpha-\sigma)$.
	Assume that $(\mathbf{H1})$--$(\mathbf{H6})$ hold. Then we have
	\begin{equation}\label{Averaging}
		\lim_{\varepsilon\rightarrow 0}  \mathbb E\left[\sup_{t \in [0,T]}|X_t^\varepsilon-\bar{X}_t|_{\mathcal H_\gamma}^{2}\right]=0.
	\end{equation}
\end{theorem}

\section{Slow--fast system of rough PDEs}\label{sect3}

\subsection{Deterministic aspects} \label{sect3.1}

In this subsection we discuss slow--fast system (\ref{eq1}) of rough PDEs from a deterministic perspective.
To this end, we first introduce some notations.
As before, $\alpha \in (1/3,\alpha_0)$ is assumed.
Denote by $\mathbf R^{d_1+d_2}:=\mathbf R^{d_1} \oplus \mathbf R^{d_2}$ the direct sum decomposition.
Let $\mathbf{M}=(M, \mathbb{M})$ is a rough path over $\mathbf R^{d_1+d_2}$ and goes as follows
\begin{equation} \label{mixRP}
	M_t=\begin{pmatrix} B_t \\ W_t  \end{pmatrix},   ~~~~\mathbb M_{s,t}=\begin{pmatrix} \mathbb{B}_{s,t} &\displaystyle I[B,W]_{s,t}  \\ \displaystyle I[W,B]_{s,t} & \mathbb W_{s,t} \end{pmatrix},
\end{equation}
where $\mathbf B=(B, \mathbb B) \in \mathscr{C}^{\alpha}([0,T]; \mathbf R^{d_1})$, $\mathbf W=(W, \mathbb W) \in \mathscr{C}^{\alpha}([0,T]; \mathbf R^{d_2})$ and $I[B,W]$ takes values in $\mathbf R^{d_1 \times d_2}$.
Roughly speaking, $I[B,W]_{s,t}$ plays the role of the iterated integral $\int_s^t B_{s,u} \otimes \mathrm{d} W_u$. And $I[W,B]$ is interpreted in a similar way.

For $\gamma \in \mathbf R$,
let $\mathscr H_\gamma:=\mathcal{H}_\gamma \times \mathcal{H}_\gamma$ be the product Hilbert space. For any $u=(u_1,u_2),v=(v_1,v_2)\in\mathscr{H}_\gamma$, we denote the scalar product and the induced norm by
$$\langle u,v \rangle_{\mathscr{H}_\gamma}=\langle u_1,v_1\rangle_{\mathcal{H}_\gamma}+\langle u_2,
v_2\rangle_{\mathcal{H}_\gamma},~~~~|u|_{\mathscr{H}_\gamma}=\sqrt{\langle u,u\rangle_{\mathscr{H}_\gamma}}=\sqrt{|u_1|_{\mathcal{H}_\gamma}^2+|u_2|_{\mathcal{H}_\gamma}^2}.$$
Let $Z^\varepsilon=(X^\varepsilon, Y^\varepsilon)^\top \in \mathscr H_\gamma$. We write
\begin{align*}
	\mathcal S_t&:=\text{diag} \big( S_t ,S_{t/\varepsilon} \big),\\
	\mathcal F(Z_s^\varepsilon)&:= \left( F_1(X_s^\varepsilon, Y_s^\varepsilon), \frac 1 \varepsilon  F_2 (X_s^\varepsilon, Y_s^\varepsilon)\right)^\top, \\
	\mathcal G(Z_s^\varepsilon)&:= \text{diag} \left( G_1(X_s^\varepsilon), \frac 1 {\sqrt{\varepsilon}}  G_2 (X_s^\varepsilon, Y_s^\varepsilon)\right).
\end{align*}
With the above preparations, the slow--fast system (\ref{eq1}) of rough PDEs can be reformulated in a deterministic sense as follows:
\begin{equation} \label{eq2}
Z_t^\varepsilon= \mathcal S_t Z_0^\varepsilon+\int_0^t \mathcal{S}_{t-s} \mathcal F(Z_s^\varepsilon) \mathrm{d} s+\int_0^t \mathcal{S}_{t-s} \mathcal G(Z_s^\varepsilon) \mathrm{d} \mathbf M_s,  ~~~~(Z^\varepsilon)_t^\prime=\mathcal G(Z_t^\varepsilon),
\end{equation}
where $Z_0^\varepsilon=(x, y)^\top$.

Under assumptions $(\mathbf{H1})$--$(\mathbf{H3})$,
it can be proved by the concatenation procedure that Eq. (\ref{eq2}) has a unique global solution given by $(Z^\varepsilon, \mathcal G(Z^\varepsilon)) \in \mathcal{D}_{M}^{2\alpha}([0,T];\mathscr{H}_\gamma)$,
which is controlled by $\mathbf M=(M, \mathbb M) \in \mathscr{C}^\alpha ([0,T]; \mathbf R^{d_1+d_2})$ for $\alpha \in (1/3, \alpha_0)$, see \cite[Theorem 3.2]{HN22}.

The following result come from \cite[Lemma 3.1]{LLPX25}, the details of the proof are omitted.

\begin{lemma}
Let $(Z^\varepsilon, \mathcal G(Z^\varepsilon)) \in \mathcal{D}_{M}^{2\alpha}([0,T];\mathscr{H}_\gamma)$ be a unique solution of Eq. (\ref{eq2}).
Then $(X^\varepsilon, G_1(X^\varepsilon))$ belongs to $\mathcal D_{B}^{2\alpha}([0,T];\mathcal{H}_\gamma)$ and is a unique solution of the following rough PDE driven by $\mathbf B=(B,\mathbb B)$:
\begin{equation}\label{sloweq}
X_t^\varepsilon=S_t x+\int_0^t S_{t-s} F_1 (X_s^\varepsilon,Y_s^\varepsilon) \mathrm{d} s+ \int_0^t S_{t-s} G_1 (X_s^\varepsilon) \mathrm{d} \mathbf{B}_s, ~~~~(X^\varepsilon)'_t=G_1 (X_t^\varepsilon).
\end{equation}
\end{lemma}

Below we give some a priori estimates for rough PDEs (\ref{sloweq}), which will be used frequently in the sequel.

\begin{lemma}\label{funCRP}
	Let $G_1$ satisfies the condition $(\mathbf{H2})$, and let $(X^\varepsilon, G_1(X^\varepsilon)) \in \mathcal D_{B}^{2\alpha}([0,T];\mathcal{H}_\gamma)$ be a solution to Eq. (\ref{sloweq}). 	
	Then $(G_1(X^\varepsilon),DG_1(X^\varepsilon)G_1(X^\varepsilon))\in \mathcal{D}_{ B}^{2\alpha}([0,T];\mathcal{H}_{\gamma-\sigma}^{d_1})$ and satisfies the following bound:
	$$\|G_1(X^\varepsilon),DG_1(X^\varepsilon)G_1(X^\varepsilon)\|_{\mathcal{D}_{B,\gamma-\sigma}^{2\alpha} }
	\lesssim (1+\varrho_{\alpha}(\mathbf{B}))^2 \left(1+\|X^\varepsilon,G_1(X^\varepsilon)\|_{\mathcal{D}_{B,\gamma}^{2\alpha}} \right).$$
\end{lemma}
\begin{proof}
	The proof follows by arguments analogous to those in Lemma \ref{Glem2} (see also \cite[Lemma 3.5]{HN22} or \cite[Lemma 3.2]{YLZ23}), and is therefore omitted here.
\end{proof}

	\begin{lemma}\label{slowest}	
		Let $x \in \mathcal H_\gamma$. Then there exists a constant $C_{T,\varrho_{\alpha}(\mathbf{B})}>0$ dependent on the parameters $T,\varrho_{\alpha}(\mathbf{B})$ such that	
		\begin{equation}\label{X1}
			\|X^\varepsilon,G_1(X^\varepsilon)\|_{\mathcal{D}_{B,\gamma}^{2\alpha}} \leq C_{T,\varrho_{\alpha}(\mathbf{B})} (1+|x|_{\mathcal H_\gamma}).
		\end{equation}
		Furthermore, by (\ref{norm}), it is straightforward that
		\begin{equation}\label{X2}	
		\|X^\varepsilon\|_{\infty, \mathcal H_\gamma} \leq C_{T,\varrho_{\alpha}(\mathbf{B})} (1+|x|_{\mathcal H_\gamma}).	
		\end{equation}	
	\end{lemma}
	
	\begin{proof}
		Recall that
		\begin{equation*}
			X_t^\varepsilon=S_t x+\int_0^t S_{t-s} F_1 (X_s^\varepsilon,Y_s^\varepsilon) \mathrm{d} s+ \int_0^t S_{t-s} G_1 (X_s^\varepsilon) \mathrm{d} \mathbf{B}_s.
		\end{equation*}
		It is easy to see that
		\begin{align*}
			\|X^\varepsilon,G_1(X^\varepsilon)\|_{\mathcal{D}_{B,\gamma}^{2\alpha}}
			\leq & \|S_\cdot x, 0 \|_{\mathcal{D}_{B,\gamma}^{2\alpha}}
			+\left\|\int_0^\cdot S_{\cdot-s} F_1(X_s^\varepsilon,Y_s^\varepsilon) \mathrm{d} s, 0 \right\|_{\mathcal{D}_{B,\gamma}^{2\alpha}}\\
			&+\left\|\int_0^\cdot S_{\cdot-s} G_1(X_s^\varepsilon) \mathrm{d} \mathbf B_s, G_1(X^\varepsilon) \right\|_{\mathcal{D}_{B,\gamma}^{2\alpha}}.
		\end{align*}
		For $\|S_\cdot x, 0 \|_{\mathcal{D}_{B,\gamma}^{2\alpha}}$, by (\ref{norm}) we have
		$$\|S_\cdot x, 0\|_{\mathcal{D}_{B,\gamma}^{2\alpha}}=\|S_\cdot x\|_{\infty, \mathcal H_\gamma}+|S_\cdot x|_{\alpha,\mathcal H_{\gamma-\alpha}}+|S_\cdot x|_{2\alpha,\mathcal H_{\gamma-2\alpha}}.$$
		Clearly,
		\begin{equation*}
			\|S_\cdot x\|_{\infty, \mathcal H_\gamma}=\sup_{t \in [0,T]} |S_t x|_{\mathcal H_\gamma} \lesssim |x|_{\mathcal H_\gamma}.
		\end{equation*}	
		For $|S_\cdot x|_{\theta,\mathcal H_{\gamma-\theta}}$ with $\theta \in \{\alpha, 2\alpha\}$, using (\ref{group2}) we have
		\begin{align*}
			|S_\cdot x|_{\theta,\mathcal H_{\gamma-\theta}}
			= & \sup_{0\leq s<t \leq T}  \frac{|S_t x-S_s x|_{\mathcal H_{\gamma-\theta}}}{|t-s|^\theta}
			=\sup_{0\leq s<t \leq T}  \frac{|(S_{t-s}-\mathrm{id} )S_s x|_{\mathcal H_{\gamma-\theta}}}{|t-s|^\theta} \\
			\lesssim & \sup_{0\leq s<t \leq T}  \frac{|t-s|^\theta  |S_s x|_{\mathcal H_\gamma}}{|t-s|^\theta}
			\lesssim |x|_{\mathcal H_\gamma}.
		\end{align*}
		It is therefore easy to obtain that
		\begin{equation}\label{xest}
			\|S_\cdot x, 0\|_{\mathcal{D}_{B,\gamma}^{2\alpha}} \lesssim |x|_{\mathcal H_\gamma}.
		\end{equation}
		For $ \left\|\int_0^\cdot S_{\cdot-s} F_1(X_s^\varepsilon,Y_s^\varepsilon) \mathrm{d} s, 0 \right\|_{\mathcal{D}_{B,\gamma}^{2\alpha}}$, one has
		\begin{align*}
			\left\|\int_0^\cdot S_{\cdot-s} F_1(X_s^\varepsilon,Y_s^\varepsilon) \mathrm{d} s, 0 \right\|_{\mathcal{D}_{B,\gamma}^{2\alpha}}
			=&\left\|\int_0^\cdot S_{\cdot-s} F_1(X_s^\varepsilon,Y_s^\varepsilon) \mathrm{d} s \right\|_{\infty,\mathcal H_\gamma}
			+\left|\int_0^\cdot S_{\cdot-s} F_1(X_s^\varepsilon,Y_s^\varepsilon) \mathrm{d} s \right|_{\alpha,\mathcal H_{\gamma-\alpha}} \nonumber \\
			&+ \left|\int_0^\cdot S_{\cdot-s} F_1(X_s^\varepsilon,Y_s^\varepsilon) \mathrm{d} s  \right|_{2\alpha,\mathcal H_{\gamma-2\alpha}}.
		\end{align*}
		According to (\ref{group1}) and (\ref{F1Bound}) we have
		\begin{align}\label{F1sup}
			&\left\|\int_0^\cdot S_{\cdot-s} F_1(X_s^\varepsilon,Y_s^\varepsilon) \mathrm{d} s \right\|_{\infty,\mathcal H_\gamma}
			=\sup_{t\in [0,T]} \left| \int_0^t S_{t-s} F_1(X_s^\varepsilon,Y_s^\varepsilon)  \mathrm{d} s \right|_{\mathcal H_\gamma} \nonumber\\
			\lesssim & \sup_{t\in [0,T]} \int_0^t (t-s)^{-\alpha} |F_1(X_s^\varepsilon,Y_s^\varepsilon)|_{\mathcal H_{\gamma-\alpha}}  \mathrm{d} s
			\lesssim  \sup_{t\in [0,T]} \int_0^t  (t-s)^{-\alpha}  \mathrm{d} s
			\lesssim  T^{1-\alpha} .
		\end{align}
		For $\left| \int_0^\cdot S_{\cdot-s} F_1(X_s^\varepsilon,Y_s^\varepsilon) \mathrm{d} s \right|_{\theta,\mathcal H_{\gamma-\theta}}$ with $\theta \in \{\alpha, 2\alpha\}$, it is trivial that
		\begin{align*}
			&\int_0^t S_{t-r} F_1(X_r^\varepsilon,Y_r^\varepsilon)\mathrm{d} r-\int_0^s S_{s-r} F_1(X_r^\varepsilon,Y_r^\varepsilon)\mathrm{d} r \\
			=&\int_s^t S_{t-r} F_1(X_r^\varepsilon,Y_r^\varepsilon)\mathrm{d} r
			+(S_{t-s}-\mathrm{id} ) \int_0^s S_{s-r} F_1(X_r^\varepsilon,Y_r^\varepsilon)\mathrm{d} r.
		\end{align*}	
		Using (\ref{group1}), $\mathcal H_{\gamma-\alpha} \hookrightarrow \mathcal H_{\gamma-2\alpha}$ and (\ref{F1Bound}) one has
		\begin{align}\label{F1alp}
			\left|\int_s^t S_{t-r} F_1(X_r^\varepsilon,Y_r^\varepsilon)\mathrm{d} r \right|_{\mathcal H_{\gamma-\theta}}
			\lesssim &\int_s^t (t-r)^{\theta-2\alpha} |F_1(X_r^\varepsilon,Y_r^\varepsilon)|_{\mathcal H_{\gamma-2\alpha}} \mathrm{d} r
			\lesssim (t-s)^{1+\theta-2\alpha}.
		\end{align}
		From (\ref{group2}) and (\ref{F1sup}) we obtain
		\begin{align}\label{F12alp}
			& \left| (S_{t-s}-\mathrm{id} ) \int_0^s S_{s-r} F_1(X_r^\varepsilon,Y_r^\varepsilon)\mathrm{d} r \right|_{\mathcal H_{\gamma-\theta}} \nonumber\\
			\lesssim & |t-s|^\theta \left|\int_0^s S_{s-r} F_1(X_r^\varepsilon,Y_r^\varepsilon)\mathrm{d} r \right|_{\mathcal H_\gamma}	
			\lesssim  (t-s)^\theta T^{1-\alpha}.
		\end{align}
		Combining estimates (\ref{F1sup})--(\ref{F12alp}) we get
		\begin{equation}\label{F1res}
			\left\|\int_0^\cdot S_{\cdot-s} F_1(X_s^\varepsilon,Y_s^\varepsilon) \mathrm{d} s, 0 \right\|_{\mathcal{D}_{B,\gamma}^{2\alpha}}
			\lesssim T^{1-2\alpha}+T^{1-\alpha}.
		\end{equation}

		For $ \left\|\int_0^\cdot S_{\cdot-s} G_1(X_s^\varepsilon) \mathrm{d} \mathbf B_s, G_1(X^\varepsilon) \right\|_{\mathcal{D}_{B,\gamma}^{2\alpha}}$,
		using (\ref{Int}), (\ref{G1}), (\ref{DG1G1}) and $\mathcal H_\gamma \hookrightarrow \mathcal H_{\gamma-\alpha} $ we have
		\begin{align}\label{G1Int}
			&\left\|\int_0^\cdot S_{\cdot-s} G_1(X_s^\varepsilon) \mathrm{d} \mathbf B_s, G_1(X^\varepsilon) \right\|_{\mathcal{D}_{B,\gamma}^{2\alpha}} \nonumber\\
			\lesssim & T^{\alpha-\sigma} \left(1+\varrho_{\alpha}(\mathbf{B}) \right) \|G_1(X^\varepsilon),DG_1(X^\varepsilon)G_1(X^\varepsilon)\|_{\mathcal{D}_{B,\gamma-\sigma}^{2\alpha}} \nonumber\\
			&+(1+\varrho_{\alpha}(\mathbf{B})) \left(|G_1(x)|_{\mathcal H_{\gamma-\sigma}^{d_1}}+|DG_1(x)G_1(x)|_{\mathcal H_{\gamma-\alpha-\sigma}^{d_1 \times d_1}} \right) \nonumber\\
			\lesssim & 	(1+\varrho_{\alpha}(\mathbf{B})) \left(1+|x|_{\mathcal H_{\gamma}} \right)
			+T^{\alpha-\sigma} \left(1+\varrho_{\alpha}(\mathbf{B}) \right) \|G_1(X^\varepsilon),DG_1(X^\varepsilon)G_1(X^\varepsilon)\|_{\mathcal{D}_{B,\gamma-\sigma}^{2\alpha}} .
		\end{align}
		Finally, combining estimates (\ref{xest}), (\ref{F1res}) and (\ref{G1Int}), we obtain
		\begin{align*}
			& \|X^\varepsilon,G_1(X^\varepsilon)\|_{\mathcal{D}_{B,\gamma}^{2\alpha}} \nonumber\\
			\lesssim & (1+\varrho_{\alpha}(\mathbf{B}))(1+|x|_{\mathcal H_\gamma})+T^{1-2\alpha}+T^{1-\alpha}
			+T^{\alpha-\sigma} \left(1+\varrho_{\alpha}(\mathbf{B}) \right) \|G_1(X^\varepsilon),DG_1(X^\varepsilon)G_1(X^\varepsilon)\|_{\mathcal{D}_{B,\gamma-\sigma}^{2\alpha}} \nonumber\\
			\leq & C_{T} (1+\varrho_{\alpha}(\mathbf{B})) (1+|x|_{\mathcal H_\gamma})+CT^{\alpha-\sigma} (1+\varrho_{\alpha}(\mathbf{B}))^3 \left(1+\|X^\varepsilon,G_1(X^\varepsilon)\|_{\mathcal{D}_{B,\gamma}^{2\alpha}} \right) \nonumber\\
			\leq & C_{T} (1+\varrho_{\alpha}(\mathbf{B}))^3  (1+|x|_{\mathcal H_\gamma})
			+CT^{\alpha-\sigma}  (1+\varrho_{\alpha}(\mathbf{B}))^3  \|X^\varepsilon,G_1(X^\varepsilon)\|_{\mathcal{D}_{B,\gamma}^{2\alpha}},
		\end{align*}
		where for the second inequality we used Lemma \ref{funCRP}.
		Assuming that the time horizon $T\in (0,1]$ is small enough such that $C T^{\alpha-\sigma} (1+\varrho_{\alpha}(\mathbf{B}))^3 \leq 1/2$, we directly deduce the following estimate:
		\begin{equation*}
			\|X^\varepsilon,G_1(X^\varepsilon)\|_{\mathcal{D}_{B,\gamma}^{2\alpha}} \leq  C_{T,\varrho_{\alpha}(\mathbf{B}) } (1+|x|_{\mathcal H_\gamma}).
		\end{equation*}

		For any $T>0$, we divide the time interval $[0,T]$ into $N \in \mathbf N$ identical subintervals such that $   1/4 < C {(\frac T N)}^{\alpha-\sigma} (1+\varrho_{\alpha}(\mathbf{B}))^3 \leq 1/2$,
		then on each subinterval $\left[\frac kN T,\frac {k+1}N T \right]$ with $k=0, \cdots , N-1$, we have
		\begin{equation}\label{normk}
			\|X^\varepsilon,G_1(X^\varepsilon)\|_{\mathcal{D}_{B,\gamma}^{2\alpha}([\frac kN T,\frac {k+1}N T])}
			\leq  C_{T,\varrho_{\alpha}(\mathbf{B}) } \left(1+ \left|X^\varepsilon_{\frac kN T} \right|_{\mathcal H_\gamma} \right).
		\end{equation}
		Therefore, for $k=0$, according to (\ref{normk}) we obtain
		\begin{equation*}
			\|X^\varepsilon,G_1(X^\varepsilon)\|_{\mathcal{D}_{B,\gamma}^{2\alpha}([0,\frac TN])}
			\leq C_{T,\varrho_{\alpha}(\mathbf{B}) } (1+|x|_{\mathcal H_\gamma}).
		\end{equation*}
		For $k=1$ we have
		\begin{align*}
			\|X^\varepsilon,G_1(X^\varepsilon)\|_{\mathcal{D}_{B,\gamma}^{2\alpha}([\frac TN,\frac {2T}N ])  }
			\leq & C_{T,\varrho_{\alpha}(\mathbf{B}) } \left(1+\left|X_{\frac TN}^\varepsilon \right|_{\mathcal H_\gamma} \right)
			\leq C_{T,\varrho_{\alpha}(\mathbf{B}) } \left( 1+\|X^\varepsilon\|_{\infty,\mathcal H_\gamma, [0,\frac TN]} \right) \nonumber\\
			\leq & C_{T,\varrho_{\alpha}(\mathbf{B}) } \left(1+\|X^\varepsilon,G_1(X^\varepsilon)\|_{\mathcal{D}_{B,\gamma}^{2\alpha}([0,\frac TN])} \right)
			\leq 2(C_{T,\varrho_{\alpha}(\mathbf{B}) })^2 (1+|x|_{\mathcal H_\gamma}).
		\end{align*}
		Using (\ref{normk}) recursively, we can conclude that
		\begin{equation}\label{xk}
			\|X^\varepsilon,G_1(X^\varepsilon)\|_{\mathcal{D}_{B,\gamma}^{2\alpha}([\frac kN T,\frac {k+1}N T])}
			\leq 2^k(C_{T,\varrho_{\alpha}(\mathbf{B})})^{k+1} (1+|x|_{\mathcal H_\gamma}), ~~~~\text{for any}~ k=0,\cdots, N-1.
		\end{equation}
		Therefore, it is easy to obtain that
		\begin{equation*}
			\|X^\varepsilon,G_1(X^\varepsilon)\|_{\mathcal{D}_{B,\gamma}^{2\alpha}([0,T])}
			\leq  \max_{k\in \{0, \cdots,N-1\}} \|X^\varepsilon,G_1(X^\varepsilon)\|_{\mathcal{D}_{B,\gamma}^{2\alpha}([\frac kN T,\frac {k+1}N T])}
			\leq 2^{N-1} (C_{T,\varrho_{\alpha}(\mathbf{B})})^{N} (1+|x|_{\mathcal H_\gamma}) .
		\end{equation*}
		Finally, using $  C {(\frac T N)}^{\alpha-\sigma} (1+\varrho_{\alpha}(\mathbf{B}))^3>1/4 $ we can bound
		$N < T (4 C (1+\varrho_{\alpha}(\mathbf{B}))^3)^{ \frac 1{\alpha-\sigma}}$.
		Thus we obtain the claim that
			\begin{equation*}
			\|X^\varepsilon,G_1(X^\varepsilon)\|_{\mathcal{D}_{B,\gamma}^{2\alpha}} \leq  C_{T,\varrho_{\alpha}(\mathbf{B}) } (1+|x|_{\mathcal H_\gamma}).
		\end{equation*}
	The proof is completed.
	\end{proof}

The following estimates for the solution to the averaged equation (\ref{aveq}) follow by arguments similar to those in the proof of Lemma \ref{slowest}. We therefore omit the details.

\begin{lemma}\label{barXest}
	Let $ x \in \mathcal{H}_\gamma $. Then the averaged equation (\ref{aveq}) admits a unique mild solution $\bar{X}_t$ such that for all $t \in [0,T]$, we have
	\begin{equation*}
	\bar{X}_t = S_t x + \int_0^t S_{t-s} \bar{F}_1(\bar{X}_s) \mathrm{d} s + \int_0^t S_{t-s} G_1(\bar{X}_s)  \mathrm{d} \mathbf B_s, ~~~~(\bar{X})_t^\prime=G_1(\bar{X}_t).	
	\end{equation*}
	Moreover, there exists a constant $ C_{T,\varrho_{\alpha}(\mathbf{B})} > 0 $ such that
	\begin{equation*}
		\|\bar{X},G_1(\bar{X})\|_{\mathcal{D}_{B,\gamma}^{2\alpha}} \leq  C_{T,\varrho_{\alpha}(\mathbf{B})} (1+|x|_{\mathcal H_\gamma}) .
	\end{equation*}
\end{lemma}

A key step in the proof of Theorem \ref{MainRes} is to obtain a bound on
$\|G_1(X^\varepsilon)-G_1(\bar{X}),DG_1(X^\varepsilon)G_1(X^\varepsilon)-DG_1(\bar{X})G_1(\bar{X})\|_{\mathcal{D}_{B,\gamma-\sigma}^{2\alpha} }$.

\begin{lemma}\label{Glem2}
Suppose that $G_1$ satisfies assumption $(\mathbf{H2})$.
Let $(X^\varepsilon,G_1(X^\varepsilon)) \in \mathcal{D}_{B}^{2\alpha}([0,T];\mathcal H_\gamma)$ be a solution to Eq. (\ref{sloweq}), and let $(\bar{X},G_1(\bar{X})) \in \mathcal{D}_{B}^{2\alpha}([0,T];\mathcal H_\gamma)$ satisfies Eq. (\ref{aveq}). Then $(G_1(X^\varepsilon),DG_1(X^\varepsilon)G_1(X^\varepsilon)) \in \mathcal{D}_{B}^{2\alpha}([0,T];\mathcal H_{\gamma-\sigma}^{d_1}), (G_1(\bar{X}),DG_1(\bar{X})G_1(\bar{X})) \in \mathcal{D}_{B}^{2\alpha}([0,T];\mathcal H_{\gamma-\sigma}^{d_1})$ and the following estimate holds:
\begin{align}\label{DGG}
	& \|G_1(X^\varepsilon)-G_1(\bar{X}),DG_1(X^\varepsilon)G_1(X^\varepsilon)-DG_1(\bar{X})G_1(\bar{X})\|_{\mathcal{D}_{B,\gamma-\sigma}^{2\alpha} }  \nonumber\\
	\leq & C_{T,\varrho_{\alpha}(\mathbf{B})} (1+|x|_{\mathcal H_\gamma})^2 \|X^\varepsilon-\bar{X},G_1(X^\varepsilon)-G_1(\bar{X}) \|_{\mathcal{D}_{B,\gamma}^{2\alpha}}.
\end{align}
\end{lemma}
\begin{proof}
	It is straightforward that
	\begin{align*}
		&\|G_1(X^\varepsilon)-G_1(\bar{X}),DG_1(X^\varepsilon)G_1(X^\varepsilon)-DG_1(\bar{X})G_1(\bar{X})\|_{\mathcal{D}_{ B,\gamma-\sigma}^{2\alpha}} \\
		=&\|G_1(X^\varepsilon)-G_1(\bar{X})\|_{\infty, \mathcal H_{\gamma-\sigma}^{d_1}}	
		+\|DG_1(X^\varepsilon)G_1(X^\varepsilon)-DG_1(\bar{X})G_1(\bar{X})\|_{\infty, \mathcal H_{\gamma-\alpha-\sigma}^{d_1 \times d_1}} \\
		&+|DG_1(X^\varepsilon)G_1(X^\varepsilon)-DG_1(\bar{X})G_1(\bar{X})|_{\alpha, \mathcal H_{\gamma-2\alpha-\sigma}^{d_1 \times d_1} }
		+|R^{G_1(X^\varepsilon)-G_1(\bar{X})}|_{\alpha, \mathcal H_{\gamma-\alpha-\sigma}^{d_1}} \\
		&+|R^{G_1(X^\varepsilon)-G_1(\bar{X})}|_{2\alpha, \mathcal H_{\gamma-2\alpha-\sigma}^{d_1}}.
	\end{align*}	
	For $\|G_1(X^\varepsilon)-G_1(\bar{X})\|_{\infty, \mathcal H_{\gamma-\sigma}^{d_1}}$, by (\ref{G1}) we have
	\begin{align}\label{Gdif1}
		\|G_1(X^\varepsilon)-G_1(\bar{X})\|_{\infty, \mathcal H_{\gamma-\sigma}^{d_1}}
		=\sup_{t\in [0,T]} |G_1(X_t^\varepsilon)-G_1(\bar{X}_t)|_{\mathcal H_{\gamma-\sigma}^{d_1}}
		\leq C\|X^\varepsilon-\bar{X}\|_{\infty, \mathcal H_{\gamma}}.
	\end{align}	
	For $\|DG_1(X^\varepsilon)G_1(X^\varepsilon)-DG_1(\bar{X})G_1(\bar{X})\|_{\infty, \mathcal H_{\gamma-\alpha-\sigma}^{d_1 \times d_1}} $, by (\ref{DG1G1}) and $\mathcal H_\gamma \hookrightarrow \mathcal H_{\gamma-\alpha}$ we have
	\begin{align}\label{Gdif2}
		\|DG_1(X^\varepsilon)G_1(X^\varepsilon)-DG_1(\bar{X})G_1(\bar{X})\|_{\infty, \mathcal H_{\gamma-\alpha-\sigma}^{d_1 \times d_1}}
		\leq C\|X^\varepsilon-\bar{X}\|_{\infty, \mathcal H_{\gamma-\alpha}}
		\leq C\|X^\varepsilon-\bar{X}\|_{\infty, \mathcal H_{\gamma}}.
	\end{align}
	Notice that	
	\begin{align*}
		&DG_1(X_t^\varepsilon)G_1(X_t^\varepsilon)-DG_1(X_s^\varepsilon)G_1(X_s^\varepsilon) \nonumber\\
		=&(DG_1(X_t^\varepsilon)-DG_1(X_s^\varepsilon)) G_1(X_t^\varepsilon)
		+DG_1(X_s^\varepsilon)(G_1(X_t^\varepsilon)-G_1(X_s^\varepsilon)) \nonumber\\
		=& \int_0^1 D^2 G_1 (X_s^\varepsilon+\theta(X_t^\varepsilon-X_s^\varepsilon)) (X_t^\varepsilon-X_s^\varepsilon) \mathrm{d} \theta G_1(X_t^\varepsilon)
		+DG_1(X_s^\varepsilon)(G_1(X_t^\varepsilon)-G_1(X_s^\varepsilon)).
	\end{align*}
	Therefore, we have
	\begin{align*}
		&\big|DG_1(X_t^\varepsilon)G_1(X_t^\varepsilon)-DG_1(X_s^\varepsilon)G_1(X_s^\varepsilon)-(DG_1(\bar{X}_t)G_1(\bar{X}_t)-DG_1(\bar{X}_s)G_1(\bar{X}_s)) \big|_{\mathcal H_{\gamma-2\alpha-\sigma}^{d_1 \times d_1}}	 \nonumber\\
		=& \bigg| \int_0^1 D^2 G_1 (X_s^\varepsilon+\theta(X_t^\varepsilon-X_s^\varepsilon)) (X_t^\varepsilon-X_s^\varepsilon) \mathrm{d} \theta G_1(X_t^\varepsilon)
		+DG_1(X_s^\varepsilon)(G_1(X_t^\varepsilon)-G_1(X_s^\varepsilon)) \nonumber\\
		&-\int_0^1 D^2 G_1(\bar{X}_s+\theta(\bar{X}_t-\bar{X}_s)) (\bar{X}_t-\bar{X}_s)  \mathrm{d} \theta G_1(\bar{X}_t)
		-DG_1(\bar{X}_s)(G_1(\bar{X}_t)-G_1(\bar{X}_s)) \bigg|_{\mathcal H_{\gamma-2\alpha-\sigma}^{d_1 \times d_1} } \nonumber\\
		\leq & \bigg| \int_0^1 \big[D^2 G_1 (X_s^\varepsilon+\theta(X_t^\varepsilon-X_s^\varepsilon)) (X_t^\varepsilon-X_s^\varepsilon)G_1(X_t^\varepsilon)\\
&~~~~~-D^2 G_1(\bar{X}_s+\theta(\bar{X}_t-\bar{X}_s)) (\bar{X}_t-\bar{X}_s) G_1(\bar{X}_t) \big] \mathrm{d} \theta \bigg|_{\mathcal H_{\gamma-2\alpha-\sigma}^{d_1 \times d_1} } \nonumber\\
		&+|DG_1(X_s^\varepsilon)(G_1(X_t^\varepsilon)-G_1(X_s^\varepsilon))-DG_1(\bar{X}_s)(G_1(\bar{X}_t)-G_1(\bar{X}_s))|_{\mathcal H_{\gamma-2\alpha-\sigma}^{d_1 \times d_1} }  \nonumber\\
		=:&I_1+I_2.
	\end{align*}
	For $I_2$, it is easy to see that
	\begin{align*}
		I_2 \leq & \left|DG_1(X_s^\varepsilon)[G_1(X_t^\varepsilon)-G_1(\bar{X}_t)-(G_1(X_s^\varepsilon)-G_1(\bar{X}_s) )]  \right|_{\mathcal H_{\gamma-2\alpha-\sigma}^{d_1 \times d_1} }  \nonumber\\
		&+\left|(DG_1(X_s^\varepsilon)-DG_1(\bar{X}_s)) (G_1(\bar{X}_t)-G_1(\bar{X}_s))  \right|_{\mathcal H_{\gamma-2\alpha-\sigma}^{d_1 \times d_1} } \nonumber\\
		=:& I_{21}+I_{22}.
	\end{align*}
	For $I_{21}$, we have
	\begin{align}\label{I21}
		I_{21} \leq &|DG_1(X_s^\varepsilon)|_{\mathcal L(\mathcal H_{\gamma-2\alpha}; \mathcal H_{\gamma-2\alpha-\sigma}^{d_1 } )}
		|G_1(X_t^\varepsilon)-G_1(\bar{X}_t)-(G_1(X_s^\varepsilon)-G_1(\bar{X}_s) )|_{\mathcal H_{\gamma-2\alpha}^{d_1}} \nonumber\\
		\leq & C |t-s|^\alpha |G_1(X^\varepsilon)-G_1(\bar{X})|_{\alpha, \mathcal H_{\gamma-2\alpha}^{d_1}}.
	\end{align}
	For $I_{22}$, by (\ref{DG1}) and $\mathcal H_\gamma \hookrightarrow \mathcal H_{\gamma-2\alpha}$ we have	
	\begin{align}\label{I22}
		I_{22}=&\left|(DG_1(X_s^\varepsilon)-DG_1(\bar{X}_s)) (G_1(\bar{X}_t)-G_1(\bar{X}_s))  \right|_{\mathcal H_{\gamma-2\alpha-\sigma}^{d_1 \times d_1}} \nonumber\\
		\leq & \left|DG_1(X_s^\varepsilon)-DG_1(\bar{X}_s) \right|_{\mathcal L(\mathcal H_{\gamma-2\alpha};\mathcal H_{\gamma-2\alpha-\sigma}^{d_1 })} |G_1(\bar{X}_t)-G_1(\bar{X}_s)|_{\mathcal H_{\gamma-2\alpha}^{d_1 }} \nonumber\\
		\leq & C |X_s^\varepsilon-\bar{X}_s|_{\mathcal H_{\gamma-2\alpha}} |G_1(\bar{X})|_{\alpha, \mathcal H_{\gamma-2\alpha}^{d_1}} |t-s|^\alpha \nonumber\\
		\leq & C \|X^\varepsilon-\bar{X}\|_{\infty,\mathcal H_\gamma} \|\bar{X},G_1(\bar{X})\|_{\mathcal D_{B,\gamma}^{2\alpha}}  |t-s|^\alpha.
	\end{align}

	For $I_1$, a straightforward computation shows that
	\begin{align*}
		I_1 \leq &  \int_0^1 \big|D^2 G_1 (X_s^\varepsilon+\theta(X_t^\varepsilon-X_s^\varepsilon)) (X_t^\varepsilon-X_s^\varepsilon)G_1(X_t^\varepsilon)\\
&~~~~~-D^2 G_1(\bar{X}_s+\theta(\bar{X}_t-\bar{X}_s)) (\bar{X}_t-\bar{X}_s) G_1(\bar{X}_t) \big|_{\mathcal H_{\gamma-2\alpha-\sigma}^{d_1 \times d_1}} \mathrm{d} \theta \nonumber\\
		\leq & \int_0^1 \left|D^2 G_1 (X_s^\varepsilon+\theta(X_t^\varepsilon-X_s^\varepsilon))(X_t^\varepsilon-X_s^\varepsilon-(\bar{X}_t-\bar{X}_s)) G_1(X_t^\varepsilon) \right|_{\mathcal H_{\gamma-2\alpha-\sigma}^{d_1 \times d_1}} \mathrm{d} \theta \nonumber\\
		&+\int_0^1 \left|[D^2 G_1 (X_s^\varepsilon+\theta(X_t^\varepsilon-X_s^\varepsilon))-D^2 G_1(\bar{X}_s+\theta(\bar{X}_t-\bar{X}_s))] (\bar{X}_t-\bar{X}_s) G_1(X_t^\varepsilon) \right|_{\mathcal H_{\gamma-2\alpha-\sigma}^{d_1 \times d_1}} \mathrm{d} \theta \nonumber\\
		&+\int_0^1 \left| D^2 G_1(\bar{X}_s+\theta(\bar{X}_t-\bar{X}_s)) (\bar{X}_t-\bar{X}_s) (G_1(X_t^\varepsilon)-G_1(\bar{X}_t)) \right|_{\mathcal H_{\gamma-2\alpha-\sigma}^{d_1 \times d_1}} \mathrm{d} \theta \nonumber\\
		=: & I_{11}+I_{12}+I_{13}.
	\end{align*}
	For $I_{11}$, by embedding $\mathcal H_{\gamma-\alpha}^{d_1} \subset \mathcal H_{\gamma-2\alpha}^{d_1}$ continuously we obtain
	\begin{align}\label{I11}
		I_{11} \leq & \int_0^1 |D^2 G_1 (X_s^\varepsilon+\theta(X_t^\varepsilon-X_s^\varepsilon))|_{\mathcal L (\mathcal H_{\gamma-2\alpha}; \mathcal L (\mathcal H_{\gamma-2\alpha}; \mathcal H_{\gamma-2\alpha-\sigma}^{d_1 }))} \nonumber\\
		&~~~\cdot
		|G_1(X_t^\varepsilon)|_{\mathcal H_{\gamma-2\alpha}^{d_1}} |X_t^\varepsilon-X_s^\varepsilon-(\bar{X}_t-\bar{X}_s)|_{\mathcal H_{\gamma-2\alpha}} \mathrm{d} \theta \nonumber\\
		\leq & C \int_0^1 \|G_1(X^\varepsilon)\|_{\infty,\mathcal H_{\gamma-\alpha}^{d_1}} |X_t^\varepsilon-\bar{X}_t-(X_s^\varepsilon-\bar{X}_s)|_{\mathcal H_{\gamma-2\alpha}} \mathrm{d} \theta \nonumber\\
		\leq & C  \|X^\varepsilon, G_1(X^\varepsilon)\|_{\mathcal D_{B,\gamma}^{2\alpha}} |X^\varepsilon-\bar{X}|_{\alpha, \mathcal H_{\gamma-2\alpha}} |t-s|^\alpha  .
	\end{align}	
	For $I_{12}$, using (\ref{Y}), (\ref{D2G1}), $\mathcal H_{\gamma} \hookrightarrow \mathcal H_{\gamma-2\alpha}$ and the fact that $G_1$ is bounded yields
	\begin{align}\label{I12}
		I_{12} \leq & \int_0^1 |D^2 G_1 (X_s^\varepsilon+\theta(X_t^\varepsilon-X_s^\varepsilon))-D^2 G_1(\bar{X}_s+\theta(\bar{X}_t-\bar{X}_s))|_{\mathcal L (\mathcal H_{\gamma-2\alpha}; \mathcal L (\mathcal H_{\gamma-2\alpha}; \mathcal H_{\gamma-2\alpha-\sigma}^{d_1}))} \nonumber\\
		&~~~~ \cdot
		|G_1(X_t^\varepsilon)|_{\mathcal H_{\gamma-2\alpha}^{d_1}} |\bar{X}_t-\bar{X}_s|_{\mathcal H_{\gamma-2\alpha}} \mathrm{d} \theta \nonumber\\
		\leq & C \int_0^1 |(1-\theta)(X_s^\varepsilon-\bar{X}_s)+\theta (X_t^\varepsilon-\bar{X}_t)|_{\mathcal H_{\gamma-2\alpha}}   |\bar{X}_t-\bar{X}_s|_{\mathcal H_{\gamma-2\alpha}} \mathrm{d} \theta \nonumber\\
		\leq & C \int_0^1  \|X^\varepsilon-\bar{X}\|_{\infty, \mathcal H_\gamma}
		|\bar{X}|_{\alpha, \mathcal H_{\gamma-2\alpha}} |t-s|^\alpha \mathrm{d} \theta  \nonumber\\
		\leq & C \|X^\varepsilon-\bar{X}\|_{\infty, \mathcal H_\gamma} (\|G_1(\bar{X})\|_{\infty,\mathcal H_{\gamma-2\alpha}^{d_1}}|B|_\alpha+|R^{\bar{X}} |_{\alpha, \mathcal H_{\gamma-2\alpha}})  |t-s|^\alpha  .
	\end{align}	
	For $I_{13}$, according to (\ref{Y}) and $\mathcal H_{\gamma-\alpha}^{d_1} \hookrightarrow \mathcal H_{\gamma-2\alpha}^{d_1}$ we get
	\begin{align}\label{I13}
		I_{13}	
		\leq & C \|G_1(X^\varepsilon)-G_1(\bar{X})\|_{\infty,\mathcal H_{\gamma-\alpha}^{d_1}} |\bar{X}|_{\alpha,\mathcal H_{\gamma-2\alpha}} |t-s|^\alpha \nonumber\\
		\leq & C \|G_1(X^\varepsilon)-G_1(\bar{X})\|_{\infty,\mathcal H_{\gamma-\alpha}^ {d_1}} (\|G_1(\bar{X})\|_{\infty,\mathcal H_{\gamma-2\alpha}^{d_1}}|B|_\alpha+|R^{\bar{X}} |_{\alpha, \mathcal H_{\gamma-2\alpha}})  |t-s|^\alpha  .
	\end{align}

	Combining estimates (\ref{I21})--(\ref{I13}), as well as using (\ref{Y}) and $\mathcal H_{\gamma-\alpha} \hookrightarrow \mathcal H_{\gamma-2\alpha}$ we have
	\begin{align}\label{G3}
		&|DG_1(X^\varepsilon)G_1(X^\varepsilon)-DG_1(\bar{X})G_1(\bar{X})|_{\alpha, \mathcal H_{\gamma-2\alpha-\sigma}^{d_1 \times d_1}}  \nonumber\\
		\leq & C(1+|B|_\alpha) \left(1+\|X^\varepsilon,G_1(X^\varepsilon)\|_{\mathcal{D}_{ B,\gamma}^{2\alpha}} +\|\bar{X},G_1(\bar{X})\|_{\mathcal{D}_{ B,\gamma}^{2\alpha}} \right) \|X^\varepsilon-\bar{X},G_1(X^\varepsilon)-G_1(\bar{X})\|_{\mathcal{D}_{ B,\gamma}^{2\alpha}}  .
	\end{align}

	Below we estimate the remainder $|R^{G_1(X^\varepsilon)-G_1(\bar{X})} |_{\alpha, \mathcal H_{\gamma-\alpha-\sigma}^{d_1} }$.
	Using (\ref{remainder}) twice we have
	\begin{align*}
		R_{s,t}^{G_1(X^\varepsilon)-G_1(\bar{X})}
		=&G_1(X_t^\varepsilon)-G_1(\bar{X}_t)-(G_1(X_s^\varepsilon)-G_1(\bar{X}_s))  \nonumber\\ &-(DG_1(X_s^\varepsilon)G_1(X_s^\varepsilon)-DG_1(\bar{X}_s)G_1(\bar{X}_s)) B_{s,t} \nonumber\\
		=& G_1(X_t^\varepsilon)-G_1(X_s^\varepsilon)-(G_1(\bar{X}_t)-G_1(\bar{X}_s)) \nonumber\\
		&-(DG_1(X_s^\varepsilon)G_1(X_s^\varepsilon)-DG_1(\bar{X}_s)G_1(\bar{X}_s)) B_{s,t} \nonumber\\
		=& \bigg\{\int_0^1 DG_1(X_s^\varepsilon+\theta (X_t^\varepsilon-X_s^\varepsilon))-DG_1(X_s^\varepsilon) \mathrm{d} \theta G_1(X_s^\varepsilon)  B_{s,t} \nonumber\\
		&-\int_0^1 DG_1(\bar{X}_s+\theta (\bar{X}_t-\bar{X}_s))-DG_1(\bar{X}_s) \mathrm{d} \theta G_1(\bar{X}_s)  B_{s,t} \bigg\}\nonumber\\
		&+\bigg\{ \int_0^1 DG_1(X_s^\varepsilon+\theta (X_t^\varepsilon-X_s^\varepsilon)) \mathrm{d} \theta R_{s,t}^{X^\varepsilon}
		-\int_0^1 DG_1(\bar{X}_s+\theta (\bar{X}_t-\bar{X}_s)) \mathrm{d} \theta R_{s,t}^{\bar{X}} \bigg\} \nonumber\\
		=: & J_1 + J_2.
	\end{align*}	
	For $J_1$, we have
	\begin{align*}
		|J_1 |_{\mathcal H_{\gamma-\alpha-\sigma}^{d_1} }
		\leq & \bigg| \int_0^1 DG_1(X_s^\varepsilon+\theta (X_t^\varepsilon-X_s^\varepsilon))-DG_1(X_s^\varepsilon)   \nonumber\\
		&-\{ DG_1(\bar{X}_s+\theta (\bar{X}_t-\bar{X}_s))-DG_1(\bar{X}_s)\} \mathrm{d} \theta G_1(X_s^\varepsilon)  B_{s,t} \bigg|_{\mathcal H_{\gamma-\alpha-\sigma}^{d_1} }  \nonumber\\
		&+ \bigg|\int_0^1 DG_1(\bar{X}_s+\theta (\bar{X}_t-\bar{X}_s))-DG_1(\bar{X}_s) \mathrm{d} \theta (G_1(X_s^\varepsilon)-G_1(\bar{X}_s)) B_{s,t}  \bigg|_{\mathcal H_{\gamma-\alpha-\sigma}^{d_1} }  \nonumber\\
		=: & J_{11}+J_{12}.
	\end{align*}	
	According to (\ref{DG1}) and $\mathcal H_\gamma \hookrightarrow \mathcal H_{\gamma-\alpha}$ we have
	\begin{align}\label{J11}
		J_{11} \leq &\bigg| \int_0^1 DG_1(X_s^\varepsilon+\theta (X_t^\varepsilon-X_s^\varepsilon))-DG_1(X_s^\varepsilon)   \nonumber\\
		&~~~-\{ DG_1(\bar{X}_s+\theta (\bar{X}_t-\bar{X}_s))-DG_1(\bar{X}_s)\} \mathrm{d} \theta \bigg|_{\mathcal L(\mathcal H_{\gamma-\alpha};\mathcal H_{\gamma-\alpha-\sigma}^{d_1})} |G_1(X_s^\varepsilon)|_{\mathcal H_{\gamma-\alpha}^{d_1} }  |B_{s,t}|   \nonumber\\	
		\leq &  \int_0^1 | DG_1(X_s^\varepsilon+\theta (X_t^\varepsilon-X_s^\varepsilon))-DG_1(\bar{X}_s+\theta (\bar{X}_t-\bar{X}_s)) |_{\mathcal L(\mathcal H_{\gamma-\alpha};\mathcal H_{\gamma-\alpha-\sigma}^{d_1}) }  \nonumber\\
		&+|DG_1(X_s^\varepsilon)-DG_1(\bar{X}_s) |_{\mathcal L(\mathcal H_{\gamma-\alpha};\mathcal H_{\gamma-\alpha-\sigma}^{d_1}) } \mathrm{d} \theta
		\|G_1(X^\varepsilon)\|_{\infty, \mathcal H_{\gamma-\alpha}^{d_1}}  |B|_{\alpha}  |t-s|^\alpha   \nonumber\\	
		\leq & C \int_0^1 \theta |X_t^\varepsilon-\bar{X}_t|_{\mathcal H_{\gamma-\alpha}}
		+(2-\theta) |X_s^\varepsilon-\bar{X}_s|_{\mathcal H_{\gamma-\alpha}} \mathrm{d} \theta  \|X^\varepsilon, G_1(X^\varepsilon)\|_{\mathcal{D}_{ B,\gamma}^{2\alpha}}
		|B|_{\alpha}  |t-s|^\alpha   \nonumber\\
		\leq & C \|X^\varepsilon-\bar{X}\|_{\infty, \mathcal H_{\gamma}}
		\|X^\varepsilon, G_1(X^\varepsilon)\|_{\mathcal{D}_{ B,\gamma}^{2\alpha}}
		|B|_{\alpha}  |t-s|^\alpha   .
	\end{align}
	Following the condition $(\mathbf{H2})$ we obtain
	\begin{align}\label{J12}
	J_{12} \leq & \int_0^1 |DG_1(\bar{X}_s+\theta (\bar{X}_t-\bar{X}_s))|_{\mathcal L(\mathcal H_{\gamma-\alpha};\mathcal H_{\gamma-\alpha-\sigma}^{d_1}) } +|DG_1(\bar{X}_s)|_{\mathcal L(\mathcal H_{\gamma-\alpha};\mathcal H_{\gamma-\alpha-\sigma}^{d_1}) } \mathrm{d} \theta \nonumber\\
		& \cdot |G_1(X_s^\varepsilon)-G_1(\bar{X}_s)|_{\mathcal H_{\gamma-\alpha}^{d_1}} |B_{s,t}|  \nonumber\\
		\leq & C \|G_1(X^\varepsilon)-G_1(\bar{X})\|_{\infty, \mathcal H_{\gamma-\alpha}^{d_1} } |B|_\alpha |t-s|^\alpha.
	\end{align}
	Using (\ref{DG1}), $\mathcal H_\gamma \hookrightarrow \mathcal H_{\gamma-\alpha}$ and the boundedness of $DG_1$ we have
	\begin{align}\label{remJ2}
		|J_2 |_{\mathcal H_{\gamma-\alpha-\sigma}^{d_1} }
		\leq & \left|\int_0^1 DG_1(X_s^\varepsilon+\theta (X_t^\varepsilon-X_s^\varepsilon))-DG_1(\bar{X}_s+\theta (\bar{X}_t-\bar{X}_s)) \mathrm{d} \theta R_{s,t}^{X^\varepsilon} \right|_{\mathcal H_{\gamma-\alpha-\sigma}^{d_1} } \nonumber\\
		&+\left|\int_0^1 DG_1(\bar{X}_s+\theta (\bar{X}_t-\bar{X}_s)) \mathrm{d} \theta \left(R_{s,t}^{X^\varepsilon}-R_{s,t}^{\bar{X}} \right) \right|_{\mathcal H_{\gamma-\alpha-\sigma}^{d_1} } \nonumber\\
		\leq & \int_0^1 | DG_1(X_s^\varepsilon+\theta (X_t^\varepsilon-X_s^\varepsilon))-DG_1(\bar{X}_s+\theta (\bar{X}_t-\bar{X}_s)) |_{\mathcal L(\mathcal H_{\gamma-\alpha}; \mathcal H_{\gamma-\alpha-\sigma}^{d_1} )} \mathrm{d} \theta \left|R_{s,t}^{X^\varepsilon} \right|_{\mathcal H_{\gamma-\alpha}}  \nonumber\\
		&+\int_0^1 | DG_1(\bar{X}_s+\theta (\bar{X}_t-\bar{X}_s)) |_{\mathcal L(\mathcal H_{\gamma-\alpha}; \mathcal H_{\gamma-\alpha-\sigma}^{d_1} )}  \mathrm{d} \theta \left|R_{s,t}^{X^\varepsilon}-R_{s,t}^{\bar{X}} \right|_{\mathcal H_{\gamma-\alpha}} \nonumber\\
		\leq & C \|X^\varepsilon-\bar{X}\|_{\infty, \mathcal H_{\gamma}}  |R^{X^\varepsilon}|_{\alpha, \mathcal H_{\gamma-\alpha}} |t-s|^\alpha
		+C |R^{X^\varepsilon-\bar{X}} |_{\alpha, \mathcal H_{\gamma-\alpha}} |t-s|^\alpha \nonumber\\
		\leq & C \left(\|X^\varepsilon-\bar{X}\|_{\infty, \mathcal H_{\gamma}}
		\|X^\varepsilon, G_1(X^\varepsilon)\|_{\mathcal{D}_{ B,\gamma}^{2\alpha}} + |R^{X^\varepsilon-\bar{X}} |_{\alpha, \mathcal H_{\gamma-\alpha}} \right) |t-s|^\alpha .
	\end{align}
	Collecting the bounds established in (\ref{J11})--(\ref{remJ2}) yields the estimate
	\begin{align}\label{G1rem1}
		|R^{G_1(X^\varepsilon)-G_1(\bar{X})} |_{\alpha, \mathcal H_{\gamma-\alpha-\sigma}^{d_1}} \leq C (1+|B|_\alpha) (1+\|X^\varepsilon, G_1(X^\varepsilon)\|_{\mathcal{D}_{ B,\gamma}^{2\alpha}} ) \|X^\varepsilon-\bar{X},G_1(X^\varepsilon)-G_1(\bar{X})\|_{\mathcal{D}_{ B,\gamma}^{2\alpha}}.
	\end{align}

	For $|R^{G_1(X^\varepsilon)-G_1(\bar{X})} |_{2\alpha, \mathcal H_{\gamma-2\alpha-\sigma}^{d_1 } }$, by (\ref{remainder}) we have	
	\begin{align*}
		\left|R^{G_1(X^\varepsilon)-G_1(\bar{X})}_{s,t} \right|_{\mathcal H_{\gamma-2\alpha-\sigma}^{d_1 }}
		=&\big|G_1(X_t^\varepsilon)-G_1(\bar{X}_t)-G_1(X_s^\varepsilon)+G_1(\bar{X}_s)  \nonumber\\ &-(DG_1(X_s^\varepsilon)G_1(X_s^\varepsilon)-DG_1(\bar{X}_s)G_1(\bar{X}_s)) B_{s,t} \big|_{\mathcal H_{\gamma-2\alpha-\sigma}^{d_1 }}.
	\end{align*}	
	Applying Taylor's formula we have
	\begin{equation}\label{difG1}
		G_1(X_t^\varepsilon)-G_1(X_s^\varepsilon)=DG_1(X_s^\varepsilon)X_{s,t}^\varepsilon+\int_0^1 D^2 G_1(X_s^\varepsilon+\theta X_{s,t}^\varepsilon) \langle X_{s,t}^\varepsilon, X_{s,t}^\varepsilon \rangle_{\mathcal H_{\gamma-2\alpha} }  (1-\theta) \mathrm{d} \theta.	
	\end{equation}
	Therefore, by (\ref{difG1}) we have
	\begin{align*}
		\left|R^{G_1(X^\varepsilon)-G_1(\bar{X})}_{s,t} \right|_{\mathcal H_{\gamma-2\alpha-\sigma}^{d_1 }}
		\leq & |DG_1(X_s^\varepsilon)(X_{s,t}^\varepsilon -G_1(X_s^\varepsilon) B_{s,t})-DG_1(\bar{X}_s)(\bar{X}_{s,t}-G_1(\bar{X}_s) B_{s,t}) |_{\mathcal H_{\gamma-2\alpha-\sigma}^{d_1 }}  \nonumber\\
		&+\int_0^1 |D^2 G_1(X_s^\varepsilon+\theta X_{s,t}^\varepsilon)  \langle X_{s,t}^\varepsilon, X_{s,t}^\varepsilon \rangle_{\mathcal H_{\gamma-2\alpha} }   \nonumber\\
		&~~~~~~~~~-D^2 G_1(\bar{X}_s+\theta \bar{X}_{s,t})  \langle \bar{X}_{s,t}, \bar{X}_{s,t} \rangle_{\mathcal H_{\gamma-2\alpha} }  |_{\mathcal H_{\gamma-2\alpha-\sigma}^{d_1 }} (1-\theta) \mathrm{d} \theta  \nonumber\\
		=: & J_3+J_4.
	\end{align*}
	According to the condition $(\mathbf{H2})$, (\ref{remainder}) and $\mathcal H_\gamma \hookrightarrow \mathcal H_{\gamma-2\alpha}$, we have
	\begin{align}\label{J3}
		J_3=&|DG_1(X_s^\varepsilon) R_{s,t}^{X^\varepsilon}-DG_1(\bar{X}_s) R_{s,t}^{\bar{X}} |_{\mathcal H_{\gamma-2\alpha-\sigma}^{d_1 }}  \nonumber\\
		\leq & |DG_1(X_s^\varepsilon) (R_{s,t}^{X^\varepsilon}-R_{s,t}^{\bar{X}}) |_{\mathcal H_{\gamma-2\alpha-\sigma}^{d_1 }}
		+| (DG_1(X_s^\varepsilon)-DG_1(\bar{X}_s)) R_{s,t}^{\bar{X}} |_{\mathcal H_{\gamma-2\alpha-\sigma}^{d_1 }} \nonumber\\
		\leq & |DG_1(X_s^\varepsilon)|_{\mathcal L(\mathcal H_{\gamma-2\alpha}; \mathcal H_{\gamma-2\alpha-\sigma}^{d_1 })} |R_{s,t}^{X^\varepsilon}-R_{s,t}^{\bar{X}}|_{\mathcal H_{\gamma-2\alpha}}  \nonumber\\
		&+|DG_1(X_s^\varepsilon)-DG_1(\bar{X}_s)|_{\mathcal L (\mathcal H_{\gamma-2\alpha}; \mathcal H_{\gamma-2\alpha-\sigma}^{d_1 })} |R_{s,t}^{\bar{X}}|_{\mathcal H_{\gamma-2\alpha}} \nonumber\\
		\leq & C |R^{X^\varepsilon-\bar{X}}|_{2\alpha, \mathcal H_{\gamma-2\alpha}} |t-s|^{2\alpha} +C |X_s^\varepsilon-\bar{X}_s|_{\mathcal H_{\gamma-2\alpha}} |R^{\bar{X}} |_{2\alpha, \mathcal H_{\gamma-2\alpha}} |t-s|^{2\alpha} \nonumber\\
		\leq & C \left(|R^{X^\varepsilon-\bar{X}}|_{2\alpha, \mathcal H_{\gamma-2\alpha}}+\|X^\varepsilon-\bar{X}\|_{\infty, \mathcal H_\gamma} \|\bar{X}, G_1(\bar{X})\|_{\mathcal{D}_{ B,\gamma}^{2\alpha}}  \right) |t-s|^{2\alpha}  \nonumber\\
		\leq & C \left( 1+\|\bar{X}, G_1(\bar{X})\|_{\mathcal{D}_{ B,\gamma}^{2\alpha}} \right) \|X^\varepsilon-\bar{X},G_1(X^\varepsilon)-G_1(\bar{X})\|_{\mathcal{D}_{ B,\gamma}^{2\alpha}} |t-s|^{2\alpha}.
	\end{align}
	For $J_4$, we have
	\begin{align*}
		J_4 \leq & \int_0^1 |(D^2 G_1(X_s^\varepsilon+\theta X_{s,t}^\varepsilon)-D^2 G_1(\bar{X}_s+\theta \bar{X}_{s,t}))  \langle X_{s,t}^\varepsilon, X_{s,t}^\varepsilon \rangle_{\mathcal H_{\gamma-2\alpha} }   |_{\mathcal H_{\gamma-2\alpha-\sigma}^{d_1 }} \mathrm{d} \theta \nonumber\\
		&+\int_0^1 |D^2 G_1(\bar{X}_s+\theta \bar{X}_{s,t}) ( \langle X_{s,t}^\varepsilon, X_{s,t}^\varepsilon \rangle_{\mathcal H_{\gamma-2\alpha} }-\langle \bar{X}_{s,t}, \bar{X}_{s,t} \rangle_{\mathcal H_{\gamma-2\alpha} } ) |_{\mathcal H_{\gamma-2\alpha-\sigma}^{d_1 }} \mathrm{d} \theta \nonumber\\
		=:& J_{41}+J_{42}.
	\end{align*}
	For $J_{41}$, by (\ref{Y}), (\ref{D2G1}) and $\mathcal H_\gamma \hookrightarrow \mathcal H_{\gamma-2\alpha}$ we have
	\begin{align}\label{J41}
		J_{41} 	\leq & C  \int_0^1 |\theta (X_t^\varepsilon-\bar{X}_t)+(1-\theta) (X_s^\varepsilon-\bar{X}_s)|_{\mathcal H_{\gamma-2\alpha}} |X_{s,t}^\varepsilon|_{\mathcal H_{\gamma-2\alpha}}^2 \mathrm{d} \theta \nonumber\\
		\leq  & C \|X^\varepsilon-\bar{X}\|_{\infty, \mathcal H_\gamma} |X_{s,t}^\varepsilon|_{\mathcal H_{\gamma-2\alpha}}^2
		\leq C \|X^\varepsilon-\bar{X}\|_{\infty, \mathcal H_\gamma} |X^\varepsilon|^2_{\alpha , \mathcal H_{\gamma-2\alpha}} |t-s|^{2\alpha}  \nonumber\\
		\leq & C \|X^\varepsilon-\bar{X}\|_{\infty, \mathcal H_\gamma}  (\|G_1(X^\varepsilon)\|_{\infty,\mathcal H_{\gamma-2\alpha}^{d_1}} |B|_\alpha+|R^{X^\varepsilon}|_{\alpha,\mathcal H_{\gamma-2\alpha}})^2  |t-s|^{2\alpha} \nonumber\\
		\leq & C (1+|B|_\alpha)^2 \|X^\varepsilon,G_1(X^\varepsilon)\|_{\mathcal{D}_{ B,\gamma}^{2\alpha}}^2 \|X^\varepsilon-\bar{X},G_1(X^\varepsilon)-G_1(\bar{X})\|_{\mathcal{D}_{ B,\gamma}^{2\alpha}} |t-s|^{2\alpha} ,
	\end{align}	
	where we used $\mathcal H_{\gamma-\alpha} \hookrightarrow \mathcal H_{\gamma-2\alpha}$ for the last inequality.
	For $J_{42}$, using (\ref{Y}) and $\mathcal H_{\gamma-\alpha} \hookrightarrow \mathcal H_{\gamma-2\alpha}$ we get
	\begin{align}\label{J42}
		J_{42} \leq & C | \langle X_{s,t}^\varepsilon-\bar{X}_{s,t}, X_{s,t}^\varepsilon \rangle_{\mathcal H_{\gamma-2\alpha} }|+|\langle \bar{X}_{s,t}, X_{s,t}^\varepsilon-\bar{X}_{s,t} \rangle_{\mathcal H_{\gamma-2\alpha} } | \nonumber\\
		\leq & C |X^\varepsilon-\bar{X}|_{\alpha, \mathcal H_{\gamma-2\alpha}} (|X^\varepsilon|_{\alpha, \mathcal H_{\gamma-2\alpha}}+|\bar{X}|_{\alpha, \mathcal H_{\gamma-2\alpha}} ) |t-s|^{2\alpha}  \nonumber\\
		\leq & C(1+|B|_\alpha)^2 \|X^\varepsilon-\bar{X},G_1(X^\varepsilon)-G_1(\bar{X})\|_{\mathcal{D}_{ B,\gamma}^{2\alpha}}  \left(\|X^\varepsilon,G_1(X^\varepsilon)\|_{\mathcal{D}_{ B,\gamma}^{2\alpha}} +\|\bar{X},G_1(\bar{X})\|_{\mathcal{D}_{ B,\gamma}^{2\alpha}} \right) |t-s|^{2\alpha}.
	\end{align}
	Combining the estimates in (\ref{J3})--(\ref{J42}), it follows that
	\begin{align}\label{G1rem2}
		|R^{G_1(X^\varepsilon)-G_1(\bar{X})} |_{2\alpha, \mathcal H_{\gamma-2\alpha-\sigma}^{d_1 } }
		\leq & C (1+|B|_\alpha)^2 \|X^\varepsilon-\bar{X},G_1(X^\varepsilon)-G_1(\bar{X})\|_{\mathcal{D}_{ B,\gamma}^{2\alpha}}  \nonumber\\ & \cdot \left(1+\|\bar{X},G_1(\bar{X})\|_{\mathcal{D}_{ B,\gamma}^{2\alpha}}+\|X^\varepsilon,G_1(X^\varepsilon)\|_{\mathcal{D}_{ B,\gamma}^{2\alpha}} +\|X^\varepsilon,G_1(X^\varepsilon)\|_{\mathcal{D}_{ B,\gamma}^{2\alpha}}^2 \right).
	\end{align}
	Combined with the above calculations (\ref{Gdif1}), (\ref{Gdif2}), (\ref{G3}), (\ref{G1rem1}) and (\ref{G1rem2}), we have
	\begin{align*}
	&\|G_1(X^\varepsilon)-G_1(\bar{X}),DG_1(X^\varepsilon)G_1(X^\varepsilon)-DG_1(\bar{X})G_1(\bar{X})\|_{\mathcal{D}_{B,\gamma-\sigma}^{2\alpha} }  \nonumber\\
	\lesssim & (1+|B|_\alpha)^2 \left(1+\|\bar{X},G_1(\bar{X})\|_{\mathcal{D}_{B,\gamma}^{2\alpha}}+\|X^\varepsilon,G_1(X^\varepsilon)\|_{\mathcal{D}_{B,\gamma}^{2\alpha}} +\|X^\varepsilon,G_1(X^\varepsilon)\|_{\mathcal{D}_{ B,\gamma}^{2\alpha}}^2 \right) \nonumber\\
	& \cdot \|X^\varepsilon-\bar{X},G_1(X^\varepsilon)-G_1(\bar{X}) \|_{\mathcal{D}_{B,\gamma}^{2\alpha}}.	
	\end{align*}
The conclusion is established by using Lemmas \ref{slowest} and \ref{barXest}.	
	\end{proof}

\subsection{Probabilistic aspects} \label{proba}

From now on, we work in a stochastic framework.
More precisely, we study the slow--fast system (\ref{eq1}) of rough PDEs from a probabilistic perspective.
To do this, we first introduce the random rough path $\mathbf M$, where $\mathbf M$ is as defined in (\ref{mixRP}).

As before, we fix $\alpha_0 \in (1/3,1/2]$ and let $(\Omega, \mathcal{F}, \{\mathcal{F}_t\}_{t \in [0, T]},\mathbb{P})$ be a filtered probability space satisfying the usual conditions.
Let
$\mathbf B=\{(B_{s,t}, \mathbb B_{s,t})\}_{0\leq s\leq t\leq T}$ be an $\mathscr{C}^{\alpha}([0,T];\mathbf{R}^{d_1})$-valued random variable defined on $(\Omega, \mathcal{F}, \{\mathcal{F}_t\}_{t \in [0, T]}, \mathbb{P})$ for every $\alpha \in (1/3, \alpha_0)$,
and let $t\mapsto (B_{0,t}, \mathbb B_{0,t})$ is $\{\mathcal{F}_t\}$-adapted.
Let $\{W_t\}_{t\in [0,T]}$ be a standard $\mathbf R^{d_2}$-valued $\{\mathcal{F}_t\}$-Brownian motion, independent of $\mathbf B$.
The definition of the random rough path $\mathbf M$ is given below.

\begin{definition}\label{Xi}
Define the ``lift" of Brownian motion $W$ by
$$\mathbb W_{s,t}^{\text{It\^{o}}}:= \int_s^t  W_{s,r} \otimes \mathrm{d} W_r,$$
where the integration is understood in the sense of It\^{o} stochastic integral.
Then $\mathbf W^{\text{It\^{o}} }=(W,\mathbb W^{\text{It\^{o}}})$ is an It\^{o}-type Brownian rough path
and for any $\alpha \in (1/3, 1/2)$, $\mathbf W^{\text{It\^{o}} }=(W,\mathbb W^{\text{It\^{o}}}) \in \mathscr{C}^\alpha ([0,T]; \mathbb R^{d_2})$, $\mathbb P$-a.s..
Furthermore, we set
$$I[B,W]_{s,t}:=\int_s^t  B_{s,r} \otimes \mathrm{d} W_r$$
as an It\^{o} stochastic integral, we can then define
$$I[W,B]_{s,t}:= W_{s,t} \otimes  B_{s,t} -\int_s^t \mathrm{d} W_r \otimes  B_{s,r} $$
by imposing integration by parts.
All components of the random rough path $\mathbf M=(M, \mathbb M)$ have hereby been defined.
\end{definition}

By Kolmogorov's continuity criterion we obtain the following lemma (see \cite[Lemma 4.6]{In22}).

\begin{lemma} \label{mixrandom}
Assume that condition $(\mathbf{H6})$ holds. For any $\alpha \in (1/3, \alpha_0)$ and $\mathbf M$ as introduced in Definition $\ref{Xi}$, we have $\mathbf M(\omega)=\mathbf M=(M, \mathbb M) \in \mathscr{C}^\alpha ([0,T]; \mathbb R^{d_1+d_2})$, $\mathbb P$-a.s..
Furthermore, for any $p \in [1,\infty)$, we have $\mathbb E \big[\interleave \mathbf M \interleave_\alpha^p \big] <\infty$.
\end{lemma}

To prove Theorem \ref{MainRes}, exponential ergodicity is indispensable.
We therefore present the following lemma,
which asserts that the fast component $Y^\varepsilon$, defined with respect to an It\^{o}-type Brownian rough path, is equivalent to an It\^{o} stochastic PDE.
The proof is similar to that in \cite[Lemma 3.3]{LLPX25}, and is therefore omitted.
\begin{lemma}\label{RSPDE}
	Assume that $(\mathbf{H6})$ holds. Then, for every fixed $\varepsilon \in (0,1]$, $Y^\varepsilon$ satisfies the following It\^{o} stochastic PDE:
	\begin{equation}\label{eqY}
		Y_t^\varepsilon= S_{t/\varepsilon} y+ \frac 1\varepsilon \int_0^{t} S_{(t-s)/\varepsilon} F_2(X_s^\varepsilon, Y_s^\varepsilon) \mathrm{d} s
		+\frac 1{\sqrt{\varepsilon}} \int_0^{t} S_{(t-s)/\varepsilon} G_2(X_s^\varepsilon, Y_s^\varepsilon) \mathrm{d} W_s, ~~~t \in [0,T].
	\end{equation}
\end{lemma}

The following an energy estimate on Eq. (\ref{eqY}) can be proved as in \cite[Lemma 4.1]{LLPX25}. We omit the details here.
\begin{lemma}\label{avlemma1}
	For every $(x,y) \in \mathcal H_\gamma \times \mathcal H_\gamma$ and $\varepsilon >0$ small enough, there exists a constant $C_T> 0$ such that
	\begin{equation}\label{Yest}
		\sup_{t\in[0,T]}\mathbb{E}|Y_{t}^{\varepsilon}|_{\mathcal H_\gamma }^4
		\leq  C_T(1+| x |_{\mathcal H_\gamma }^4+| y |_{\mathcal H_\gamma }^4).
	\end{equation}
\end{lemma}

\subsection{Frozen equation}

Next, in order to define the coefficient $\bar{F}_1$ in the averaged equation (\ref{aveq}), we consider the frozen equation associated with Eq. (\ref{eqY}) for a fixed slow component $x\in \mathcal H_\gamma$.
More precisely, we will present the existence and uniqueness of invariant probability measures as well as exponential ergodicity with respect to the frozen equation.

For a fixed slow component $x \in \mathcal H_\gamma$, the frozen equation corresponding to Eq. (\ref{eqY}) is given by
\begin{equation*}
	\left\{ \begin{aligned}
		\mathrm{d} Y_t&=\big[L Y_t+F_2(x,Y_t)\big] \mathrm{d} t+G_2(x,Y_t) \mathrm{d} \widetilde{W}_t,\\
		Y_0&= y \in \mathcal H_\gamma ,
	\end{aligned} \right.
\end{equation*}
where $\{\widetilde{W}_t\}_{t\geq0}$ is a $d_2$-dimensional Brownian motion defined on another probability space $(\widetilde{\Omega} ,\widetilde{\mathcal F},\{\widetilde{\mathcal F}_t\}_{t\geq0},\widetilde{\mathbb P})$.

It is known that the frozen equation admits a unique strong solution, denoted by $\{Y_t^{x,y}\}_{t\geq 0}$, which is a time homogeneous Markov process.
Let $\{P^x_t\}_{t\geq0}$ be a Markov transition semigroup of process $\{Y^{x,y}_t\}_{t\geq 0}$, i.e., for any bounded measurable function $f:\mathcal{H}_\gamma \rightarrow \mathcal{H}_\gamma $ we have
$$P^x_t f(y):=\widetilde{\mathbb E} [f(Y^{x,y}_t)],~~~~ y \in \mathcal{H}_\gamma,~t>0.$$
Here, $\widetilde{\mathbb E}$ denotes the expectation operator corresponding to $\widetilde{\mathbb P}$.
According to \cite[Theorem 2.2]{HLL20}, for each fixed $x \in \mathcal H_\gamma$, the transition semigroup $\{P^x_t\}_{t\geq0}$ admits a unique invariant probability measure $\mu^x$.

Following the similar argument as in the proof of \cite[Lemma 4.5]{LLPX25},
we can derive the following estimates for the solution to the frozen equation.
\begin{lemma}\label{lem3.4}
	For any given $x,x_1,x_2,y\in \mathcal{H}_\gamma$, there exists a constant $C>0$ such that
	\begin{align}
		&\sup_{t\geq0}\widetilde{\mathbb E}|Y_t^{x_1,y}-Y_t^{x_2,y}|_{\mathcal{H}_\gamma}^2\leq C|x_1-x_2|_{\mathcal{H}_\gamma }^2, \label{frozendifference}\\
		&\sup_{t\geq0}\widetilde{\mathbb E}|Y^{x,y}_t|_{\mathcal{H}_\gamma}^2\leq C(1+|x|_{\mathcal{H}_\gamma }^2+|y|_{\mathcal{H}_\gamma}^2).\label{frozenest}
	\end{align}
\end{lemma}

We now provide the following exponential ergodicity, the proof of which is similar to \cite[Proposition 4.2]{HLL22}. We omit the details here.

\begin{lemma}\label{ergo}
For any $(x,y) \in \mathcal H_\gamma \times \mathcal H_\gamma$, there exist constants $\lambda_1-L_{F_2}-L_{G_2}^2>0$ and $C> 0$ such that
$$|\widetilde{\mathbb E}F_1(x, Y_t^{x,y})-\bar{F}_1(x)|_{\mathcal H_{\gamma-\alpha} }\leq C e^{ -(\lambda_1-L_{F_2}-L_{G_2}^2) t} (1+| x |_{\mathcal H_\gamma }+| y |_{\mathcal H_\gamma }),$$
where $\bar{F}_1 (x)=\displaystyle\int_{\mathcal{H}_\gamma} F_1 (x,y) \mu^x (\mathrm{d}y)$ for $x\in\mathcal{H}_\gamma$.
\end{lemma}

\section{Proof of main result}\label{sect4}

\subsection{A time increment estimate and auxiliary process}\label{fro}

The time discretization technique proposed by Khasminskii in \cite{K68} will be used in the proof of Theorem \ref{MainRes}.
To this end, we provide an estimate of time increments of the slow component to Eq. (\ref{eq1}) and introduce an auxiliary process.

\begin{lemma}\label{Xdifflem}
	For $T>0$ and $\zeta \in (\frac \alpha2,\alpha-\sigma)$, there exists a constant $C_{T,\varrho_{\alpha}(\mathbf{B})}>0$ such that for any $x \in \mathcal{H}_{\gamma+\zeta}, y \in \mathcal H_\gamma$ and $\varepsilon, \delta >0$ small enough, we have
	$$\sup_{t \in [0,T]} \mathbb E |X_t^\varepsilon-X_{t(\delta)}^\varepsilon|_{\mathcal H_\gamma }^4
	\leq \delta^{4\zeta}  C_{T,\varrho_{\alpha}(\mathbf{B})} (1+|x|_{\mathcal H_{\gamma+\zeta} }^4+|y|_{\mathcal H_\gamma}^4),$$
	where $t(\delta):=[\frac{t}{\delta}]\delta$ with $[x]$ being the floor function, i.e. $[x]$ is the largest integer not exceeding $x$.
\end{lemma}
\begin{proof}
	By performing some simple calculations, we obtain the following result:
	\begin{align*}
		X_t^\varepsilon-X_{t(\delta)}^\varepsilon &=(S_{t-{t(\delta)}}-\mathrm{id} ) S_{{t(\delta)}}x+\int_{t(\delta)}^t S_{t-s} F_1(X_s^\varepsilon, Y_s^\varepsilon) \mathrm{d} s \nonumber\\
		&~~~+ (S_{t-{t(\delta)}}-\mathrm{id} )\int_0^{t(\delta)} S_{t(\delta)-s} F_1(X_s^\varepsilon, Y_s^\varepsilon) \mathrm{d} s
		\nonumber\\
		&~~~+\int_{t(\delta)}^t S_{t-s} G_1(X_s^\varepsilon) \mathrm{d} \mathbf B_s  +(S_{t-{t(\delta)}}-\mathrm{id} )\int_0^{t(\delta)} S_{t(\delta)-s} G_1(X_s^\varepsilon) \mathrm{d} \mathbf B_s  \nonumber\\
		&=: K_1+K_2+K_3+K_4+K_5.
	\end{align*}
	For $K_1$, using (\ref{group2}) we have
	\begin{equation}\label{Xt}
		|(S_{t-{t(\delta)}}-\mathrm{id} ) S_{{t(\delta)}}x|_{\mathcal H_\gamma }^4
		\lesssim |t-t(\delta)|^{4\zeta} |S_{{t(\delta)}}x|_{\mathcal H_{\gamma+\zeta} }^4
		\lesssim \delta^{4\zeta}  |x|_{\mathcal H_{\gamma+\zeta} }^4.
	\end{equation}
	For $K_2$, by (\ref{group1}), (\ref{F1Lip}) and H\"{o}lder's inequality we get
	\begin{align}\label{Xth1}
		\mathbb E \left |\int_{t(\delta)}^t S_{t-s} F_1(X_s^\varepsilon, Y_s^\varepsilon) \mathrm{d} s \right|_{\mathcal H_\gamma }^4
		&\lesssim \mathbb E \left (  \int_{t(\delta)}^t |t-s|^{-\alpha} | F_1(X_s^\varepsilon, Y_s^\varepsilon)|_{\mathcal H_{\gamma-\alpha} } \mathrm{d} s \right)^4 \nonumber\\
		& \lesssim |t-t(\delta)|^{3-4\alpha}  \mathbb E\int_{t(\delta)}^t (1+|X_s^\varepsilon|_{\mathcal H_\gamma}^4+|Y_s^\varepsilon|_{\mathcal H_\gamma}^4) \mathrm{d} s \nonumber\\
		&\lesssim \delta^{4(1-\alpha)} \left(1+\mathbb E\|X^\varepsilon\|_{\infty,\mathcal H_\gamma}^4+\sup_{s \in [0,T]} \mathbb E|Y_s^\varepsilon|_{\mathcal H_\gamma}^4 \right) \nonumber\\
		&\leq \delta^{4(1-\alpha)} C_{T,\varrho_{\alpha}(\mathbf{B})} (1+|x|_{\mathcal H_\gamma}^4+|y|_{\mathcal H_\gamma}^4),
	\end{align}
	where we used (\ref{X2}) and (\ref{Yest}) for the last inequality.
	For $K_3$, by (\ref{group1}), (\ref{group2}) and (\ref{F1Bound}) we get
	\begin{align}\label{Xth2}
		&~~~ \mathbb E \left |(S_{t-{t(\delta)}}-\mathrm{id} ) \int_0^{t(\delta)} S_{t(\delta)-s} F_1(X_s^\varepsilon, Y_s^\varepsilon) \mathrm{d} s \right|_{\mathcal H_\gamma }^4  \nonumber\\
		&\lesssim |t-t(\delta)|^{4\alpha}  \mathbb E \left | \int_0^{t(\delta)} S_{t(\delta)-s} F_1(X_s^\varepsilon, Y_s^\varepsilon) \mathrm{d} s \right|_{\mathcal H_{\gamma+\alpha} }^4 \nonumber\\
		&\lesssim \delta^{4\alpha} \mathbb E
		\left (  \int_0^{t(\delta)} |t(\delta)-s|^{-2\alpha} | F_1(X_s^\varepsilon, Y_s^\varepsilon)|_{\mathcal H_{\gamma-\alpha} } \mathrm{d} s \right)^4
		\leq C_T  \delta^{4\alpha}.
	\end{align}
	Using (\ref{group1}), (\ref{IntBound}), Lemmas \ref{funCRP} and \ref{slowest} we have
	\begin{align}\label{B1}
		|K_4|_{\mathcal H_\gamma }
		\leq &\left| \int_{t(\delta)}^t S_{t-s} G_1(X_s^\varepsilon) \mathrm{d} \mathbf B_s-S_{t-t(\delta)} G_1(X_{t(\delta)}^\varepsilon)  B_{t(\delta), t}-S_{t-t(\delta)} DG_1(X_{t(\delta)}^\varepsilon)G_1(X_{t(\delta)}^\varepsilon) \mathbb{B}_{t(\delta),t} \right|_{\mathcal H_\gamma } \nonumber\\
		&+ \left|S_{t-t(\delta)} G_1(X_{t(\delta)}^\varepsilon)  B_{t(\delta), t} \right|_{\mathcal H_\gamma }
		+\left|S_{t-t(\delta)} DG_1(X_{t(\delta)}^\varepsilon)G_1(X_{t(\delta)}^\varepsilon) \mathbb{B}_{t(\delta),t} \right|_{\mathcal H_\gamma } \nonumber\\
		\lesssim & |t-t(\delta)|^{\alpha-\sigma} \varrho_{\alpha}(\mathbf{B}) \|G_1(X^\varepsilon),DG_1(X^\varepsilon)G_1(X^\varepsilon) \|_{\mathcal{D}_{B,\gamma-\sigma}^{2\alpha}} \nonumber\\
		&+|G_1(X_{t(\delta)}^\varepsilon)|_{\mathcal H_{\gamma-\sigma}^{d_1}  } |B|_\alpha |t-t(\delta)|^{\alpha-\sigma}+ |DG_1(X_{t(\delta)}^\varepsilon)G_1(X_{t(\delta)}^\varepsilon)|_{\mathcal H_{\gamma-\alpha-\sigma}^{d_1\times d_1} } |\mathbb{B}|_{2\alpha} |t-t(\delta)|^{\alpha-\sigma} \nonumber\\
		\lesssim & \delta^{\alpha-\sigma} \varrho_{\alpha}(\mathbf{B}) \|G_1(X^\varepsilon),DG_1(X^\varepsilon)G_1(X^\varepsilon) \|_{\mathcal{D}_{B,\gamma-\sigma}^{2\alpha}} \nonumber\\
		\lesssim &  \delta^{\alpha-\sigma} (1+\varrho_{\alpha}(\mathbf{B}))^2 \varrho_{\alpha}(\mathbf{B}) \left(1+\|X^\varepsilon,G_1(X^\varepsilon)\|_{\mathcal{D}_{B,\gamma}^{2\alpha}} \right)
		\leq C_{T,\varrho_{\alpha}(\mathbf{B})} (1+|x|_{\mathcal H_\gamma}) \delta^{\alpha-\sigma},
	\end{align}
	where we used Lemma \ref{slowest} for the last inequality.
	By (\ref{group1}), (\ref{group2}), Lemmas \ref{funCRP} and \ref{slowest} we have
	\begin{align}\label{B2}
		|K_5|_{\mathcal H_\gamma }
		\lesssim & |t-t(\delta)|^{\zeta} \left| \int_0^{t(\delta)} S_{t(\delta)-s} G_1(X_s^\varepsilon) \mathrm{d} \mathbf B_s  \right|_{\mathcal H_{\gamma+\zeta} } \nonumber\\
		\lesssim & \delta^\zeta \left| \int_0^{t(\delta)} S_{t(\delta)-s} G_1(X_s^\varepsilon) \mathrm{d} \mathbf B_s -S_{t(\delta)} G_1(x) B_{0,{t(\delta)}}-S_{t(\delta)} DG_1(x)G_1(x)  \mathbb B_{0,{t(\delta)}}  \right|_{\mathcal H_{\gamma+\zeta} } \nonumber\\
		&+\delta^\zeta |t(\delta)|^\alpha |B|_\alpha |S_{t(\delta)} G_1(x)|_{\mathcal H_{\gamma+\zeta}^{d_1} }
		+\delta^\zeta |t(\delta)|^{2\alpha} |\mathbb{B}|_{2\alpha} |S_{t(\delta)} DG_1(x)G_1(x) |_{\mathcal H_{\gamma+\zeta}^{d_1 \times d_1} } \nonumber\\
		\lesssim & \delta^\zeta \varrho_{\alpha}(\mathbf{B}) \|G_1(X^\varepsilon),DG_1(X^\varepsilon)G_1(X^\varepsilon)\|_{\mathcal{D}_{B,\gamma-\sigma}^{2\alpha}} |t(\delta)|^{\alpha-\sigma-\zeta} \nonumber\\
		&+ \delta^\zeta |t(\delta)|^{\alpha-\sigma-\zeta} |B|_\alpha |G_1(x)|_{\mathcal H_{\gamma-\sigma}^{d_1}  }
		+\delta^\zeta |t(\delta)|^{\alpha-\sigma-\zeta} |\mathbb{B}|_{2\alpha}
		|DG_1(x)G_1(x) |_{\mathcal H_{\gamma-\alpha-\sigma}^{d_1 \times d_1} } \nonumber\\
		\lesssim & T^{\alpha-\sigma-\zeta} \delta^\zeta (1+\varrho_{\alpha}(\mathbf{B}))^2 \varrho_{\alpha}(\mathbf{B}) \left(1+\|X^\varepsilon,G_1(X^\varepsilon)\|_{\mathcal{D}_{B,\gamma}^{2\alpha}} \right) \nonumber\\
		&+\delta^\zeta T^{\alpha-\sigma-\zeta} |B|_\alpha (1+|x|_{\mathcal H_\gamma})
		+\delta^\zeta T^{\alpha-\sigma-\zeta} |\mathbb{B}|_{2\alpha} (1+|x|_{\mathcal H_\gamma}) \nonumber\\
		\lesssim & \delta^\zeta C_{T,\varrho_{\alpha}(\mathbf{B})} (1+|x|_{\mathcal H_\gamma}),
	\end{align}
	where we used (\ref{G1}), (\ref{DG1G1}) and $\mathcal H_\gamma \hookrightarrow \mathcal H_{\gamma-\alpha}$ for the fourth inequality.

	Combining the above computations and $\mathcal H_{\gamma+\zeta} \hookrightarrow \mathcal H_\gamma$, we have
	\begin{equation*}
		\mathbb E |X_t^\varepsilon-X_{t(\delta)}^\varepsilon|_{\mathcal H_\gamma }^4 \leq \delta^{4\zeta}  C_{T,\varrho_{\alpha}(\mathbf{B})} (1+|x|_{\mathcal H_{\gamma+\zeta} }^4+|y|_{\mathcal H_{\gamma} }^4).
	\end{equation*}
	The proof is completed.
\end{proof}

Next, we introduce an auxiliary process.
More specifically, we divide the time interval $[0, T]$ into some subintervals of size $\delta$,
where $\delta$ is a positive number depending on $\varepsilon$ and to be chosen later, we then consider
\begin{equation}\label{Auxeq}
	\left\{ \begin{aligned}
		\mathrm{d} \hat{Y}_t^\varepsilon&=\frac 1\varepsilon \big[L \hat{Y}_t^\varepsilon+F_2 (X_{t(\delta)}^\varepsilon,\hat{Y}_t^\varepsilon)\big] \mathrm{d} t+\frac 1{\sqrt{\varepsilon}} G_2(X_{t(\delta)}^\varepsilon,\hat{Y}_t^\varepsilon)\mathrm{d} W_t,\\
		\hat{Y}_0^\varepsilon&=y \in \mathcal{H}_\gamma.
	\end{aligned} \right.
\end{equation}
The definition of $t(\delta)$ is given in Lemma \ref{Xdifflem}. Then, on each subinterval $[k\delta,(k+1)\delta \wedge T]$ with $k\in\mathbf{N}_0$, Eq. (\ref{Auxeq}) is equivalent to
\begin{equation*}
	\hat{Y}_t^\varepsilon=\hat{Y}_{k\delta}^\varepsilon+\frac 1\varepsilon\int_{k\delta}^t L \hat{Y}_s^\varepsilon \mathrm{d} s+\frac 1\varepsilon\int_{k\delta}^t F_2(X_{k\delta}^\varepsilon,\hat{Y}_s^\varepsilon) \mathrm{d} s+\frac 1{\sqrt{\varepsilon}}\int_{k\delta}^tG_2(X_{k\delta}^\varepsilon,\hat{Y}_s^\varepsilon)\mathrm{d} W_s.
\end{equation*}

In order to prove Theorem \ref{MainRes}, we need an estimate of the difference between the solution $Y^\varepsilon$ to Eq. (\ref{eqY}) and the auxiliary process $\hat{Y}^\varepsilon$.
One can refer to \cite[Lemma 4.7]{LLPX25} for the detailed proof.

\begin{lemma}\label{AuxiliaryEstimate}
	For $\zeta \in (\frac \alpha2, \alpha-\sigma)$ and $\varepsilon, \delta >0$ small enough, there exists a constant $C_{T,\varrho_{\alpha}(\mathbf{B})}>0$ independent of $\varepsilon$ and $\delta$ such that for any $x \in \mathcal{H}_{\gamma+\zeta}, y \in \mathcal H_\gamma$ we have
	\begin{equation}\label{AimAuxdifference}
		\sup_{t\in [0,T]}\mathbb E | Y_t^\varepsilon-\hat{Y}_t^\varepsilon|_{\mathcal H_\gamma }^4 \leq \delta^{4 \zeta} C_{T,\varrho_{\alpha}(\mathbf{B})}  (1+|x|_{\mathcal{H}_{\gamma+\zeta} }^4+|y|_{\mathcal H_\gamma }^4 ),
	\end{equation}
	and
	\begin{equation}\label{Auxest}
		\sup_{t\in[0,T]}\mathbb E|\hat{Y}_t^\varepsilon|_{\mathcal H_\gamma }^2\leq C_T(1+| x |_{\mathcal H_\gamma }^2+|y|_{\mathcal H_\gamma }^2),
	\end{equation}
where $C_T$ is a positive constant that depends only on the time $T$.
\end{lemma}

\subsection{Proof of Theorem \ref{MainRes}}\label{proof}

We now prove Theorem \ref{MainRes}, namely the solutions to the slow component in the slow--fast system (\ref{eq1}) converge strongly to the solution to the averaged equation (\ref{aveq}) in the sense of $L^2(\Omega; \mathcal C([0,T]; \mathcal H_\gamma))$ as $\varepsilon \to 0$.
The proof is divided into the following four steps:

\textbf{Step1:} According to the definitions of solutions to Eqs. (\ref{eq1}) and (\ref{aveq}), we have
\begin{align*}
X_t^\varepsilon-\bar{X}_t
=& \int_0^t S_{t-s} \big[ F_1(X_s^\varepsilon,Y_s^\varepsilon)-F_1(X_{s(\delta)}^\varepsilon,\hat{Y}_s^\varepsilon) \big] \mathrm{d} s +\int_0^t S_{t-s} \big[F_1(X_{s(\delta)}^\varepsilon,\hat{Y}_s^\varepsilon)-\bar{F}_1(X_{s(\delta)}^\varepsilon) \big]\mathrm{d} s \nonumber\\
&+\int_0^t S_{t-s} \big[\bar{F}_1(X_{s(\delta)}^\varepsilon)-\bar{F}_1(X_s^\varepsilon) \big]\mathrm{d} s  \nonumber+\int_0^t S_{t-s} \big[\bar{F}_1(X_s^\varepsilon)-\bar{F}_1(\bar{X}_s) \big]\mathrm{d} s  \nonumber\\
&+ \int_0^t S_{t-s} \big[G_1(X_s^\varepsilon)-G_1(\bar{X}_s) \big] \mathrm{d} \mathbf B_s.
\end{align*}
We define
\begin{align*}
\varGamma_t^\varepsilon
:=& \int_0^t S_{t-s} \big[F_1(X_s^\varepsilon,Y_s^\varepsilon)-F_1(X_{s(\delta)}^\varepsilon,\hat{Y}_s^\varepsilon) \big]\mathrm{d} s +\int_0^t S_{t-s} \big[F_1(X_{s(\delta)}^\varepsilon,\hat{Y}_s^\varepsilon)-\bar{F}_1(X_{s(\delta)}^\varepsilon) \big]\mathrm{d} s\nonumber\\
&+\int_0^t S_{t-s} \big[\bar{F}_1(X_{s(\delta)}^\varepsilon)-\bar{F}_1(X_s^\varepsilon) \big]\mathrm{d} s.
\end{align*}
It is easy to deduce that
\begin{align}\label{Xdif}
&\left\|(X^\varepsilon-\bar{X}, G_1(X^\varepsilon)-G_1(\bar{X}))-(\varGamma^\varepsilon,0) \right\|_{\mathcal{D}_{ B,\gamma}^{2\alpha}} \nonumber\\
= &\left\| \int_0^\cdot S_{\cdot-s} \big[\bar{F}_1(X_s^\varepsilon)-\bar{F}_1(\bar{X}_s) \big]\mathrm{d} s+\int_0^\cdot S_{\cdot-s} \big[G_1(X_s^\varepsilon)-G_1(\bar{X}_s) \big] \mathrm{d} \mathbf B_s, G_1(X^\varepsilon)-G_1(\bar{X})  \right\|_{\mathcal{D}_{ B,\gamma}^{2\alpha}} \nonumber\\
\leq & \left\| \int_0^\cdot S_{\cdot-s} \big[\bar{F}_1(X_s^\varepsilon)-\bar{F}_1(\bar{X}_s) \big]\mathrm{d} s,0 \right\|_{\mathcal{D}_{ B,\gamma}^{2\alpha}} \nonumber\\
&+\left\|\int_0^\cdot S_{\cdot-s} \big[G_1(X_s^\varepsilon)-G_1(\bar{X}_s) \big] \mathrm{d} \mathbf B_s, G_1(X^\varepsilon)-G_1(\bar{X})  \right\|_{\mathcal{D}_{ B,\gamma}^{2\alpha}} \nonumber\\
=:& L_1+L_2.
\end{align}
For $L_1$, from (\ref{remainder}) and (\ref{norm}) we have
\begin{align}\label{L1}
L_1
=& \left\| \int_0^\cdot S_{\cdot-s} \big[\bar{F}_1(X_s^\varepsilon)-\bar{F}_1(\bar{X}_s) \big]\mathrm{d} s \right\|_{\infty,\mathcal H_\gamma}
+\left| \int_0^\cdot S_{\cdot-s} \big[\bar{F}_1(X_s^\varepsilon)-\bar{F}_1(\bar{X}_s) \big]\mathrm{d} s \right|_{\alpha,\mathcal H_{\gamma-\alpha}} \nonumber\\
&+\left| \int_0^\cdot S_{\cdot-s} \big[\bar{F}_1(X_s^\varepsilon)-\bar{F}_1(\bar{X}_s) \big]\mathrm{d} s \right|_{2\alpha,\mathcal H_{\gamma-2\alpha}}.
\end{align}
For the first term of the right-hand side of (\ref{L1}), by (\ref{group1}) and (\ref{Fbar}) we have
\begin{align} \label{barFsup}
&\left\| \int_0^\cdot S_{\cdot-s} \big[\bar{F}_1(X_s^\varepsilon)-\bar{F}_1(\bar{X}_s) \big]\mathrm{d} s \right\|_{\infty,\mathcal H_\gamma}
 =\sup_{t \in [0,T]}\left| \int_0^t S_{t-s}  \big[\bar{F}_1(X_s^\varepsilon)-\bar{F}_1(\bar{X}_s) \big]\mathrm{d} s \right|_{\mathcal H_\gamma } \nonumber \\
\lesssim & \sup_{t \in [0,T]} \left(\int_0^t (t-s)^{-\alpha} \big|\bar{F}_1(X_s^\varepsilon)-\bar{F}_1(\bar{X}_s) \big|_{\mathcal H_{\gamma-\alpha} } \mathrm{d} s  \right) \nonumber \\
\lesssim & \sup_{t \in [0,T]} \left( \int_0^t (t-s)^{-\alpha} |X_s^\varepsilon-\bar{X}_s|_{\mathcal H_\gamma } \mathrm{d} s \right)\nonumber \\
\lesssim &  \|X^\varepsilon-\bar{X}\|_{\infty,\mathcal H_\gamma} T^{1-\alpha}.
\end{align}
For the remainder terms $ \left| \int_0^\cdot S_{\cdot-s} [\bar{F}_1(X_s^\varepsilon)-\bar{F}_1(\bar{X}_s) ]  \mathrm{d} s \right|_{\theta,\mathcal H_{\gamma-\theta}}$ with $\theta \in {\{\alpha, 2\alpha\}}$,
by simple calculations we have
\begin{align*}
&\left|\int_0^t S_{t-r} \big[\bar{F}_1(X_r^\varepsilon)-\bar{F}_1(\bar{X}_r) \big] \mathrm{d} r-\int_0^s S_{s-r} \big[\bar{F}_1(X_r^\varepsilon)-\bar{F}_1(\bar{X}_r) \big] \mathrm{d} r \right|_{\mathcal H_{\gamma-\theta}}\\
=&\left| \int_s^t S_{t-r} \big[\bar{F}_1(X_r^\varepsilon)-\bar{F}_1(\bar{X}_r) \big] \mathrm{d} r+\int_0^s (S_{t-s}-\mathrm{id} )S_{s-r} \big[\bar{F}_1(X_r^\varepsilon)-\bar{F}_1(\bar{X}_r) \big] \mathrm{d} r \right|_{\mathcal H_{\gamma-\theta}}\\
\leq & \left|\int_s^t S_{t-r} \big[\bar{F}_1(X_r^\varepsilon)-\bar{F}_1(\bar{X}_r) \big] \mathrm{d} r \right|_{\mathcal H_{\gamma-\theta}}
+\left|(S_{t-s}-\mathrm{id} )\int_0^s S_{s-r} \big[\bar{F}_1(X_r^\varepsilon)-\bar{F}_1(\bar{X}_r) \big] \mathrm{d} r \right|_{\mathcal H_{\gamma-\theta}}.
\end{align*}
For $ \left|\int_s^t S_{t-r} [\bar{F}_1(X_r^\varepsilon)-\bar{F}_1(\bar{X}_r) ] \mathrm{d} r \right|_{\mathcal H_{\gamma-\theta}}$, applying (\ref{group1}) and (\ref{Fbar}) we have
\begin{align}\label{F1the1}
&~~~\left|\int_s^t S_{t-r} \big[\bar{F}_1(X_r^\varepsilon)-\bar{F}_1(\bar{X}_r) \big] \mathrm{d} r \right|_{\mathcal H_{\gamma-\theta}}
\leq \int_s^t \left|S_{t-r} \big[\bar{F}_1(X_r^\varepsilon)-\bar{F}_1(\bar{X}_r) \big] \right|_{\mathcal H_{\gamma-\theta}}\mathrm{d} r \nonumber\\
& \lesssim  \int_s^t (t-r)^{\theta-2\alpha} \left|\bar{F}_1(X_r^\varepsilon)-\bar{F}_1(\bar{X}_r)\right|_{\mathcal H_{\gamma-2\alpha} }\mathrm{d} r
 \lesssim  \int_s^t (t-r)^{\theta-2\alpha} |X_r^\varepsilon-\bar{X}_r|_{\mathcal H_\gamma} \mathrm{d} r \nonumber\\
& \lesssim \|X^\varepsilon-\bar{X}\|_{\infty,\mathcal H_\gamma} (t-s)^{1+\theta-2\alpha}.
\end{align}
According to (\ref{group2}) and (\ref{barFsup}), it follows that
\begin{align} \label{F1the2}
& \left|(S_{t-s}-\mathrm{id} )\int_0^s S_{s-r} \big[\bar{F}_1(X_r^\varepsilon)-\bar{F}_1(\bar{X}_r) \big] \mathrm{d} r \right|_{\mathcal H_{\gamma-\theta}} \nonumber\\
\lesssim & (t-s)^\theta \left|\int_0^s S_{s-r} \big[\bar{F}_1(X_r^\varepsilon)-\bar{F}_1(\bar{X}_r) \big] \mathrm{d} r \right|_{\mathcal H_\gamma } \nonumber\\
\lesssim & (t-s)^\theta \|X^\varepsilon-\bar{X}\|_{\infty,\mathcal H_\gamma} T^{1-\alpha}.
\end{align}
Combining (\ref{F1the1}) and (\ref{F1the2}) we have
\begin{equation} \label{F1Hoder}
\left| \int_0^\cdot S_{\cdot-s} \big[\bar{F}_1(X_s^\varepsilon)-\bar{F}_1(\bar{X}_s) \big]\mathrm{d} s \right|_{\theta,\mathcal H_{\gamma-\theta}}
\lesssim \|X^\varepsilon-\bar{X}\|_{\infty,\mathcal H_\gamma} (T^{1-2\alpha}+T^{1-\alpha}).
\end{equation}
Thus by (\ref{barFsup}) and (\ref{F1Hoder}) we obtain
\begin{equation}\label{Frem}
\left\| \int_0^\cdot S_{\cdot-s} \big[\bar{F}_1(X_s^\varepsilon)-\bar{F}_1(\bar{X}_s) \big]\mathrm{d} s,0 \right\|_{\mathcal{D}_{ B,\gamma}^{2\alpha}}
\lesssim \|X^\varepsilon-\bar{X}\|_{\infty,\mathcal H_\gamma} (T^{1-2\alpha}+T^{1-\alpha}).
\end{equation}

For $L_2$, using Lemmas \ref{stab} and \ref{Glem2} we have
\begin{align}\label{L2}
L_2&=\left\|\int_0^\cdot S_{\cdot-s} \big[G_1(X_s^\varepsilon)-G_1(\bar{X}_s) \big] \mathrm{d} \mathbf B_s, G_1(X^\varepsilon)-G_1(\bar{X})  \right\|_{\mathcal{D}_{ B,\gamma}^{2\alpha}}   \nonumber\\
& \lesssim (1+\varrho_{\alpha}(\mathbf{B}) ) \left( |G_1(X_0^\varepsilon)-G_1(\bar{X}_0)|_{\mathcal H_{\gamma-\sigma}^{d_1} }+|DG_1(X_0^\varepsilon)G_1(X_0^\varepsilon)-DG_1(\bar{X}_0)G_1(\bar{X}_0)|_{\mathcal H_{\gamma-\alpha-\sigma}^{d_1 \times d_1} } \right) \nonumber\\
&~~~ + T^{\alpha-\sigma}(1+\varrho_{\alpha}(\mathbf{B}))  \|G_1(X^\varepsilon)-G_1(\bar{X}),DG_1(X^\varepsilon)G_1(X^\varepsilon)-DG_1(\bar{X})G_1(\bar{X})\|_{\mathcal{D}_{ B,\gamma-\sigma}^{2\alpha}} \nonumber\\
& \lesssim T^{\alpha-\sigma} C_{T,\varrho_{\alpha}(\mathbf{B})} (1+|x|_{\mathcal H_\gamma})^2 \|X^\varepsilon-\bar{X},G_1(X^\varepsilon)-G_1(\bar{X})\|_{\mathcal{D}_{ B,\gamma}^{2\alpha}},
\end{align}
where we used $X_0^\varepsilon=\bar{X}_0=x$ in the last one.
Substituting (\ref{Frem}) and (\ref{L2}) into (\ref{Xdif}) we have
\begin{align*}
&\|(X^\varepsilon-\bar{X}, G_1(X^\varepsilon)-G_1(\bar{X}))-(\varGamma^\varepsilon,0)\|_{\mathcal{D}_{ B,\gamma}^{2\alpha}} \nonumber\\
\leq& C_{T,\varrho_{\alpha}(\mathbf{B})} (1+|x|_{\mathcal H_\gamma})^2 (T^{1-2\alpha} +T^{1-\alpha}+T^{\alpha-\sigma}) \|X^\varepsilon-\bar{X},G_1(X^\varepsilon)-G_1(\bar{X})\|_{\mathcal{D}_{ B,\gamma}^{2\alpha}}.
\end{align*}

Let $0<T<1$ be small enough such that $3 C_{T,\varrho_{\alpha}(\mathbf{B})} (1+|x|_{\mathcal H_\gamma})^2 T^{(1-2\alpha)\wedge (\alpha-\sigma)} <1/2$, then we have
\begin{align*}
&\|X^\varepsilon-\bar{X}, G_1(X^\varepsilon)-G_1(\bar{X})\|_{\mathcal{D}_{B,\gamma}^{2\alpha}([0,T])} \nonumber\\
\leq &\|\varGamma^\varepsilon,0\|_{\mathcal{D}_{B,\gamma}^{2\alpha}}
+ \frac 12 \|X^\varepsilon-\bar{X},G_1(X^\varepsilon)-G_1(\bar{X})\|_{\mathcal{D}_{B,\gamma}^{2\alpha}([0,T])}.
\end{align*}
It is easy to obtain that
\begin{equation}\label{normT}
 \|X^\varepsilon-\bar{X}, G_1(X^\varepsilon)-G_1(\bar{X})\|_{\mathcal{D}_{B,\gamma}^{2\alpha}([0,T])} \leq 2 \|\varGamma^\varepsilon,0\|_{\mathcal{D}_{B,\gamma}^{2\alpha}}.
\end{equation}

For any $T>0$, we divide the time interval $[0,T]$ into $N \in \mathbf N$ identical subintervals such that $ 3 C_{T,\varrho_{\alpha}(\mathbf{B})} (1+|x|_{\mathcal H_\gamma})^2 {(\frac T N)}^{(1-2\alpha)\wedge (\alpha-\sigma)} <1/2$,
then on each subinterval $[\frac kN T,\frac {k+1}N T]$ with $k=0, \cdots , N-1$, similar to the estimation of (\ref{normT}), one has
\begin{align}\label{Xk}
\|X^\varepsilon-\bar{X}, G_1(X^\varepsilon)-G_1(\bar{X})\|_{\mathcal{D}_{B,\gamma}^{2\alpha}([\frac kN T,\frac {k+1}N T])}
\leq & 2 \|\varGamma^\varepsilon,0\|_{\mathcal{D}_{B,\gamma}^{2\alpha}}+C_{\varrho_{\alpha}(\mathbf{B})} \left|X_{\frac kN T}^\varepsilon-\bar{X}_{\frac kN T} \right|_{\mathcal H_\gamma } \nonumber\\
\leq & 2 \|\varGamma^\varepsilon,0\|_{\mathcal{D}_{B,\gamma}^{2\alpha}}+C_{\varrho_{\alpha}(\mathbf{B})} \|X^\varepsilon-\bar{X}\|_{\infty,\mathcal H_\gamma,[\frac {k-1}N T,\frac kN T]}.
\end{align}
For $k=1$, using (\ref{norm}) and (\ref{normT}) we obtain
\begin{align*}
&\|X^\varepsilon-\bar{X}, G_1(X^\varepsilon)-G_1(\bar{X})\|_{\mathcal{D}_{B,\gamma}^{2\alpha}([\frac TN,\frac {2T}N ])}\\
\leq & 2 \|\varGamma^\varepsilon,0\|_{\mathcal{D}_{B,\gamma}^{2\alpha}}+C_{\varrho_{\alpha}(\mathbf{B})}\|X^\varepsilon-\bar{X}, G_1(X^\varepsilon)-G_1(\bar{X})\|_{\mathcal{D}_{B,\gamma}^{2\alpha}([0,\frac TN ])} \\
\leq & 2 \|\varGamma^\varepsilon,0\|_{\mathcal{D}_{B,\gamma}^{2\alpha}}+2C_{\varrho_{\alpha}(\mathbf{B})}\|\varGamma^\varepsilon,0\|_{\mathcal{D}_{B,\gamma}^{2\alpha}}
\leq 2^2 C_{\varrho_{\alpha}(\mathbf{B})}\|\varGamma^\varepsilon,0\|_{\mathcal{D}_{B,\gamma}^{2\alpha}}.
\end{align*}
Similarly, for $k=2$, we have
\begin{align*}
&\|X^\varepsilon-\bar{X}, G_1(X^\varepsilon)-G_1(\bar{X})\|_{\mathcal{D}_{B,\gamma}^{2\alpha}([\frac {2T}N,\frac {3T}N ])}\\
\leq & 2 \|\varGamma^\varepsilon,0\|_{\mathcal{D}_{B,\gamma}^{2\alpha}}+C_{\varrho_{\alpha}(\mathbf{B})}\|X^\varepsilon-\bar{X}, G_1(X^\varepsilon)-G_1(\bar{X})\|_{\mathcal{D}_{B,\gamma}^{2\alpha}([\frac TN,\frac {2T}N ])} \\
\leq & 2 \|\varGamma^\varepsilon,0\|_{\mathcal{D}_{B,\gamma}^{2\alpha}}+2^2(C_{\varrho_{\alpha}(\mathbf{B})})^2\|\varGamma^\varepsilon,0\|_{\mathcal{D}_{B,\gamma}^{2\alpha}}
\leq 2^3 (C_{\varrho_{\alpha}(\mathbf{B})})^2\|\varGamma^\varepsilon,0\|_{\mathcal{D}_{B,\gamma}^{2\alpha}}.
\end{align*}
Therefore, for every subinterval $[\frac kN T,\frac {k+1}N T]$ with $k=0, \cdots , N-1$, by recursively using (\ref{Xk}) yields
\begin{align}\label{subk}
&\|X^\varepsilon-\bar{X}\|_{\infty,\mathcal H_\gamma,[\frac kN T,\frac {k+1}N T]} \nonumber\\
\leq &\|X^\varepsilon-\bar{X}, G_1(X^\varepsilon)-G_1(\bar{X})\|_{\mathcal{D}_{B,\gamma}^{2\alpha}([\frac kN T,\frac {k+1}N T])}
\leq  2^{k+1} (C_{\varrho_{\alpha}(\mathbf{B})})^k \|\varGamma^\varepsilon,0\|_{\mathcal{D}_{B,\gamma}^{2\alpha}}.
\end{align}

Consequently, for any $T>0$, we get
\begin{equation*}
\|X^\varepsilon-\bar{X}\|_{\infty,\mathcal H_\gamma,[0,T]}=\max_{k\in \{0, \cdots,N-1\}} \|X^\varepsilon-\bar{X}\|_{\infty,\mathcal H_\gamma,[\frac kN T,\frac {k+1}N T]} \leq  2^{N} (C_{\varrho_{\alpha}(\mathbf{B})})^{N-1} \|\varGamma^\varepsilon,0\|_{\mathcal{D}_{B,\gamma}^{2\alpha}}.
\end{equation*}
Recall that
\begin{align*}
	\varGamma_t^\varepsilon&= \int_0^t S_{t-s} \big[F_1(X_s^\varepsilon,Y_s^\varepsilon)
	-F_1(X_{s(\delta)}^\varepsilon,\hat{Y}_s^\varepsilon) \big]\mathrm{d} s+\int_0^t S_{t-s} \big[F_1(X_{s(\delta)}^\varepsilon,\hat{Y}_s^\varepsilon)-\bar{F}_1(X_{s(\delta)}^\varepsilon) \big]\mathrm{d} s\nonumber\\
	&~~~+\int_0^t S_{t-s} \big[\bar{F}_1(X_{s(\delta)}^\varepsilon)
	-\bar{F}_1(X_s^\varepsilon) \big]\mathrm{d} s.
\end{align*}
Based on the above calculations we obtain
\begin{align}\label{EXdif}
\mathbb E \|X^\varepsilon-\bar{X}\|_{\infty,\mathcal H_\gamma,[0,T]}^2	
& \leq  C_{N,\varrho_{\alpha}(\mathbf{B})} \mathbb E \left\| \displaystyle\int_0^\cdot S_{\cdot-s} \big[F_1(X_s^\varepsilon,Y_s^\varepsilon)
-F_1(X_{s(\delta)}^\varepsilon,\hat{Y}_s^\varepsilon) \big]\mathrm{d} s,0 \right\|_{\mathcal{D}_{B,\gamma}^{2\alpha}}^2 \nonumber\\
&~~~+ C_{N,\varrho_{\alpha}(\mathbf{B})} \mathbb E \left\| \displaystyle\int_0^\cdot S_{\cdot-s} \big[\bar{F}_1(X_{s(\delta)}^\varepsilon)
-\bar{F}_1(X_s^\varepsilon) \big]\mathrm{d} s,0 \right\|_{\mathcal{D}_{B,\gamma}^{2\alpha}}^2 \nonumber\\
&~~~+C_{N,\varrho_{\alpha}(\mathbf{B})} \mathbb E \left\| \displaystyle\int_0^\cdot S_{\cdot-s} \big[F_1(X_{s(\delta)}^\varepsilon,\hat{Y}_s^\varepsilon)-\bar{F}_1(X_{s(\delta)}^\varepsilon) \big]\mathrm{d} s,0 \right\|_{\mathcal{D}_{B,\gamma}^{2\alpha}}^2.
\end{align}

\textbf{Step2:} We estimate $\mathbb E \left\| \int_0^\cdot S_{\cdot-s} [F_1(X_s^\varepsilon,Y_s^\varepsilon)
-F_1(X_{s(\delta)}^\varepsilon,\hat{Y}_s^\varepsilon) ]\mathrm{d} s,0 \right\|_{\mathcal{D}_{B,\gamma}^{2\alpha}}^2$.
From (\ref{norm}), it is easy to see that
\begin{align}\label{F1dif}
&~~~ \mathbb E \left\| \displaystyle\int_0^\cdot S_{\cdot-s} \big[F_1(X_s^\varepsilon,Y_s^\varepsilon)
-F_1(X_{s(\delta)}^\varepsilon,\hat{Y}_s^\varepsilon) \big]\mathrm{d} s,0 \right\|_{\mathcal{D}_{B,\gamma}^{2\alpha}}^2 \nonumber\\	
&\leq  3 \mathbb E \left\| \displaystyle\int_0^\cdot S_{\cdot-s} \big[F_1(X_s^\varepsilon,Y_s^\varepsilon)
-F_1(X_{s(\delta)}^\varepsilon,\hat{Y}_s^\varepsilon) \big]\mathrm{d} s \right\|_{\infty,\mathcal H_\gamma}^2 \nonumber\\	
&~~~+ 3 \mathbb E \left| \displaystyle\int_0^\cdot S_{\cdot-s} \big[F_1(X_s^\varepsilon,Y_s^\varepsilon)
-F_1(X_{s(\delta)}^\varepsilon,\hat{Y}_s^\varepsilon) \big]\mathrm{d} s \right|_{\alpha,\mathcal H_{\gamma-\alpha}}^2 \nonumber\\
&~~~+3  \mathbb E \left| \displaystyle\int_0^\cdot S_{\cdot-s} \big[F_1(X_s^\varepsilon,Y_s^\varepsilon)
-F_1(X_{s(\delta)}^\varepsilon,\hat{Y}_s^\varepsilon) \big]\mathrm{d} s \right|_{2\alpha,\mathcal H_{\gamma-2\alpha}}^2 \nonumber\\
& =:\hat{I}_1+\hat{I}_2+\hat{I}_3.	
\end{align}
For $\hat{I}_1$, using (\ref{group1}), (\ref{F1Lip}) and H\"{o}lder's inequality we have
\begin{align}\label{K_1eta}
\hat{I}_1 & = 3\mathbb E \left [ \sup_{t \in [0,T]}
	\left| \int_0^t S_{t-s} \big[F_1(X_s^\varepsilon,Y_s^\varepsilon)-F_1(X_{s(\delta)}^\varepsilon,\hat{Y}_s^\varepsilon) \big]\mathrm{d} s \right|_{\mathcal H_\gamma }^2\right ] \nonumber\\
	& \lesssim  \mathbb E \left [ \sup_{t\in [0,T]} \left( \int_0^t |t-s|^{-\alpha}   \big| F_1(X_s^\varepsilon,Y_s^\varepsilon)-F_1(X_{s(\delta)}^\varepsilon,\hat{Y}_s^\varepsilon) \big|_{\mathcal H_{\gamma-\alpha} } \mathrm{d} s \right)^2\right ] \nonumber\\
	&\lesssim T^{1-2\alpha} \mathbb E\int_0^T \big| F_1(X_s^\varepsilon,Y_s^\varepsilon)-F_1(X_{s(\delta)}^\varepsilon,\hat{Y}_s^\varepsilon) \big|_{\mathcal H_{\gamma-\alpha} }^2 \mathrm{d} s \nonumber\\
	&\lesssim  T^{1-2\alpha}  \int_0^T  \mathbb E \big[| X_s^\varepsilon
	-X_{s(\delta)}^\varepsilon |_{\mathcal H_\gamma }^2+|Y_s^\varepsilon-\hat{Y}_s^\varepsilon|_{\mathcal H_\gamma }^2 \big] \mathrm{d} s \nonumber\\
	&\leq  C_{T,\varrho_{\alpha}(\mathbf{B})} (1+|x|_{\mathcal{H}_{\gamma+\zeta}}^2+|y|_{\mathcal{H}_\gamma}^2)\delta^{2\zeta},
\end{align}
where we used (\ref{AimAuxdifference}) and Lemma \ref{Xdifflem} in the last step.
For $\hat{I}_2$, it is easy to obtain that
\begin{align*}
\hat{I}_2 & \leq C   \mathbb E \left [ \sup_{s<t} |t-s|^{-2\alpha} \left| \int_s^t S_{t-r}\big[ F_1(X_r^\varepsilon,Y_r^\varepsilon)-F_1(X_{r(\delta)}^\varepsilon,\hat{Y}_r^\varepsilon) \big]\mathrm{d} r \right|_{\mathcal H_{\gamma-\alpha} }^2 \right ] \nonumber\\
&~~~+C \mathbb E \left [ \sup_{s<t} |t-s|^{-2\alpha} \left| (S_{t-s}-\mathrm{\mathrm{id} })\int_0^s S_{s-r}\big[F_1(X_r^\varepsilon,Y_r^\varepsilon)-F_1(X_{r(\delta)}^\varepsilon,\hat{Y}_r^\varepsilon) \big] \mathrm{d} r \right|_{\mathcal H_{\gamma-\alpha} }^2 \right ] \nonumber\\
&=: \hat{I}_{21}+\hat{I}_{22}.	
\end{align*}
Using H\"{o}lder's inequality, (\ref{F1Lip}), (\ref{AimAuxdifference}) and Lemma \ref{Xdifflem}, $\hat{I}_{21}$ is controlled by
\begin{align}\label{hatI21}
\hat{I}_{21} & \leq C \mathbb E \left [ \sup_{s<t} |t-s|^{1-2\alpha}  \int_s^t \big| S_{t-r} \big[ F_1(X_r^\varepsilon,Y_r^\varepsilon)-F_1(X_{r(\delta)}^\varepsilon,\hat{Y}_r^\varepsilon) \big] \big|_{\mathcal H_{\gamma-\alpha} }^2  \mathrm{d} r \right ] \nonumber\\
&\lesssim T^{1-2\alpha} \mathbb E \int_0^T \big| F_1(X_r^\varepsilon,Y_r^\varepsilon)-F_1(X_{r(\delta)}^\varepsilon,\hat{Y}_r^\varepsilon)  \big|_{\mathcal H_{\gamma-\alpha} }^2  \mathrm{d} r \nonumber\\
&\lesssim  T^{1-2\alpha}  \int_0^T  \mathbb E \big[| X_r^\varepsilon
-X_{r(\delta)}^\varepsilon |_{\mathcal H_\gamma }^2+|Y_r^\varepsilon-\hat{Y}_r^\varepsilon|_{\mathcal H_\gamma }^2 \big] \mathrm{d} r \nonumber\\
& \leq  C_{T,\varrho_{\alpha}(\mathbf{B})} (1+|x|_{\mathcal{H}_{\gamma+\zeta}}^2+|y|_{\mathcal{H}_\gamma}^2)\delta^{2\zeta}.
\end{align}
According to (\ref{group2}) and (\ref{K_1eta}) we get
\begin{align}\label{hatI22}
\hat{I}_{22} & \leq C \mathbb E \left [ \sup_{s<t} |t-s|^{-2\alpha} \left( |t-s|^\alpha \left| \int_0^s S_{s-r}\big[F_1(X_r^\varepsilon,Y_r^\varepsilon)-F_1(X_{r(\delta)}^\varepsilon,\hat{Y}_r^\varepsilon) \big] \mathrm{d} r \right|_{\mathcal H_\gamma }  \right)^2\right ] \nonumber\\	
& \leq C \mathbb E \left [ \sup_{s \in [0,T]}  \left| \int_0^s S_{s-r}\big[F_1(X_r^\varepsilon,Y_r^\varepsilon)-F_1(X_{r(\delta)}^\varepsilon,\hat{Y}_r^\varepsilon) \big] \mathrm{d} r \right|_{\mathcal H_{\gamma} }^2\right ] \nonumber\\
& \leq C_{T,\varrho_{\alpha}(\mathbf{B})} (1+|x|_{\mathcal{H}_{\gamma+\zeta}}^2+|y|_{\mathcal{H}_\gamma}^2)\delta^{2\zeta}.
\end{align}

For $\hat{I}_3$ we have
\begin{align*}
\hat{I}_3 & \leq C   \mathbb E \left [ \sup_{s<t} |t-s|^{-4\alpha} \left| \int_s^t S_{t-r}\big[ F_1(X_r^\varepsilon,Y_r^\varepsilon)-F_1(X_{r(\delta)}^\varepsilon,\hat{Y}_r^\varepsilon) \big]\mathrm{d} r \right|_{\mathcal H_{\gamma-2\alpha} }^2 \right ] \nonumber\\
	&~~~+C \mathbb E \left [ \sup_{s<t} |t-s|^{-4\alpha} \left| (S_{t-s}-\mathrm{\mathrm{id} })\int_0^s S_{s-r}\big[F_1(X_r^\varepsilon,Y_r^\varepsilon)-F_1(X_{r(\delta)}^\varepsilon,\hat{Y}_r^\varepsilon) \big] \mathrm{d} r \right|_{\mathcal H_{\gamma-2\alpha} }^2 \right ] \nonumber\\
	&=: \hat{I}_{31}+\hat{I}_{32}.	
\end{align*}
Similar to the estimation in (\ref{hatI22}), applying (\ref{group2}) and (\ref{K_1eta}) we obtain
\begin{align}\label{hatI32}
\hat{I}_{32} &\leq C \mathbb E \left [ \sup_{s<t} |t-s|^{-4\alpha} \left( |t-s|^{2\alpha} \left| \int_0^s S_{s-r}\big[F_1(X_r^\varepsilon,Y_r^\varepsilon)-F_1(X_{r(\delta)}^\varepsilon,\hat{Y}_r^\varepsilon) \big] \mathrm{d} r \right|_{\mathcal H_\gamma } \right)^2 \right ] \nonumber\\
&\leq C \mathbb E \left [ \sup_{s \in [0,T]}  \left| \int_0^s S_{s-r}\big[F_1(X_r^\varepsilon,Y_r^\varepsilon)-F_1(X_{r(\delta)}^\varepsilon,\hat{Y}_r^\varepsilon) \big] \mathrm{d} r \right|_{\mathcal H_{\gamma} }^2\right ] \nonumber\\
& \leq C_{T,\varrho_{\alpha}(\mathbf{B})} (1+|x|_{\mathcal{H}_{\gamma+\zeta}}^2+|y|_{\mathcal{H}_\gamma}^2)\delta^{2\zeta}.
\end{align}
For $\hat{I}_{31}$, it is straightforward that
\begin{align}\label{hatI31}
\hat{I}_{31}
	&\leq C   \mathbb E \left [ \sup_{s<t} |t-s|^{-4\alpha} \left| \int_s^t S_{t-r}\big[ F_1(X_r^\varepsilon,Y_r^\varepsilon)-F_1(X_{r(\delta)}^\varepsilon,\hat{Y}_r^\varepsilon) \big]\mathrm{d} r \right|_{\mathcal H_{\gamma-2\alpha} }^2  1_{ \{|t-s|\leq \delta\}}\right ] \nonumber\\
	&~~~+C   \mathbb E \left [ \sup_{s<t} |t-s|^{-4\alpha} \left| \int_s^t S_{t-r}\big[ F_1(X_r^\varepsilon,Y_r^\varepsilon)-F_1(X_{r(\delta)}^\varepsilon,\hat{Y}_r^\varepsilon) \big]\mathrm{d} r \right|_{\mathcal H_{\gamma-2\alpha} }^2  1_{ \{|t-s| >\delta\}}\right ] .
\end{align}
For the first term on the right-hand side of (\ref{hatI31}), by $\mathcal H_{\gamma-\alpha} \hookrightarrow \mathcal H_{\gamma-2\alpha}$ and (\ref{F1Bound}) we obtain
\begin{align}\label{tilK1}
	& C   \mathbb E \left [ \sup_{s<t} |t-s|^{-4\alpha} \left| \int_s^t S_{t-r}\big[ F_1(X_r^\varepsilon,Y_r^\varepsilon)-F_1(X_{r(\delta)}^\varepsilon,\hat{Y}_r^\varepsilon) \big]\mathrm{d} r \right|_{\mathcal H_{\gamma-2\alpha} }^2  1_{ \{|t-s|\leq \delta\}}\right ] \nonumber\\
  \leq  &C \mathbb E  \left [ \sup_{s<t} |t-s|^{-4\alpha} \left( \int_s^t \big| F_1(X_r^\varepsilon,Y_r^\varepsilon)-F_1(X_{r(\delta)}^\varepsilon,\hat{Y}_r^\varepsilon) \big|_{\mathcal H_{\gamma-2\alpha} } \mathrm{d} r \right)^2 1_{ \{|t-s|\leq \delta\}}\right ] \nonumber\\
 \leq & C \mathbb E \left [ \sup_{s<t} |t-s|^{-4\alpha} \left( \int_s^t \|F_1\|_{\infty, \mathcal H_{\gamma-\alpha}}  \mathrm{d} r \right)^2 1_{ \{|t-s|\leq \delta\}}\right ]
\leq C \delta^{2(1-2\alpha)}.
\end{align}
For the second term on the right-hand side of (\ref{hatI31}), using H\"{o}lder's inequality, (\ref{F1Lip}), (\ref{AimAuxdifference}) and Lemma \ref{Xdifflem} we get
\begin{align}\label{K12}
& C   \mathbb E \left [ \sup_{s<t} |t-s|^{-4\alpha} \left| \int_s^t S_{t-r}\big[ F_1(X_r^\varepsilon,Y_r^\varepsilon)-F_1(X_{r(\delta)}^\varepsilon,\hat{Y}_r^\varepsilon) \big]\mathrm{d} r \right|_{\mathcal H_{\gamma-2\alpha} }^2  1_{ \{|t-s| >\delta\}}\right ] \nonumber\\
\leq & \mathbb E \left [ \sup_{s<t} |t-s|^{\frac 32-4\alpha}  \left(\int_s^t \big|S_{t-r}\big[F_1(X_r^\varepsilon,Y_r^\varepsilon)-F_1(X_{r(\delta)}^\varepsilon,\hat{Y}_r^\varepsilon) \big] \big|_{\mathcal H_{\gamma-\alpha} }^4 \mathrm{d} r \right)^{\frac 12}  1_{ \{|t-s|> \delta\} } \right] \nonumber\\
	 \lesssim &  T^{3(\frac 12-\alpha)} \delta^{-\alpha} \mathbb E \left[ \sup_{s<t} \left(\int_s^t \big|F_1(X_r^\varepsilon,Y_r^\varepsilon)-F_1(X_{r(\delta)}^\varepsilon,\hat{Y}_r^\varepsilon)  \big|_{\mathcal H_{\gamma-\alpha} }^4 \mathrm{d} r \right)^{\frac 12}  1_{ \{|t-s|> \delta\} }  \right] \nonumber\\
	\leq  &C_T \delta^{-\alpha} \left\{ \mathbb E \int_0^T \big|F_1(X_r^\varepsilon,Y_r^\varepsilon)-F_1(X_{r(\delta)}^\varepsilon,\hat{Y}_r^\varepsilon)  \big|_{\mathcal H_{\gamma-\alpha} }^4 \mathrm{d} r \right \}^{\frac 12} \nonumber\\
	 \leq  & C_{T,\varrho_{\alpha}(\mathbf{B})} (1+|x|_{\mathcal{H}_{\gamma+\zeta}}^2+|y|_{\mathcal{H}_\gamma}^2)\delta^{2\zeta-\alpha},
\end{align}
where $\zeta \in (\frac \alpha2, \alpha-\sigma)$. Combining estimates (\ref{K_1eta})--(\ref{K12}) we have
\begin{align}\label{dif1}
 &\mathbb E \left\| \displaystyle\int_0^\cdot S_{\cdot-s} \big[F_1(X_s^\varepsilon,Y_s^\varepsilon)
-F_1(X_{s(\delta)}^\varepsilon,\hat{Y}_s^\varepsilon) \big]\mathrm{d} s,0 \right\|_{\mathcal{D}_{B,\gamma}^{2\alpha}}^2  \nonumber\\
\leq & C_{T,\varrho_{\alpha}(\mathbf{B})} (1+|x|_{\mathcal{H}_{\gamma+\zeta}}^2+|y|_{\mathcal{H}_\gamma}^2) (\delta^{2\zeta}+\delta^{2\zeta-\alpha}+\delta^{2(1-2\alpha)}).
\end{align}

For the second term on the right-hand side of (\ref{EXdif}), following similar arguments as in the proof of (\ref{dif1}) gives that
\begin{align}\label{dif2}
&\mathbb E \left\| \displaystyle\int_0^\cdot S_{\cdot-s} \big[\bar{F}_1(X_{s(\delta)}^\varepsilon)
-\bar{F}_1(X_s^\varepsilon) \big]\mathrm{d} s,0 \right\|_{\mathcal{D}_{B,\gamma}^{2\alpha}}^2 \nonumber\\
\leq & C_{T,\varrho_{\alpha}(\mathbf{B})} (1+|x|_{\mathcal{H}_{\gamma+\zeta}}^2+|y|_{\mathcal{H}_\gamma}^2) (\delta^{2\zeta}+\delta^{2\zeta-\alpha}+\delta^{2(1-2\alpha)}).
\end{align}

\textbf{Step3:} Next, let's estimate $ \mathbb E \left\|\int_0^\cdot S_{\cdot-s} \big[F_1(X_{s(\delta)}^\varepsilon,\hat{Y}_s^\varepsilon)-\bar{F}_1(X_{s(\delta)}^\varepsilon) \big]\mathrm{d} s,0 \right\|_{\mathcal{D}_{B,\gamma}^{2\alpha}}^2$.
It is straightforward that
\begin{align*}
	&~~~ \mathbb E \left\| \displaystyle\int_0^\cdot S_{\cdot-s} \big[F_1(X_{s(\delta)}^\varepsilon,\hat{Y}_s^\varepsilon)-\bar{F}_1(X_{s(\delta)}^\varepsilon) \big]\mathrm{d} s,0 \right\|_{\mathcal{D}_{B,\gamma}^{2\alpha}}^2 \nonumber\\	
	&\leq  3 \mathbb E \left\| \displaystyle\int_0^\cdot S_{\cdot-s} \big[F_1(X_{s(\delta)}^\varepsilon,\hat{Y}_s^\varepsilon)-\bar{F}_1(X_{s(\delta)}^\varepsilon) \big]\mathrm{d} s \right\|_{\infty,\mathcal H_\gamma}^2 \nonumber\\	
	&~~~+ 3 \mathbb E \left| \displaystyle\int_0^\cdot S_{\cdot-s} \big[F_1(X_{s(\delta)}^\varepsilon,\hat{Y}_s^\varepsilon)-\bar{F}_1(X_{s(\delta)}^\varepsilon) \big]\mathrm{d} s \right|_{\alpha,\mathcal H_{\gamma-\alpha}}^2 \nonumber\\
	&~~~+3  \mathbb E \left| \displaystyle\int_0^\cdot S_{\cdot-s} \big[F_1(X_{s(\delta)}^\varepsilon,\hat{Y}_s^\varepsilon)-\bar{F}_1(X_{s(\delta)}^\varepsilon) \big]\mathrm{d} s \right|_{2\alpha,\mathcal H_{\gamma-2\alpha}}^2 \nonumber\\
	& =:\hat{J}_1+\hat{J}_2+\hat{J}_3.	
\end{align*}
For $\hat{J}_1$, we have
\begin{align}\label{F3sup}
\hat{J}_1 &=3 \mathbb E \left[ \sup_{t \in [0,T]} \left| \int_0^t S_{t-s} \big[ F_1(X_{s(\delta)}^\varepsilon,\hat{Y}_s^\varepsilon)-\bar{F}_1(X_{s(\delta)}^\varepsilon) \big]  \mathrm{d} s \right|_{\mathcal H_\gamma }^2 \right]	 \nonumber\\
& \leq C \mathbb E \left[ \sup_{t \in [0,T]} \left| \int_0^t (S_{t-s}-S_{t-s(\delta)}) \big[ F_1(X_{s(\delta)}^\varepsilon,\hat{Y}_s^\varepsilon)-\bar{F}_1(X_{s(\delta)}^\varepsilon) \big] \mathrm{d} s \right|_{\mathcal H_\gamma }^2 \right]	 \nonumber\\
&~~~+ C\mathbb E \left[ \sup_{t \in [0,T]} \left| \int_0^t S_{t-s(\delta)} \big[ F_1(X_{s(\delta)}^\varepsilon,\hat{Y}_s^\varepsilon)-\bar{F}_1(X_{s(\delta)}^\varepsilon) \big]  \mathrm{d} s  \right|_{\mathcal H_\gamma }^2 \right].
\end{align}
For the first term on the right-hand side of (\ref{F3sup}), using (\ref{group1}) and (\ref{group2}) we have
\begin{align}\label{Fsup1}
&~~~\mathbb E \left[ \sup_{t \in [0,T]} \left| \int_0^t (S_{t-s}-S_{t-s(\delta)}) \big[ F_1(X_{s(\delta)}^\varepsilon,\hat{Y}_s^\varepsilon)-\bar{F}_1(X_{s(\delta)}^\varepsilon) \big] \mathrm{d} s \right|_{\mathcal H_\gamma }^2 \right]	 \nonumber\\
& \lesssim  \mathbb E \left[ \sup_{t \in [0,T]} \left( \int_0^t |t-s|^{-2\alpha}  \big| (S_{s-s(\delta)}-\mathrm{id}) \big[ F_1(X_{s(\delta)}^\varepsilon,\hat{Y}_s^\varepsilon)-\bar{F}_1(X_{s(\delta)}^\varepsilon) \big] \big|_{\mathcal H_{\gamma-2\alpha} } \mathrm{d} s \right)^2 \right]	 \nonumber\\	
& \lesssim \delta^{2\alpha} \mathbb E \left[ \sup_{t \in [0,T]} \left( \int_0^t |t-s|^{-2\alpha}  \big|  F_1(X_{s(\delta)}^\varepsilon,\hat{Y}_s^\varepsilon)-\bar{F}_1(X_{s(\delta)}^\varepsilon) \big|_{\mathcal H_{\gamma-\alpha} } \mathrm{d} s \right)^2 \right] \nonumber\\
& \lesssim \delta^{2\alpha} \mathbb E \left[\left( \|F_1\|_{\infty, \mathcal H_{\gamma-\alpha}}^2+ \|\bar{F}_1\|_{\infty, \mathcal H_{\gamma-\alpha}}^2   \right) \cdot \sup_{t \in [0,T]} \left( \int_0^t |t-s|^{-2\alpha} \mathrm{d} s \right)^2 \right]
\leq C_T \delta^{2\alpha},
\end{align}
where we used (\ref{F1Bound}) and $\|\bar{F}_1\|_{\infty,\mathcal H_{\gamma-\alpha}} \leq \|F_1\|_{\infty,\mathcal H_{\gamma-\alpha}}$ in the last one.
For the second term on the right-hand side of (\ref{F3sup}), it is easy to see that
\begin{align*}
&~~~\mathbb E \left[ \sup_{t \in [0,T]} \left| \int_0^t S_{t-s(\delta)} \big[ F_1(X_{s(\delta)}^\varepsilon,\hat{Y}_s^\varepsilon)-\bar{F}_1(X_{s(\delta)}^\varepsilon) \big]  \mathrm{d} s  \right|_{\mathcal H_\gamma }^2 \right] \nonumber\\
& \leq 2 \mathbb E \left[ \sup_{t \in [0,T]} \left| \int_{t(\delta)}^t S_{t-s(\delta)} \big[ F_1(X_{s(\delta)}^\varepsilon,\hat{Y}_s^\varepsilon)-\bar{F}_1(X_{s(\delta)}^\varepsilon) \big]  \mathrm{d} s  \right|_{\mathcal H_\gamma }^2 \right]	\nonumber\\
&~~~+2 \mathbb E \left[ \sup_{t \in [0,T]} \left| \sum_{k=0}^{[t/\delta]-1} \int_{k\delta}^{(k+1)\delta} S_{t-k\delta} \big[ F_1(X_{k\delta}^\varepsilon,\hat{Y}_s^\varepsilon)-\bar{F}_1(X_{k\delta}^\varepsilon) \big]  \mathrm{d} s  \right|_{\mathcal H_\gamma }^2 \right]  \nonumber\\
&=:\hat{J}_{11}+\hat{J}_{12}.
\end{align*}
Using (\ref{group1}) and H\"{o}lder's inequality we get
\begin{align}\label{hatJ11}
\hat{J}_{11}& \lesssim 	\mathbb E \left[ \sup_{t \in [0,T]} \left( \int_{t(\delta)}^t  |t-s(\delta)|^{-\alpha} \big|F_1(X_{s(\delta)}^\varepsilon,\hat{Y}_s^\varepsilon)-\bar{F}_1(X_{s(\delta)}^\varepsilon) \big|_{\mathcal H_{\gamma-\alpha} } \mathrm{d} s  \right)^2 \right]	\nonumber\\
& \lesssim \mathbb E \left[ \sup_{t \in [0,T]} \left( \int_{t(\delta)}^t |t-s(\delta)|^{-2\alpha} \mathrm{d} s  \int_{t(\delta)}^t  \big| F_1(X_{s(\delta)}^\varepsilon,\hat{Y}_s^\varepsilon)-\bar{F}_1(X_{s(\delta)}^\varepsilon) \big|_{\mathcal H_{\gamma-\alpha} }^2 \mathrm{d} s \right)  \right]	\nonumber\\
& \lesssim \mathbb E \left[ \sup_{t \in [0,T]} \left( \int_{t(\delta)}^t |t-s|^{-2\alpha} \mathrm{d} s  \int_{t(\delta)}^t  \big| F_1(X_{s(\delta)}^\varepsilon,\hat{Y}_s^\varepsilon)-\bar{F}_1(X_{s(\delta)}^\varepsilon) \big|_{\mathcal H_{\gamma-\alpha} }^2 \mathrm{d} s \right)  \right]	\nonumber\\
& \lesssim \delta^{2(1-\alpha)} \left( \|F_1\|_{\infty, \mathcal H_{\gamma-\alpha}}^2+ \|\bar{F}_1\|_{\infty, \mathcal H_{\gamma-\alpha}}^2 \right)
\lesssim \delta^{2(1-\alpha)}.
\end{align}
For $\hat{J}_{12}$, by (\ref{group1}) we obtain
\begin{align}\label{hatJ12}
\hat{J}_{12}& \leq  \frac T\delta	\mathbb E \left[ \sum_{k=0}^{[T/\delta]-1}  \left|S_{t-k\delta}  \int_{k\delta}^{(k+1)\delta}  F_1(X_{k\delta}^\varepsilon,\hat{Y}_s^\varepsilon)-\bar{F}_1(X_{k\delta}^\varepsilon)   \mathrm{d} s  \right|_{\mathcal H_{\gamma} }^2 \right] \nonumber\\
& \leq  \frac {C_T} {\delta^{1+2\alpha}} \sum_{k=0}^{[T/\delta]-1} \mathbb E  \left| \int_{k\delta}^{(k+1)\delta}  F_1(X_{k\delta}^\varepsilon,\hat{Y}_s^\varepsilon)-\bar{F}_1(X_{k\delta}^\varepsilon)   \mathrm{d} s   \right|_{\mathcal H_{\gamma-\alpha} }^2  \nonumber\\
& \leq  \frac {C_T} {\delta^{2(1+\alpha)}}  \max_{0\leq k\leq[T/\delta]-1} \mathbb E  \left| \int_{k\delta}^{(k+1)\delta}  F_1(X_{k\delta}^\varepsilon,\hat{Y}_s^\varepsilon)-\bar{F}_1(X_{k\delta}^\varepsilon)   \mathrm{d} s   \right|_{\mathcal H_{\gamma-\alpha} }^2  \nonumber\\
&= \frac {C_T \varepsilon^2} {\delta^{2(1+\alpha)}}  \max_{0\leq k\leq[T/\delta]-1} \mathbb E  \left| \int_0^{\frac \delta \varepsilon}  F_1(X_{k\delta}^\varepsilon,\hat{Y}_{s\varepsilon+k\delta}^\varepsilon)-\bar{F}_1(X_{k\delta}^\varepsilon)   \mathrm{d} s   \right|_{\mathcal H_{\gamma-\alpha} }^2  \nonumber\\
& \leq \frac {C_T \varepsilon^2} {\delta^{2(1+\alpha)}}  \max_{0\leq k\leq[T/\delta]-1} \int_0^{\frac \delta \varepsilon} \int_r^{\frac \delta \varepsilon} \Phi_k(s,r) \mathrm{d} s \mathrm{d} r,
\end{align}
where for any $0 \leq r \leq s \leq \frac \delta \varepsilon$,
\begin{equation}\label{phi}
	\Phi_k(s,r):=\mathbb E \left[ \langle F_1(X_{k\delta}^\varepsilon,\hat{Y}_{s\varepsilon+k\delta}^\varepsilon)-\bar{F}_1(X_{k\delta}^\varepsilon),F_1(X_{k\delta}^\varepsilon,\hat{Y}_{r\varepsilon+k\delta}^\varepsilon)-\bar{F}_1(X_{k\delta}^\varepsilon) \rangle_{\mathcal H_{\gamma-\alpha}} \right].
\end{equation}
According to Lemma \ref{phiest} one has
\begin{equation}\label{phires}
\Phi_k(s,r) \leq C (1+|x|_{\mathcal H_\gamma}^2+| y |_{\mathcal H_\gamma }^2) e^{-(\lambda_1-L_{F_2}-L_{G_2}^2)(s-r)}.
\end{equation}
Substituting (\ref{phires}) into (\ref{hatJ12}) we get
\begin{equation}\label{J12est}
\hat{J}_{12} \leq C_T (1+|x|_{\mathcal H_\gamma}^2+| y |_{\mathcal H_\gamma }^2) \left(\frac {\varepsilon} {\delta^{1+2\alpha}} +\frac {\varepsilon^2} {\delta^{2(1+\alpha)}}  \right).	
\end{equation}
Based on the above estimates (\ref{Fsup1})--(\ref{J12est}), we can conclude that
\begin{equation}\label{J1res}
\hat{J}_1 \leq C_T(1+|x|_{\mathcal H_\gamma}^2+| y |_{\mathcal H_\gamma }^2) \left(\delta^{2\alpha}+\delta^{2(1-\alpha)}+\frac {\varepsilon} {\delta^{1+2\alpha}} +\frac {\varepsilon^2} {\delta^{2(1+\alpha)}}  \right).	
\end{equation}
For $\hat{J}_2$, we have
\begin{align*}
\hat{J}_2 & \leq C \mathbb E \left[ \sup_{s<t} |t-s|^{-2\alpha} \left| \int_s^t S_{t-r} \big[ F_1(X_{r(\delta)}^\varepsilon,\hat{Y}_r^\varepsilon)-\bar{F}_1(X_{r(\delta)}^\varepsilon) \big] \mathrm{d} r \right|_{\mathcal H_{\gamma-\alpha} }^2 \right]	 \nonumber\\
&~~~+ C \mathbb E \left[ \sup_{s<t} |t-s|^{-2\alpha} \left| \int_0^s (S_{t-s}-\mathrm{id}) S_{s-r}\big[ F_1(X_{r(\delta)}^\varepsilon,\hat{Y}_r^\varepsilon)-\bar{F}_1(X_{r(\delta)}^\varepsilon) \big] \mathrm{d} r \right|_{\mathcal H_{\gamma-\alpha} }^2 \right] \nonumber\\
&=: \hat{J}_{21}+\hat{J}_{22}.
\end{align*}
For $\hat{J}_{22}$, using (\ref{group2}) and estimate of $\hat{J}_1$ we have
\begin{align}\label{hatJ22}
	\hat{J}_{22}&=C \mathbb E \left[ \sup_{s<t} |t-s|^{-2\alpha} \left| (S_{t-s}-\mathrm{id}) \int_0^s S_{s-r} \big[ F_1(X_{r(\delta)}^\varepsilon,\hat{Y}_r^\varepsilon)-\bar{F}_1(X_{r(\delta)}^\varepsilon) \big] \mathrm{d} r \right|_{\mathcal H_{\gamma-\alpha} }^2 \right] \nonumber\\
	& \leq  C \mathbb E \left[ \sup_{s<t} |t-s|^{-2\alpha}  \left( |t-s|^\alpha \left| \int_0^s S_{s-r} \big[ F_1(X_{r(\delta)}^\varepsilon,\hat{Y}_r^\varepsilon)-\bar{F}_1(X_{r(\delta)}^\varepsilon) \big] \mathrm{d} r \right|_{\mathcal H_\gamma } \right)^2 \right] \nonumber\\
	&\leq  C \mathbb E \left[ \sup_{s \in [0,T]}  \left| \int_0^s S_{s-r} \big[ F_1(X_{r(\delta)}^\varepsilon,\hat{Y}_r^\varepsilon)-\bar{F}_1(X_{r(\delta)}^\varepsilon) \big] \mathrm{d} r \right|_{\mathcal H_\gamma }^2 \right] \nonumber\\
&\leq C_T(1+|x|_{\mathcal H_\gamma}^2+| y |_{\mathcal H_\gamma }^2) \left(\delta^{2\alpha}+\delta^{2(1-\alpha)}+\frac {\varepsilon} {\delta^{1+2\alpha}} +\frac {\varepsilon^2} {\delta^{2(1+\alpha)}}  \right).
\end{align}
For $\hat{J}_{21}$, we have
\begin{align*}
	\hat{J}_{21}
	&\leq C \mathbb E \left[ \sup_{s<t} |t-s|^{-2\alpha}
	\left| \int_s^t \big( S_{t-r}-S_{t-r(\delta)} \big) \big[F_1(X_{r(\delta)}^\varepsilon,\hat{Y}_r^\varepsilon)-\bar{F}_1(X_{r(\delta)}^\varepsilon) \big] \mathrm{d} r \right|_{\mathcal H_{\gamma-\alpha} }^2 \right] \nonumber\\
	&~~~+ C\mathbb E \left[ \sup_{s<t} |t-s|^{-2\alpha}
	\left| \int_s^t  S_{t-r(\delta)}  \big[F_1(X_{r(\delta)}^\varepsilon,\hat{Y}_r^\varepsilon)-\bar{F}_1(X_{r(\delta)}^\varepsilon) \big] \mathrm{d} r \right|_{\mathcal H_{\gamma-\alpha} }^2 \right] \nonumber\\
	&=: \hat{J}_{211}+\hat{J}_{212}.
\end{align*}
For $\hat{J}_{211}$, by H\"{o}lder's inequality, (\ref{group1}), (\ref{group2}), (\ref{F1Bound}) and $\|\bar{F}_1\|_{\infty,\mathcal H_{\gamma-\alpha}} \leq \|F_1\|_{\infty,\mathcal H_{\gamma-\alpha}}$ we obtain
\begin{align}\label{K2eta1}
	\hat{J}_{211} & \leq C\mathbb E \left [ \sup_{s<t} |t-s|^{-2\alpha}
	 \left( \int_s^t  \left| S_{t-r}(\mathrm{id}-S_{r-r(\delta)})  \big[F_1(X_{r(\delta)}^\varepsilon,\hat{Y}_r^\varepsilon)-\bar{F}_1(X_{r(\delta)}^\varepsilon) \big] \right|_{\mathcal H_{\gamma-\alpha} } \mathrm{d} r \right)^2  \right ] \nonumber\\
	&\leq C \mathbb E \left [ \sup_{s<t} |t-s|^{-2\alpha} \left(
	\int_s^t |t-r|^{-\alpha}  \left| (S_{r-r(\delta)}-\mathrm{id}) \big[F_1(X_{r(\delta)}^\varepsilon,\hat{Y}_r^\varepsilon)-\bar{F}_1(X_{r(\delta)}^\varepsilon) \big] \right|_{\mathcal H_{\gamma-2\alpha} } \mathrm{d} r \right)^2 \right ] \nonumber\\
	&\leq C \mathbb E \left [ \sup_{s<t} |t-s|^{-2\alpha}
	\left( \int_s^t |t-r|^{-\alpha}  |r-r(\delta)|^\alpha  \big| F_1(X_{r(\delta)}^\varepsilon,\hat{Y}_r^\varepsilon)-\bar{F}_1(X_{r(\delta)}^\varepsilon) \big|_{\mathcal H_{\gamma-\alpha} } \mathrm{d} r \right)^2 \right ] \nonumber\\
	&\leq C \delta^{2\alpha} \mathbb E \left [ \left( \|F_1\|_{\infty, \mathcal H_{\gamma-\alpha}}^2+ \|\bar{F}_1\|_{\infty, \mathcal H_{\gamma-\alpha}}^2 \right)
	\sup_{s<t} |t-s|^{-2\alpha}
	\left(\int_s^t |t-r|^{-\alpha}  \mathrm{d} r\right)^2  \right ] \nonumber\\
	&\leq C T^{2(1-2\alpha)} \delta^{2\alpha}.
\end{align}
For $\hat{J}_{212}$, we note that
\begin{align*}
\hat{J}_{212}&\leq C\mathbb E \left [ \sup_{s<t} |t-s|^{-2\alpha}
	\left| \int_s^t  S_{t-r(\delta)}  \big[F_1(X_{r(\delta)}^\varepsilon,\hat{Y}_r^\varepsilon)-\bar{F}_1(X_{r(\delta)}^\varepsilon) \big]\mathrm{d} r \right|_{\mathcal H_{\gamma-\alpha} }^2 1_{\{|t-s|\leq \delta \}}\right ] \nonumber\\
	&~~~+ C\mathbb E \left [ \sup_{s<t} |t-s|^{-2\alpha}
	\left| \int_s^t  S_{t-r(\delta)}  \big[F_1(X_{r(\delta)}^\varepsilon,\hat{Y}_r^\varepsilon)-\bar{F}_1(X_{r(\delta)}^\varepsilon) \big]\mathrm{d} r \right|_{\mathcal H_{\gamma-\alpha} }^2 1_{ \{|t-s|>\delta \}} \right ] \nonumber\\
	&\leq C \mathbb E \left [ \sup_{s<t} |t-s|^{-2\alpha}
	\left| \int_s^t  S_{t-r(\delta)}  \big[F_1(X_{r(\delta)}^\varepsilon,\hat{Y}_r^\varepsilon)-\bar{F}_1(X_{r(\delta)}^\varepsilon) \big]\mathrm{d} r \right|_{\mathcal H_{\gamma-\alpha} }^2 1_{\{|t-s|\leq \delta \}}\right ] \nonumber\\
	&~~~+ C \mathbb E \left [ \sup_{s<t} |t-s|^{-2\alpha}
	\left| \int_s^{([s/\delta]+1)\delta}  S_{t-r(\delta)}  \big[F_1(X_{r(\delta)}^\varepsilon,\hat{Y}_r^\varepsilon)-\bar{F}_1(X_{r(\delta)}^\varepsilon) \big]\mathrm{d} r \right|_{\mathcal H_{\gamma-\alpha} }^2 1_{ \{|t-s|>\delta \}} \right ] \nonumber\\
	&~~~+ C \mathbb E \left [ \sup_{s<t} |t-s|^{-2\alpha}
	\left| \int_{t(\delta)}^t  S_{t-r(\delta)}  \big[F_1(X_{r(\delta)}^\varepsilon,\hat{Y}_r^\varepsilon)-\bar{F}_1(X_{r(\delta)}^\varepsilon) \big]\mathrm{d} r \right|_{\mathcal H_{\gamma-\alpha} }^2 1_{ \{|t-s|>\delta \}} \right ] \nonumber\\
	&~~~+ C \mathbb E \left[ \sup_{s<t} |t-s|^{-2\alpha}
	\left| \sum_{k=[s/\delta]+1}^{[t/\delta]-1} \int_{k\delta}^{(k+1)\delta} S_{t-k\delta} \big[F_1(X_{k\delta}^\varepsilon,\hat{Y}_r^\varepsilon)-\bar{F}_1(X_{k\delta}^\varepsilon) \big]\mathrm{d} r \right|_{\mathcal H_{\gamma-\alpha} }^2 1_{ \{|t-s|>\delta \}} \right ] \nonumber\\
	&=:\hat{K}_1+\hat{K}_2+\hat{K}_3+\hat{K}_4.
\end{align*}
Firstly, using (\ref{F1Bound}) and $\|\bar{F}_1\|_{\infty,\mathcal H_{\gamma-\alpha}} \leq \|F_1\|_{\infty,\mathcal H_{\gamma-\alpha}}$, the term $\hat{K}_1$ is controlled by
\begin{align}\label{hatK1}
	\hat{K}_1 &\leq  C \mathbb E \left[ \sup_{s<t} |t-s|^{-2\alpha}
	\left( \int_s^t \left| S_{t-r(\delta)}  \big[ F_1(X_{r(\delta)}^\varepsilon,\hat{Y}_r^\varepsilon)-\bar{F}_1(X_{r(\delta)}^\varepsilon) \big] \right |_{\mathcal H_{\gamma-\alpha} } \mathrm{d} r  \right)^2  1_{\{|t-s|\leq \delta \}} \right] \nonumber\\
	&\leq  C \mathbb E \left[ \sup_{s<t} |t-s|^{-2\alpha}
	\left( \int_s^t \big| F_1(X_{r(\delta)}^\varepsilon,\hat{Y}_r^\varepsilon)-\bar{F}_1(X_{r(\delta)}^\varepsilon) \big|_{\mathcal H_{\gamma-\alpha} } \mathrm{d} r  \right)^2  1_{\{|t-s|\leq \delta \}} \right] \nonumber\\
	&\leq  C \mathbb E \left[ \left( \|F_1\|_{\infty, \mathcal H_{\gamma-\alpha}}^2+ \|\bar{F}_1\|_{\infty, \mathcal H_{\gamma-\alpha}}^2 \right)   \sup_{s<t} |t-s|^{2(1-\alpha)}  1_{\{|t-s|\leq \delta \}} \right]
	\leq C \delta^{2(1-\alpha)} .	
\end{align}
Secondly, using the boundedness of $F_1$ and $\bar{F}_1$ we have
\begin{align}\label{hatK2}
	\hat{K}_2 & \leq  C \mathbb E \left[ \sup_{s<t} \delta^{-2\alpha}
	\left( \int_s^{([s/\delta]+1)\delta}  \left| S_{t-r(\delta)}  \big[F_1(X_{r(\delta)}^\varepsilon,\hat{Y}_r^\varepsilon)-\bar{F}_1(X_{r(\delta)}^\varepsilon) \big] \right|_{\mathcal H_{\gamma-\alpha} } \mathrm{d} r \right)^2 1_{ \{|t-s|>\delta \}} \right] \nonumber\\
	& \leq C \mathbb E \left[ \sup_{s<t} \delta^{-2\alpha}
	\left( \int_s^{([s/\delta]+1)\delta}    \big|F_1(X_{r(\delta)}^\varepsilon,\hat{Y}_r^\varepsilon)-\bar{F}_1(X_{r(\delta)}^\varepsilon) \big|_{\mathcal H_{\gamma-\alpha} } \mathrm{d} r \right)^2 1_{ \{|t-s|>\delta \}} \right] \nonumber\\
	& \leq C \delta^{2(1-\alpha)}  \left(\|F_1\|_{\infty, \mathcal H_{\gamma-\alpha}}^2+ \|\bar{F}_1\|_{\infty, \mathcal H_{\gamma-\alpha}}^2\right)
	\leq C \delta^{2(1-\alpha)}.
\end{align}
For $\hat{K}_3$, similar to the estimation of $\hat{K}_2$, we have
\begin{equation}\label{hatK3}
\hat{K}_3 \leq C \delta^{2(1-\alpha)}.
\end{equation}
Finally, we estimate the term $\hat{K}_4$.
\begin{align}\label{Psi}
\hat{K}_4
	&\leq C \delta^{-2\alpha}\mathbb E\Bigg[\sup_{s<t} \left(\Big [\frac t\delta \Big]-\Big [\frac s\delta \Big ]-1 \right) \nonumber\\
	&~~~~~~~~~~~~~~\times \sum_{k=[s/\delta]+1}^{[t/\delta]-1} \left|  \int_{k\delta}^{(k+1)\delta} S_{t-k\delta} \big[F_1(X_{k\delta}^\varepsilon,\hat{Y}_r^\varepsilon)-\bar{F}_1(X_{k\delta}^\varepsilon) \big]\mathrm{d} r \right|_{\mathcal H_{\gamma-\alpha} }^2 1_{ \{|t-s|>\delta \}} \Bigg ] \nonumber\\
	&\leq  \frac {C_T}{\delta^{1+2\alpha}} \mathbb E \left [\sum_{k=0}^{[T/\delta]-1} \left|  \int_{k\delta}^{(k+1)\delta} S_{t-k\delta} \big [F_1(X_{k\delta}^\varepsilon,\hat{Y}_r^\varepsilon)-\bar{F}_1(X_{k\delta}^\varepsilon) \big]\mathrm{d} r \right|_{\mathcal H_{\gamma-\alpha} }^2 1_{ \{|t-s|>\delta \}} \right ] \nonumber\\
	&\leq \frac {C_T}{\delta^{2(1+\alpha)}} \max_{0\leq k\leq[T/\delta]-1}\mathbb E \left| S_{t-k\delta} \int_{k\delta}^{(k+1)\delta}  F_1(X_{k\delta}^\varepsilon,\hat{Y}_r^\varepsilon)-\bar{F}_1(X_{k\delta}^\varepsilon)  \mathrm{d} r \right|_{\mathcal H_{\gamma-\alpha} }^2  \nonumber\\
	&\leq \frac {C_T}{\delta^{2(1+\alpha)}} \max_{0\leq k\leq[T/\delta]-1}\mathbb E\left| \int_{k\delta}^{(k+1)\delta} F_1(X_{k\delta}^\varepsilon,\hat{Y}_r^\varepsilon)-\bar{F}_1(X_{k\delta}^\varepsilon)  \mathrm{d} r \right|_{\mathcal H_{\gamma-\alpha} }^2 \nonumber\\
	&\leq \frac {C_T\varepsilon^2}{\delta^{2(1+\alpha)}} \max_{0\leq k\leq[T/\delta]-1}
	\mathbb E \left|\int_0^{\frac \delta\varepsilon} F_1(X_{k\delta}^\varepsilon,\hat{Y}_{r\varepsilon+k\delta}^\varepsilon)-\bar{F}_1(X_{k\delta}^\varepsilon) \mathrm{d} r \right|_{\mathcal H_{\gamma-\alpha} }^2  \nonumber\\
	&\leq \frac {C_T\varepsilon^2}{\delta^{2(1+\alpha)}} \max_{0\leq k\leq[T/\delta]-1}
	\int_0^{\frac \delta\varepsilon}\int_r^{\frac \delta\varepsilon}\Psi_k(s,r) \mathrm{d} s \mathrm{d} r \nonumber\\
& \leq C_T (1+|x|_{\mathcal H_\gamma}^2+| y |_{\mathcal H_\gamma }^2) \left(\frac {\varepsilon} {\delta^{1+2\alpha}} +\frac {\varepsilon^2} {\delta^{2(1+\alpha)}}  \right),
\end{align}
where $\Psi_k(s,r)$ is defined as in (\ref{phi}). From estimates (\ref{hatJ22})--(\ref{Psi}), it follows that
\begin{equation}\label{hatJ2}
\hat{J}_2\leq C_T(1+|x|_{\mathcal H_\gamma}^2+| y |_{\mathcal H_\gamma }^2) \left(\delta^{2\alpha}+\delta^{2(1-\alpha)}+\frac {\varepsilon} {\delta^{1+2\alpha}} +\frac {\varepsilon^2} {\delta^{2(1+\alpha)}}  \right).
\end{equation}
For $\hat{J}_3$, an estimate similar to that of $\hat{J}_2$ yields
\begin{equation}\label{hatJ3}
\hat{J}_3\leq C_T(1+|x|_{\mathcal H_\gamma}^2+| y |_{\mathcal H_\gamma }^2) \left(\delta^{2\alpha}+\delta^{2(1-\alpha)}+\delta^{2(1-2\alpha)}+\frac {\varepsilon} {\delta^{1+2\alpha}} +\frac {\varepsilon^2} {\delta^{2(1+\alpha)}}+\frac {\varepsilon} {\delta^{1+4\alpha}} +\frac {\varepsilon^2} {\delta^{2(1+2\alpha)}} \right).
\end{equation}
Therefore, according to (\ref{J1res}), (\ref{hatJ2}) and (\ref{hatJ3}), we obtain
\begin{align}\label{dif3}
&\mathbb E \left\| \displaystyle\int_0^\cdot S_{\cdot-s} \big[F_1(X_{s(\delta)}^\varepsilon,\hat{Y}_s^\varepsilon)-\bar{F}_1(X_{s(\delta)}^\varepsilon) \big]\mathrm{d} s,0 \right\|_{\mathcal{D}_{B,\gamma}^{2\alpha}}^2 \nonumber\\
\leq & C_T(1+|x|_{\mathcal H_\gamma}^2+| y |_{\mathcal H_\gamma }^2) \left(\delta^{2\alpha}+\delta^{2(1-\alpha)}+\delta^{2(1-2\alpha)}+\frac {\varepsilon} {\delta^{1+2\alpha}} +\frac {\varepsilon^2} {\delta^{2(1+\alpha)}}+\frac {\varepsilon} {\delta^{1+4\alpha}} +\frac {\varepsilon^2} {\delta^{2(1+2\alpha)}} \right).
\end{align}

\textbf{Step4:} Finally, substituting (\ref{dif1}), (\ref{dif2}) and (\ref{dif3}) into (\ref{EXdif}) and letting $\delta:=\varepsilon^{\frac 1{2(1+2\alpha)}}$, we obtain
$$\lim_{\varepsilon \to 0} \mathbb E \big[\|X^\varepsilon-\bar{X}\|_{\infty, \mathcal H_\gamma}^2 \big]=0.$$
The proof is completed. \hfill $\square$

\section{Appendix}\label{appendix}

\appendix
\renewcommand{\thelemma}{A.\arabic{lemma}} 
\setcounter{lemma}{0}
\renewcommand{\theequation}{A.\arabic{equation}} 
\setcounter{equation}{0} 

This Appendix serves two purposes: First, to provide a proof of Lemma \ref{stab}, and second, to derive an estimate for $\Phi_k(s,r)$, which is used in the proof of Theorem \ref{MainRes}.

~~~~~~~~~~~~~~~~~~~~\\
\textbf{Proof of Lemma \ref{stab}}.	
By (\ref{norm}) we have
\begin{align*}
	&\left\|\int_0^\cdot S_{\cdot-s} Y_s \mathrm{d} \mathbf X_s, Y \right\|_{\mathcal{D}_{X,\gamma+\sigma}^{2\alpha}} \nonumber\\
	=&\left\|\int_0^\cdot S_{\cdot-s} Y_s \mathrm{d} \mathbf X_s \right\|_{\infty, \mathcal{H}_{\gamma+\sigma}}+\|Y\|_{\infty,\mathcal H_{\gamma+\sigma-\alpha}^{d}}
	+|Y|_{\alpha,\mathcal H_{\gamma+\sigma-2\alpha}^{d}} \nonumber\\
	&+ \left|R^{\int_0^\cdot S_{\cdot-s} Y_s \mathrm{d} \mathbf X_s} \right|_{\alpha,\mathcal H_{\gamma+\sigma-\alpha}}
	+\left|R^{\int_0^\cdot S_{\cdot-s} Y_s \mathrm{d} \mathbf X_s} \right|_{2\alpha,\mathcal H_{\gamma+\sigma-2\alpha}} .
\end{align*}
For $\|Y\|_{\infty,\mathcal H_{\gamma+\sigma-\alpha}^{d}}$, by  $\mathcal H_{\gamma}^{d} \hookrightarrow \mathcal H_{\gamma+\sigma-\alpha}^{d} $ we have
\begin{align}\label{Ysup}
	\|Y\|_{\infty,\mathcal H_{\gamma+\sigma-\alpha}^{d}}
	\leq |Y_0|_{\mathcal H_\gamma^d}+T^{\alpha-\sigma} |Y|_{\alpha-\sigma,\mathcal H_{\gamma+\sigma-\alpha}^{d}}.
\end{align}
For $|Y|_{\alpha-\sigma,\mathcal H_{\gamma+\sigma-\alpha}^{d}}$,
using the interpolation inequality (\ref{interp}) and Young's inequality we obtain
\begin{align}\label{Galpsig}
	&|Y|_{\alpha-\sigma,\mathcal H_{\gamma+\sigma-\alpha}^{d}}
	=\sup_{0\leq s<t \leq T}  \frac{\left|Y_t-Y_s \right|_{\mathcal H_{\gamma+\sigma-\alpha}^{d} }}{|t-s|^{\alpha-\sigma}} \nonumber\\
	\lesssim &  \|Y\|_{\infty, \mathcal H_{\gamma}^{d}}^{\frac \sigma \alpha} |Y|_{\alpha, \mathcal H_{\gamma-\alpha}^{d}}^{\frac {\alpha-\sigma} \alpha}
	\lesssim  \|Y\|_{\infty, \mathcal H_{\gamma}^{d}}+|Y|_{\alpha, \mathcal H_{\gamma-\alpha}^{d}},
\end{align}
where for the first inequality we used the following interpolation inequality: $$|Y_t-Y_s|_{\mathcal H_{\gamma+\sigma-\alpha}^{d} } \lesssim |Y_t-Y_s|_{\mathcal H_{\gamma}^{d}}^{\frac \sigma \alpha} |Y_t-Y_s|_{\mathcal H_{\gamma-\alpha}^{d}}^{\frac {\alpha-\sigma} \alpha}.$$
For $|Y|_{\alpha, \mathcal H_{\gamma-\alpha}^{d} }$, by (\ref{Y}) we have
\begin{align}\label{Yalp}
	|Y|_{\alpha, \mathcal H_{\gamma-\alpha}^{d}}
	\leq & |X|_\alpha \|Y^\prime\|_{\infty, \mathcal H_{\gamma-\alpha}^{d \times d} }	
	+|R^{Y} |_{\alpha, \mathcal H_{\gamma-\alpha}^{d} }.
\end{align}
Combined with estimates (\ref{Ysup})--(\ref{Yalp}) we have
\begin{equation}\label{Gsup1}
	\|Y\|_{\infty,\mathcal H_{\gamma+\sigma-\alpha}^{d}}
	\lesssim |Y_0|_{\mathcal H_\gamma^d}+T^{\alpha-\sigma} (1+\varrho_{\alpha}(\mathbf{X})) \|Y,Y^\prime\|_{\mathcal{D}_{X,\gamma}^{2\alpha}}.
\end{equation}

For $|Y|_{\alpha,\mathcal H_{\gamma+\sigma-2\alpha}^{d} }$,
following similar calculations as in (\ref{Yalp}) we have
\begin{equation*}
	|Y|_{\alpha,\mathcal H_{\gamma+\sigma-2\alpha}^{d}}
	\leq |X|_\alpha \|Y^\prime\|_{\infty, \mathcal H_{\gamma+\sigma-2\alpha}^{d \times d} }+|R^{Y}|_{\alpha, \mathcal H_{\gamma+\sigma-2\alpha}^{d}}.
\end{equation*}
Using the interpolation inequality (\ref{interp}) and $\mathcal H_{\gamma-\alpha}^{d \times d} \hookrightarrow \mathcal H_{\gamma+\sigma-2\alpha}^{d \times d}$ we have
\begin{align}\label{DGGsup}
	&\|Y^\prime\|_{\infty, \mathcal H_{\gamma+\sigma-2\alpha}^{d \times d}}
	\leq  |Y_0^\prime|_{\mathcal H_{\gamma+\sigma-2\alpha}^{d \times d}} + T^{\alpha-\sigma}  |Y^\prime |_{\alpha-\sigma,\mathcal H_{\gamma+\sigma-2\alpha}^{d \times d}}\nonumber\\
	\lesssim & |Y_0^\prime|_{\mathcal H_{\gamma-\alpha}^{d \times d}}+T^{\alpha-\sigma} \|Y^\prime\|_{\infty,\mathcal H_{\gamma-\alpha}^{d \times d} }^{\frac \sigma \alpha} |Y^\prime |_{\alpha,\mathcal H_{\gamma-2\alpha}^{d \times d} }^{\frac {\alpha-\sigma} \alpha},
\end{align}
and
\begin{equation*}
	|R^{Y}|_{\alpha, \mathcal H_{\gamma+\sigma-2\alpha}^{d}}
	\lesssim T^{\alpha-\sigma} |R^{Y}|_{\alpha, \mathcal H_{\gamma-\alpha}^{d} }^{\frac \sigma \alpha}
	|R^{Y}|_{2\alpha, \mathcal H_{\gamma-2\alpha}^{d} }^{\frac {\alpha-\sigma} \alpha}.
\end{equation*}
Thus, using Young's inequality we obtain
\begin{equation}\label{G1alp}
	|Y|_{\alpha,\mathcal H_{\gamma+\sigma-2\alpha}^{d}}
	\lesssim \varrho_{\alpha}(\mathbf{X})|Y_0^\prime|_{\mathcal H_{\gamma-\alpha}^{d \times d}}+T^{\alpha-\sigma}  (1+\varrho_{\alpha}(\mathbf{X})) \|Y,Y^\prime\|_{\mathcal{D}_{X,\gamma}^{2\alpha}}.
\end{equation}

For $ \left\|\int_0^\cdot S_{\cdot-s} Y_s \mathrm{d} \mathbf X_s \right\|_{\infty, \mathcal{H}_{\gamma+\sigma}}$, we note that
\begin{align*}
	\left| \int_0^t S_{t-s} Y_s \mathrm{d} \mathbf X_s \right|_{\mathcal{H}_{\gamma+\sigma}}
	\leq &	\left|\int_0^t S_{t-s} Y_s \mathrm{d} \mathbf X_s-S_t Y_0 X_{0,t}-S_t Y_0^\prime \mathbb{X}_{0,t} \right|_{\mathcal{H}_{\gamma+\sigma}} \nonumber\\
	&+|S_t Y_0 X_{0,t}|_{\mathcal{H}_{\gamma+\sigma}}+ |S_t Y_0^\prime \mathbb{X}_{0,t}|_{\mathcal{H}_{\gamma+\sigma}}.
\end{align*}
Using (\ref{IntBound}) we have
\begin{align*}
	&\left| \int_0^t S_{t-s} Y_s \mathrm{d} \mathbf X_s-S_t Y_0 X_{0,t}-S_t Y_0^\prime \mathbb{X}_{0,t} \right|_{\mathcal{H}_{\gamma+\sigma}} \\
	\lesssim & \varrho_{\alpha}(\mathbf{X}) \|Y,Y^\prime\|_{\mathcal{D}_{X,\gamma}^{2\alpha}}  t^{\alpha-\sigma}.
\end{align*}
By (\ref{group1}) we get
\begin{align*}
	|S_t Y_0 X_{0,t}|_{\mathcal{H}_{\gamma+\sigma}} \leq t^\alpha |X|_\alpha |S_t Y_0|_{\mathcal H_{\gamma+\sigma}^{d}}
	\lesssim t^{\alpha-\sigma} |X|_\alpha |Y_0|_{\mathcal H_{\gamma}^{d}}.
\end{align*}
Similarly, using (\ref{group1}) we have
\begin{align*}
	|S_t Y_0^\prime \mathbb{X}_{0,t}|_{\mathcal{H}_{\gamma+\sigma}} \leq  t^{2\alpha} |\mathbb{X}|_{2\alpha} |S_t Y_0^\prime|_{\mathcal H_{\gamma+\sigma}^{d \times d}}
	\lesssim  t^{\alpha-\sigma} |\mathbb{X}|_{2\alpha} |Y_0^\prime|_{\mathcal H_{\gamma-\alpha}^{d \times d}} .
\end{align*}
Therefore, we have
\begin{align}\label{IntGsup}
	\left\|\int_0^\cdot S_{\cdot-s} Y_s \mathrm{d} \mathbf X_s \right\|_{\infty, \mathcal{H}_{\gamma+\sigma}}
	\lesssim  T^{\alpha-\sigma}\varrho_{\alpha}(\mathbf{X}) \left(|Y_0|_{\mathcal H_\gamma^d}+|Y_0^\prime|_{\mathcal H_{\gamma-\alpha}^{d \times d}} \right)
	+T^{\alpha-\sigma} \varrho_{\alpha}(\mathbf{X}) \|Y,Y^\prime\|_{\mathcal{D}_{X,\gamma}^{2\alpha}}.	
\end{align}

For the remainder $\left|R^{\int_0^\cdot S_{\cdot-s} Y_s \mathrm{d} \mathbf X_s} \right|_{\theta,\mathcal H_{\gamma+\sigma-\theta}}$ with $\theta \in \{\alpha,2\alpha \}$, we note that
\begin{align*}
	R_{s,t}^{\int_0^\cdot S_{\cdot-s} Y_s \mathrm{d} \mathbf X_s}
	=& \int_0^t S_{t-r} Y_r \mathrm{d} \mathbf X_r-\int_0^s S_{s-r} Y_r \mathrm{d} \mathbf X_r-Y_s X_{s,t} \nonumber\\
	=& \int_s^t S_{t-r} Y_r \mathrm{d} \mathbf X_r -S_{t-s} Y_s X_{s,t} -S_{t-s} Y_s^\prime \mathbb{X}_{s,t} \nonumber\\
	&+ (S_{t-s}-\mathrm{id} ) \int_0^s S_{s-r} Y_r \mathrm{d} \mathbf X_r
	+(S_{t-s}-\mathrm{id} ) Y_s X_{s,t} \nonumber\\
	&+S_{t-s} Y_s^\prime \mathbb{X}_{s,t} \nonumber\\
	=:& \hat{L}_{s,t}^1+\hat{L}_{s,t}^2+\hat{L}_{s,t}^3+\hat{L}_{s,t}^4.
\end{align*}
For $\hat{L}_{s,t}^1$, using (\ref{IntBound}) we have
\begin{align}\label{I1}
	|\hat{L}_{s,t}^1|_{\mathcal H_{\gamma+\sigma-\theta}} \lesssim  \varrho_{\alpha}(\mathbf{X}) \|Y,Y^\prime\|_{\mathcal{D}_{X,\gamma}^{2\alpha}} |t-s|^{\alpha-\sigma+\theta}.
\end{align}	
For $\hat{L}_{s,t}^2$, (\ref{group2}) and (\ref{IntGsup}) imply that
\begin{align}\label{I2}
|\hat{L}_{s,t}^2|_{\mathcal H_{\gamma+\sigma-\theta}}=& \left| (S_{t-s}-\mathrm{id} ) \int_0^s S_{s-r} Y_r \mathrm{d} \mathbf X_r\right|_{\mathcal H_{\gamma+\sigma-\theta}}
 \lesssim  (t-s)^\theta \left|\int_0^s S_{s-r} Y_r \mathrm{d} \mathbf X_r \right|_{\mathcal{H}_{\gamma+\sigma}} \nonumber\\
	\lesssim & (t-s)^\theta \varrho_{\alpha}(\mathbf{X}) T^{\alpha-\sigma}\left(|Y_0|_{\mathcal H_\gamma^d}+|Y_0^\prime|_{\mathcal H_{\gamma-\alpha}^{d \times d}}+\|Y,Y^\prime\|_{\mathcal{D}_{X,\gamma}^{2\alpha}} \right).
\end{align}
For $\hat{L}_{s,t}^3$, by (\ref{group2}) we have
\begin{align}\label{I3}
	|\hat{L}_{s,t}^3|_{\mathcal H_{\gamma+\sigma-\theta}} \leq & |X|_\alpha |t-s|^\alpha |(S_{t-s}-\mathrm{id} ) Y_s|_{\mathcal H_{\gamma+\sigma-\theta}^{d}  }
	\lesssim |t-s|^{\theta+\alpha-\sigma} |X|_\alpha |Y_s|_{\mathcal H_{\gamma}^{d}} \nonumber\\
	\lesssim & |t-s|^{\theta+\alpha-\sigma} |X|_\alpha \|Y\|_{\infty,\mathcal H_{\gamma}^{d}}
	\lesssim    |t-s|^\theta  T^{\alpha-\sigma} \varrho_{\alpha}(\mathbf{X})  \|Y,Y^\prime\|_{\mathcal{D}_{X,\gamma}^{2\alpha}} .
\end{align}
For $\hat{L}_{s,t}^4$, according to (\ref{group1}), (\ref{DGGsup}) and Young's inequality, we obtain
\begin{align}\label{I4}
	|\hat{L}_{s,t}^4|_{\mathcal H_{\gamma+\sigma-\theta}} \lesssim & |\mathbb{X}|_{2\alpha} |t-s|^\theta |Y_s^\prime|_{\mathcal H_{\gamma+\sigma-2\alpha}^{d \times d}}
	\lesssim |t-s|^\theta |\mathbb{X}|_{2\alpha} \|Y^\prime \|_{\infty,\mathcal H_{\gamma+\sigma-2\alpha}^{d \times d}} \nonumber\\
	\lesssim & |t-s|^\theta |\mathbb{X}|_{2\alpha} \left(|Y_0^\prime|_{\mathcal H_{\gamma-\alpha}^{d \times d}}+T^{\alpha-\sigma}\|Y,Y^\prime\|_{\mathcal{D}_{X,\gamma}^{2\alpha}} \right).
\end{align}
Combining the estimates of (\ref{I1})--(\ref{I4}) we get
\begin{align}\label{Intrem}
	\left|R^{\int_0^\cdot S_{\cdot-s} Y_s \mathrm{d} \mathbf X_s} \right|_{\theta,\mathcal H_{\gamma+\sigma-\theta}}
	\lesssim \varrho_{\alpha}(\mathbf{X}) \left(|Y_0|_{\mathcal H_\gamma^d}+|Y_0^\prime|_{\mathcal H_{\gamma-\alpha}^{d \times d}} \right)
	+ \varrho_{\alpha}(\mathbf{X}) T^{\alpha-\sigma} \|Y,Y^\prime\|_{\mathcal{D}_{X,\gamma}^{2\alpha}}.
\end{align}
By using estimates (\ref{Gsup1}), (\ref{G1alp}), (\ref{IntGsup}) and (\ref{Intrem}),
we have
\begin{align}\label{BInt}
	&\left\|\int_0^\cdot S_{\cdot-s} Y_s \mathrm{d} \mathbf X_s, Y \right\|_{\mathcal{D}_{X,\gamma+\sigma}^{2\alpha}} \nonumber\\
	\lesssim & T^{\alpha-\sigma} \left(1+\varrho_{\alpha}(\mathbf{X}) \right) \|Y,Y^\prime\|_{\mathcal{D}_{X,\gamma}^{2\alpha}} +(1+\varrho_{\alpha}(\mathbf{X})) \left(|Y_0|_{\mathcal H_\gamma^d}+|Y_0^\prime|_{\mathcal H_{\gamma-\alpha}^{d \times d}} \right).
\end{align}
The proof is completed. 	 \hfill $\square$

\begin{lemma} \label{phiest}
Let $\Psi_k(s,r)$ be as given in (\ref{phi}). Then for $0 \leq r \leq s \leq \frac \delta\varepsilon$, one has
\begin{equation*}
\Psi_k(s,r) \leq  C(1+|x|_{\mathcal H_\gamma}^2+|y|_{\mathcal H_\gamma}^2)e^{-(\lambda_1-L_{F_2}-L_{G_2}^2)(s-r)}.
\end{equation*}	
\end{lemma}
\begin{proof}
Recall that
\begin{equation*}
	\Phi_k(s,r)=\mathbb E \left[ \langle F_1(X_{k\delta}^\varepsilon,\hat{Y}_{s\varepsilon+k\delta}^\varepsilon)-\bar{F}_1(X_{k\delta}^\varepsilon),F_1(X_{k\delta}^\varepsilon,\hat{Y}_{r\varepsilon+k\delta}^\varepsilon)-\bar{F}_1(X_{k\delta}^\varepsilon) \rangle_{\mathcal H_{\gamma-\alpha}} \right].
\end{equation*}
For any $s>0$ and $\mathcal F_s$-measurable $\mathcal H_\gamma$-valued random variables $\xi$ and $Y$, we consider the following equation
$$
\left\{ \begin{aligned}
	\mathrm{d} Y_t&=\frac 1\varepsilon \big[L Y_t+F_2(\xi,Y_t)\big]\mathrm{d} t+\frac 1{\sqrt{\varepsilon}}G_2(\xi,Y_t)\mathrm{d} W_t, ~~t\geq s,\\
	Y_s&=Y,
\end{aligned} \right.
$$
which admits a unique solution $\{\widetilde{Y}^{\varepsilon,s,\xi,Y}_t\}_{t\geq0}$ under conditions (\ref{F2Lip}) and ($\mathbf{H3}$).
Looking back at the construction of the auxiliary process (\ref{Auxeq}), for each $k\in\mathbf N_0$ with $k \leq [T/\delta]-1$ and any $t\in[k\delta,(k+1)\delta]$, we have $\mathbb P$-a.s.,
\begin{center}
	$\hat{Y}^\varepsilon_t=\widetilde{Y}^{\varepsilon,k\delta,X^\varepsilon_{k\delta},\hat{Y}^\varepsilon_{k\delta}}_t$.
\end{center}
Therefore, $\Psi_k(s,r)$ can be rewritten as
\begin{align*}
	\Psi_k(s,r)
	=&\mathbb E\Big[ \big\langle F_1(X^\varepsilon_{k\delta},\widetilde{Y}^{\varepsilon,k\delta,X^\varepsilon_{k\delta}, \hat{Y}^\varepsilon_{k\delta}}_{s\varepsilon+k\delta})-\bar{F}_1(X^\varepsilon_{k\delta}),F_1(X^\varepsilon_{k\delta},\widetilde{Y}^{\varepsilon,k\delta,X^\varepsilon_{k\delta},
		\hat{Y}^\varepsilon_{k\delta}}_{r\varepsilon+k\delta})-
	\bar{F}_1(X^\varepsilon_{k\delta}) \big\rangle_{\mathcal H_{\gamma-\alpha}}\Big] \\
	=&\int_{\Omega}\mathbb E\Big[ \big\langle F_1(X^\varepsilon_{k\delta},\widetilde{Y}^{\varepsilon,k\delta,X^\varepsilon_{k\delta},
		\hat{Y}^\varepsilon_{k\delta}}_{s\varepsilon+k\delta})-\bar{F}_1(X^\varepsilon_{k\delta}), \\
	&~~~~~~~~~~F_1(X^\varepsilon_{k\delta},\widetilde{Y}^{\varepsilon,k\delta,X^\varepsilon_{k\delta},
		\hat{Y}^\varepsilon_{k\delta}}_{r\varepsilon+k\delta})-
	\bar{F}_1(X^\varepsilon_{k\delta}) \big \rangle_{\mathcal H_{\gamma-\alpha}}\Big|\mathcal F_{k\delta}\Big](\omega)\mathbb P(\mathrm{d} \omega)\\
	=&\int_{\Omega}\mathbb E\Big[ \big\langle F_1(X^\varepsilon_{k\delta}(\omega),\widetilde{Y}^{\varepsilon,k\delta,X^\varepsilon_{k\delta}(\omega),
		\hat{Y}^\varepsilon_{k\delta}(\omega)}_{s\varepsilon+k\delta})-
	\bar{F}_1(X^\varepsilon_{k\delta}(\omega)),\\
	&~~~~~~~~~~ F_1(X^\varepsilon_{k\delta}(\omega),\widetilde{Y}^{\varepsilon,k\delta,X^\varepsilon_{k\delta}(\omega),
		\hat{Y}^\varepsilon_{k\delta}(\omega)}_{r\varepsilon+k\delta})-
	\bar{F}_1(X^\varepsilon_{k\delta}(\omega)) \big\rangle_{\mathcal H_{\gamma-\alpha}}\Big]\mathbb P(\mathrm{d} \omega).
\end{align*}
Here, going from second to the third equation, we used the fact that $X^\varepsilon_{k\delta}$ and $\hat{Y}^\varepsilon_{k\delta}$ are $\mathcal F_{k\delta}$-measurable, and for any fixed $(x,y)\in \mathcal H_\gamma \times \mathcal H_\gamma$, $\{\widetilde{Y}^{\varepsilon,k\delta,x,y}_{s\varepsilon+k\delta}\}_{s\geq0}$ is independent of $\mathcal F_{k\delta}$.

According to the definition of process $\{\widetilde{Y}^{\varepsilon,k\delta,
	x,y}_t\}_{t\geq0}$, for each $k\in\mathbf N_0$, using a shift transformation, we have $\mathbb P$-a.s.,
\begin{align}\label{Step3equ}
	&\widetilde{Y}^{\varepsilon,k\delta,x,y}_{s\varepsilon+k\delta}\nonumber\\
	=&y+\frac 1\varepsilon \int_{k\delta}^{s\varepsilon+k\delta} L \widetilde{Y}^{\varepsilon,k\delta,x,y}_r \mathrm{d} r+\frac 1\varepsilon\int_{k\delta}^{s\varepsilon+k\delta}F_2(x,\widetilde{Y}^{\varepsilon,k\delta,x,y}_r)\mathrm{d} r+
	\frac 1{\sqrt{\varepsilon}}\int_{k\delta}^{s\varepsilon+k\delta}G_2(x,\widetilde{Y}^{\varepsilon,k\delta,x,y}_r)\mathrm{d}W_r\nonumber\\
	=&y+\frac 1\varepsilon\int_0^{s\varepsilon} L \widetilde{Y}^{\varepsilon,k\delta,x,y}_{r+k\delta} \mathrm{d} r+\frac 1\varepsilon\int_0^{s\varepsilon}F_2(x,\widetilde{Y}^{\varepsilon,k\delta,x,y}_{r+k\delta})\mathrm{d}r+\frac 1{\sqrt{\varepsilon}}\int_0^{s\varepsilon}G_2(x,\widetilde{Y}^{\varepsilon,k\delta,x,y}_{r+k\delta})\mathrm{d}W^{k\delta}_r\nonumber\\
	=&y+\int_0^s L \widetilde{Y}^{\varepsilon,k\delta,x,y}_{r\varepsilon+k\delta}\mathrm{d}r+\int_0^s F_2(x,\widetilde{Y}^{\varepsilon,k\delta,x,y}_{r\varepsilon+k\delta})\mathrm{d}r
	+\int_0^sG_2(x,\widetilde{Y}^{\varepsilon,k\delta,x,y}_{r\varepsilon+k\delta})\mathrm{d} \bar{W}^{k\delta}_r,
\end{align}
where $W^{k\delta}_r:=W_{r+k\delta}-W_{k\delta}$ and $\bar{W}^{k\delta}_r:=\frac 1{\sqrt{\varepsilon}}W^{k\delta}_{r\varepsilon}$.

From the definition of the solution to the frozen equation (\ref{frozen}),
it is easy to see that the uniqueness of the solutions to Eqs.~(\ref{Step3equ}) and (\ref{frozen}) implies that the distribution of $\{\widetilde{Y}^{\varepsilon,k\delta,x,y}_{s\varepsilon+k\delta}\}_{0\leq s\leq \frac \delta\varepsilon}$ coincides with the distribution of $\{Y^{x,y}_s\}_{0\leq s\leq \frac \delta\varepsilon}$. So we have
\begin{align*}
	\Psi_k(s,r)=&\int_\Omega \widetilde{\mathbb E} \Big[ \big\langle F_1(X^\varepsilon_{k\delta}(\omega),Y^{X^\varepsilon_{k\delta}(\omega),
		\hat{Y}^\varepsilon_{k\delta}(\omega)}_s)-\bar{F}_1(X^\varepsilon_{k\delta}(\omega)),\\
	&\qquad\quad F_1(X^\varepsilon_{k\delta}(\omega),Y^{X^\varepsilon_{k\delta}(\omega),
		\hat{Y}^\varepsilon_{k\delta}(\omega)}_r)-\bar{F}_1(X^\varepsilon_{k\delta}(\omega)) \big\rangle_{\mathcal H_{\gamma-\alpha}}\Big] \mathbb P(\mathrm{d} \omega)\\
	=&\int_\Omega\int_{\widetilde{\Omega}}\Big\langle \widetilde{\mathbb E} \Big[F_1(X^\varepsilon_{k\delta}(\omega),Y^{X^\varepsilon_{k\delta}(\omega),
		Y^{X^\varepsilon_{k\delta}(\omega),\hat{Y}^\varepsilon_{k\delta}(\omega)}_r(\omega')}_{s-r}-
	\bar{F}_1(X^\varepsilon_{k\delta}(\omega))\Big],\\
	&\qquad\qquad F_1(X^\varepsilon_{k\delta}(\omega),Y^{X^\varepsilon_{k\delta}(\omega),
		\hat{Y}^\varepsilon_{k\delta}(\omega)}_r(\omega'))-\bar{F}_1(X^\varepsilon_{k\delta}(\omega))\Big\rangle_{\mathcal H_{\gamma-\alpha}}
	\widetilde{\mathbb P}(\mathrm{d} \omega')\mathbb P(\mathrm{d} \omega),	
\end{align*}
where in the second equality, we used the Markov property of process $\{Y^{x,y}_t\}_{t\geq 0}$.
Using Lemma \ref{ergo} and estimate (\ref{frozenest}) we have
\begin{align*}
	\Psi_k(s,r)
	&\leq C\int_\Omega \int_{\widetilde{\Omega}}\Big\{\Big[1+|X^\varepsilon_{k\delta}(\omega)|_{\mathcal H_\gamma }+|Y^{X^\varepsilon_{k\delta}(\omega),
		\hat{Y}^\varepsilon_{k\delta}(\omega)}_r(\omega')|_{\mathcal H_\gamma }\Big]  e^{-(s-r)(\lambda_1-L_{F_2}-L_{G_2}^2)}  \nonumber\\
	&~~~\times\Big[1+|X^\varepsilon_{k\delta}(\omega)|_{\mathcal H_\gamma }+|Y^{X^\varepsilon_{k\delta}(\omega),
		\hat{Y}^\varepsilon_{k\delta}(\omega)}_r(\omega')|_{\mathcal H_\gamma }\Big]\Big\}\widetilde{\mathbb
		P}(\mathrm{d} \omega')\mathbb P(\mathrm{d} \omega)  \nonumber\\
	&\leq C\int_\Omega\Big[1+|X^\varepsilon_{k\delta}(\omega)|_{\mathcal H_\gamma }^2+|\hat{Y}^\varepsilon_{k\delta}(\omega)|_{\mathcal H_\gamma }^2\Big]\mathbb P(\mathrm{d} \omega) e^{-(s-r)(\lambda_1-L_{F_2}-L_{G_2}^2)}   \nonumber\\
	&\leq C\int_\Omega\Big[1+\|X^\varepsilon(\omega)\|_{\infty,\mathcal H_\gamma}^2+|\hat{Y}^\varepsilon_{k\delta}(\omega)|_{\mathcal H_\gamma }^2\Big]\mathbb P(\mathrm{d} \omega) e^{-(s-r)(\lambda_1-L_{F_2}-L_{G_2}^2)}   \nonumber\\
	&\leq C(1+|x|_{\mathcal H_\gamma}^2+| y |_{\mathcal H_\gamma }^2) e^{-(s-r)(\lambda_1-L_{F_2}-L_{G_2}^2)},
\end{align*}
where in the last inequality, we used (\ref{X2}) and (\ref{Auxest}).
\end{proof}

\section*{Acknowledgements}
Y. Xu was partially supported by Key International (Regional) Joint Research Program of NSFC under Grant No.12120101002.
M. Li and B. Pei were partially supported by National Natural Science Foundation of China (NSFC) under Grant No.12172285, Guangdong Basic and Applied Basic Research Foundation under Grant No. 2024A1515012164 and Fundamental
Research Funds for the Central Universities.

\end{document}